\documentclass[12pt,reqno]{smfart}
\usepackage{amsmath,amssymb,smfthm,stmaryrd,epic,enumerate}
\usepackage[frenchb]{babel}
\usepackage[latin1]{inputenc}
\usepackage{a4wide,float}
\usepackage{pdfpages}

\usepackage{pstricks-add}
\usepackage{xypic}
\usepackage[active]{srcltx}

\usepackage{lmodern}
\usepackage{mathrsfs}

\usepackage{hyperref}

\newcommand*{\htarrow}{\lhook\joinrel\relbar\joinrel\twoheadrightarrow}
\newcommand*{\tarrow}{\relbar\joinrel\mid\joinrel\twoheadrightarrow}
\newcommand*{\harrow}{\lhook\joinrel\relbar\joinrel\mid\joinrel\rightarrow}
\newcommand*{\rarrow}{\relbar\joinrel\mid\joinrel\rightarrow}

\NumberTheoremsIn{subsection}\SwapTheoremNumbers

\begin{document}
\newtheorem{prop-defi}[smfthm]{Proposition-DÈfinition}
\newtheorem{notas}[smfthm]{Notations}
\newtheorem{nota}[smfthm]{Notation}
\newtheorem{defis}[smfthm]{DÈfinitions}
\newtheorem{hypo}[smfthm]{HypothËse}

\def\Xm{{\mathbb X}}
\def\Um{{\mathbb U}}
\def\Am{{\mathbb A}}
\def\Fm{{\mathbb F}}
\def\Mm{{\mathbb M}}
\def\Nm{{\mathbb N}}
\def\Pm{{\mathbb P}}
\def\Qm{{\mathbb Q}}
\def\Zm{{\mathbb Z}}
\def\Dm{{\mathbb D}}
\def\Cm{{\mathbb C}}
\def\Rm{{\mathbb R}}
\def\Gm{{\mathbb G}}
\def\Lm{{\mathbb L}}
\def\Km{{\mathbb K}}
\def\Om{{\mathbb O}}
\def\Em{{\mathbb E}}

\def\RC{{\mathcal R}}
\def\BC{{\mathcal B}}
\def\QC{{\mathcal Q}}
\def\TC{{\mathcal T}}
\def\ZC{{\mathcal Z}}
\def\AC{{\mathcal A}}
\def\CC{{\mathcal C}}
\def\DC{{\mathcal D}}
\def\EC{{\mathcal E}}
\def\FC{{\mathcal F}}
\def\GC{{\mathcal G}}
\def\HC{{\mathcal H}}
\def\IC{{\mathcal I}}
\def\JC{{\mathcal J}}
\def\KC{{\mathcal K}}
\def\LC{{\mathcal L}}
\def\MC{{\mathcal M}}
\def\NC{{\mathcal N}}
\def\OC{{\mathcal O}}
\def\PC{{\mathcal P}}
\def\UC{{\mathcal U}}
\def\VC{{\mathcal V}}
\def\XC{{\mathcal X}}
\def\SC{{\mathcal S}}

\def\BF{{\mathfrak B}}
\def\AF{{\mathfrak A}}
\def\GF{{\mathfrak G}}
\def\EF{{\mathfrak E}}
\def\CF{{\mathfrak C}}
\def\DF{{\mathfrak D}}
\def\JF{{\mathfrak J}}
\def\LF{{\mathbb L}}
\def\LFF{{\mathscr L}}
\def\MF{{\mathfrak M}}
\def\NF{{\mathfrak N}}
\def\XF{{\mathfrak X}}
\def\UF{{\mathfrak U}}
\def\KF{{\mathfrak K}}
\def\FF{{\mathfrak F}}

\def \hi{\HC}

\def \longmapright#1{\smash{\mathop{\longrightarrow}\limits^{#1}}}
\def \mapright#1{\smash{\mathop{\rightarrow}\limits^{#1}}}
\def \lexp#1#2{\kern \scriptspace \vphantom{#2}^{#1}\kern-\scriptspace#2}
\def \linf#1#2{\kern \scriptspace \vphantom{#2}_{#1}\kern-\scriptspace#2}
\def \linexp#1#2#3 {\kern \scriptspace{#3}_{#1}^{#2} \kern-\scriptspace #3}

\def \Ext{\mathop{\mathrm{Ext}}\nolimits}
\def \ad{\mathop{\mathrm{ad}}\nolimits}
\def \sh{\mathop{\mathrm{Sh}}\nolimits}
\def \irr{\mathop{\mathrm{Irr}}\nolimits}
\def \FH{\mathop{\mathrm{FH}}\nolimits}
\def \FPH{\mathop{\mathrm{FPH}}\nolimits}
\def \coh{\mathop{\mathrm{Coh}}\nolimits}
\def \res{\mathop{\mathrm{res}}\nolimits}
\def \op{\mathop{\mathrm{op}}\nolimits}
\def \rec {\mathop{\mathrm{rec}}\nolimits}
\def \art{\mathop{\mathrm{Art}}\nolimits}
\def \hyp {\mathop{\mathrm{Hyp}}\nolimits}
\def \cusp {\mathop{\mathrm{Cusp}}\nolimits}
\def \scusp {\mathop{\mathrm{Scusp}}\nolimits}
\def \Iw {\mathop{\mathrm{Iw}}\nolimits}
\def \JL {\mathop{\mathrm{JL}}\nolimits}
\def \speh {\mathop{\mathrm{Speh}}\nolimits}
\def \isom {\mathop{\mathrm{Isom}}\nolimits}
\def \Vect {\mathop{\mathrm{Vect}}\nolimits}
\def \groth {\mathop{\mathrm{Groth}}\nolimits}
\def \hom {\mathop{\mathrm{Hom}}\nolimits}
\def \deg {\mathop{\mathrm{deg}}\nolimits}
\def \val {\mathop{\mathrm{val}}\nolimits}
\def \det {\mathop{\mathrm{det}}\nolimits}
\def \rep {\mathop{\mathrm{Rep}}\nolimits}
\def \spec {\mathop{\mathrm{Spec}}\nolimits}
\def \fr {\mathop{\mathrm{Fr}}\nolimits}
\def \frob {\mathop{\mathrm{Frob}}\nolimits}
\def \ker {\mathop{\mathrm{Ker}}\nolimits}
\def \im {\mathop{\mathrm{Im}}\nolimits}
\def \Red {\mathop{\mathrm{Red}}\nolimits}
\def \red {\mathop{\mathrm{red}}\nolimits}
\def \aut {\mathop{\mathrm{Aut}}\nolimits}
\def \diag {\mathop{\mathrm{diag}}\nolimits}
\def \spf {\mathop{\mathrm{Spf}}\nolimits}
\def \Def {\mathop{\mathrm{Def}}\nolimits}
\def \twist {\mathop{\mathrm{Twist}}\nolimits}
\def \supp {\mathop{\mathrm{Supp}}\nolimits}
\def \Id {{\mathop{\mathrm{Id}}\nolimits}}
\def \lie {{\mathop{\mathrm{Lie}}\nolimits}}
\def \Ind{\mathop{\mathrm{Ind}}\nolimits}
\def \ind {\mathop{\mathrm{ind}}\nolimits}
\def \loc {\mathop{\mathrm{Loc}}\nolimits}
\def \top {\mathop{\mathrm{Top}}\nolimits}
\def \ker {\mathop{\mathrm{Ker}}\nolimits}
\def \coker {\mathop{\mathrm{Coker}}\nolimits}
\def \gal {{\mathop{\mathrm{Gal}}\nolimits}}
\def \Nr {{\mathop{\mathrm{Nr}}\nolimits}}
\def \rn {{\mathop{\mathrm{rn}}\nolimits}}
\def \tr {{\mathop{\mathrm{Tr~}}\nolimits}}
\def \Sp {{\mathop{\mathrm{Sp}}\nolimits}}
\def \st {{\mathop{\mathrm{St}}\nolimits}}
\def \sp{{\mathop{\mathrm{Sp}}\nolimits}}
\def \perv{\mathop{\mathrm{Perv}}\nolimits}
\def \tor {{\mathop{\mathrm{Tor}}\nolimits}}
\def \nrd {{\mathop{\mathrm{Nrd}}\nolimits}}
\def \nilp {{\mathop{\mathrm{Nilp}}\nolimits}}
\def \obj {{\mathop{\mathrm{Obj}}\nolimits}}
\def \cl {{\mathop{\mathrm{cl}}\nolimits}}
\def \gr {{\mathop{\mathrm{gr}}\nolimits}}
\def \grr {{\mathop{\mathrm{grr}}\nolimits}}
\def \coim {{\mathop{\mathrm{Coim}}\nolimits}}
\def \can {{\mathop{\mathrm{can}}\nolimits}}
\def \car {{\mathop{\mathrm{N(1)}}\nolimits}}
\def \unip {{\mathop{\mathrm{unip}}\nolimits}}
\def \alg {{\mathop{\mathrm{Alg}}\nolimits}}

\def \rem{{\noindent\textit{Remarque:~}}}
\def \ext {{\mathop{\mathrm{Ext}}\nolimits}}
\def \End {{\mathop{\mathrm{End}}\nolimits}}

\def\semi{\mathrel{>\!\!\!\triangleleft}}
\let \DS=\displaystyle

%\numberwithin{equation}{section}

%\input{macros2.tex}

\setcounter{secnumdepth}{3} \setcounter{tocdepth}{3}

\def \Fil{\mathop{\mathrm{Fil}}\nolimits}
\def \CoFil{\mathop{\mathrm{CoFil}}\nolimits}
\def \Fill{\mathop{\mathrm{Fill}}\nolimits}
\def \CoFill{\mathop{\mathrm{CoFill}}\nolimits}
\def\SF{{\mathfrak S}}
\def\PF{{\mathfrak P}}
\def \EFil{\mathop{\mathrm{EFil}}\nolimits}
\def \ECoFil{\mathop{\mathrm{ECoFil}}\nolimits}
\def \EFill{\mathop{\mathrm{EFill}}\nolimits}
\def \FP{\mathop{\mathrm{FP}}\nolimits}

\let \longto=\longrightarrow
\let \oo=\infty

\let \d=\delta
\let \k=\kappa

%\newcounter{num}
%\renewcommand{\thenum}{\arabic{num}}
\newcommand{\marque}{\addtocounter{smfthm}{1}
{\smallskip \noindent \textit{\thesmfthm}~---~}}

\renewcommand\atop[2]{\ensuremath{\genfrac..{0pt}{1}{#1}{#2}}}

\title[La cohomologie des espaces de Lubin-Tate est libre]
{La cohomologie des espaces de Lubin-Tate est libre}

\alttitle{The cohomology of Lubin-Tate spaces is free}

\author{Pascal Boyer}
\address{UniversitÈ Paris 13, Sorbonne Paris CitÈ \\
LAGA, CNRS, UMR 7539\\ 
F-93430, Villetaneuse (France) \\
PerCoLaTor: ANR-14-CE25}
\email{boyer@math.univ-paris13.fr}
\thanks{L'auteur remercie l'ANR pour son soutien dans le cadre du projet PerCoLaTor 14-CE25.}

\frontmatter

\begin{abstract}
Le rÈsultat principal de ce travail est l'absence de torsion dans la $ \overline \Zm_{l}$-cohomologie 
de la tour de Lubin-Tate. Comme dans \cite{boyer-invent2}, la stratÈgie est globale et repose sur l'Ètude 
de la filtration de stratification du faisceau pervers des cycles Èvanescents de certaines variÈtÈs de 
Shimura de type Kottwitz-Harris-Taylor, dont les graduÈs se dÈcrivent comme une extension intermÈdiaire
de systËmes locaux construits dans \cite{h-t}. Le point crucial consiste ‡
dÈcrire la diffÈrence entre ces extensions pour les deux $t$-structures usuelles $p$ et $p+$. 
Pour ce faire, on utilise la thÈorie des dÈrivÈes pour les reprÈsentations du 
groupe mirabolique d'aprËs \cite{zelevinski1}.
\end{abstract}

\begin{altabstract}
The principal result of this work is the freeness in the $ \overline \Zm_l$-cohomology of the Lubin-Tate 
tower. As in \cite{boyer-invent2}, the strategy is of global nature and relies on studying the filtration of
stratification of the perverse sheaf of vanishing cycles of some Shimura varieties of Kottwitz-Harris-Taylor 
types, whose graduates can be explicited as some intermediate extension of some local system
constructed in \cite{h-t}. The crucial point relies on the study of the difference between such extension
for the two classical $t$-structures $p$ and $p+$. The main ingredients use the theory of derivative
for representations of the mirabolic group in \cite{zelevinski1}.
\end{altabstract}



\altkeywords{Shimura varieties, formal modules, Langlands correspondences, Jacquet-Langlands
correspondences, monodromy filtration, perverse sheaves, vanishing cycles, stratification, quasi-abelian category, torsion theory}

\maketitle

\pagestyle{headings} \pagenumbering{arabic}

\section*{Introduction}
\renewcommand{\theequation}{\arabic{equation}}
\backmatter

Dans \cite{boyer-invent2}, nous
avons explicitÈ les groupes de cohomologie ‡ coefficients dans $\overline \Qm_l$ 
de la tour de Lubin-Tate en Ètudiant le faisceau pervers $\Psi_{\IC,\overline \Qm_l}$ 
des cycles Èvanescents d'une tour de variÈtÈs de Shimura de Kottwitz-Harris-Taylor $X_{\IC}$, 
cf. \cite{h-t},
en une place $v$ de son corps reflex $F$. Le passage
du global vers le local est fourni par un analogue du thÈorËme de Serre-Tate couplÈ
au thÈorËme de comparaison de Berkovich. Le but de ce travail est de donner une 
$\overline \Zm_l$-version de ces rÈsultats en Èlucidant \emph{du cÙtÈ global} la structure 
du faisceau pervers $\Psi_{\IC}$ sur la fibre spÈciale gÈomÈtrique $X_{\IC,\bar s}$ en $v$
et en prouvant, \emph{du cÙtÈ local} cf. le thÈorËme \ref{theo-LT}, 
l'absence de torsion dans la cohomologie
des espaces de Lubin-Tate.

Sur $\overline \Qm_l$, pour Ètudier $\Psi_{\IC,\overline \Qm_l}$ , nous l'avons dÈvissÈ en constituants 
simples $\PC_{\overline \Qm_l}(\pi_v,t)$ dÈcrits comme
extensions intermÈdiaires de systËmes locaux $HT_{\overline \Qm_l}(\pi_v,\st_t(\pi_v))$ 
construits dans \cite{h-t}. L'action de l'opÈrateur de monodromie est en particulier explicite et 
la filtration de $\Psi_{\IC,\overline \Qm_l}$ par les noyaux itÈrÈs de la monodromie donne alors 
une suite spectrale 
$$E_1^{i,j}=\hi^{i+j} \gr^K_{-j}(\Psi_{\IC,\overline \Qm_l}) \Rightarrow \hi^{i+j} \Psi_{\IC,\overline \Qm_l}$$
calculant les faisceaux de cohomologie de $\Psi_{\IC,\overline \Qm_l}$ ‡ partir de ceux de ses graduÈs.

\vspace{.2cm}

\noindent \textbf{Fait 1} \rule[.1cm]{.6cm}{.2pt}
\textit{(cf. \cite{boyer-invent2} \S 5.8) La suite spectrale $E_1^{i,j}$ dÈgÈnËre en $E_1$, i.e.
pour tout $n$, le faisceau de cohomologie $\hi^n \Psi_{\IC,\overline \Qm_l}$ 
admet une filtration dont les graduÈs sont les $\hi^n \gr_k^K(\Psi_{\IC,\overline \Qm_l})$.}

\vspace{.2cm}

Cette observation nous fournit \textit{une stratÈgie}
pour montrer que les $\hi^n \Psi_{\IC,\overline \Zm_l}$ sont sans torsion puisqu'il suffit de construire
une version entiËre, i.e. ‡ coefficients dans $\overline \Zm_l$, de la filtration de 
$\Psi_{\IC,\overline \Qm_l}$  par les noyaux itÈrÈs de la monodromie, 
puis de montrer que les $\hi^n \gr^K_k(\Psi_{\IC,\overline \Zm_l})$ sont sans torsion.

Une $\overline \Zm_l$-filtration de $\Psi_{\IC,\overline \Zm_l}$ par les noyaux itÈrÈs de la monodromie 
est donnÈe dans \cite{boyer-torsion} en utilisant des morphismes
d'adjonction $j^{=h}_!j^{=h,*} \rightarrow \Id$ associÈs ‡ la stratification de Newton 
$j^{=h}:X^{=h}_{\IC,\bar s} \hookrightarrow X_{\IC,\bar s}$ de la fibre spÈciale gÈomÈtrique
$X_{\IC,\bar s}$ de $X_\IC$ ‡ la place $v$, cf. le \S \ref{para-rappel-filtration}. 
Cette stratÈgie soulËve essentiellement deux problËmes:
\begin{itemize}
\item le premier issu de la nature mÍme de la construction des filtrations de stratification, 
qui requiËre de saturer, cf. la dÈfinition \ref{defi-saturee},
les morphismes d'adjonction de sorte que l'on ne contrÙle plus les rÈseaux des systËmes
locaux qui apparaissent et plus sÈrieusement encore, les faisceaux pervers 
d'Harris-Taylor $\PC(\pi_v,t)$ 
qui interviennent sont seulement connus ‡ bimorphisme prËs, i.e. avec la notation \ref{nota-htarrow1}
$$\lexp p j^{=tg}_{!*} HT(\pi_v,\st_t(\pi_v)) \htarrow_+ \PC(\pi_v,t) \htarrow_+ \lexp {p+} j^{=tg}_{!*}
HT(\pi_v,\st_t(\pi_v)).$$
\rem Nous verrons, cf. la proposition \ref{prop-dec-pervers},
qu'en gÈnÈral ces deux extensions intermÈdiaires $p$ et $p+$, ÈchangÈes par
la dualitÈ de Grothendieck-Verdier, sont distinctes et qu'identifier $\PC(\pi_v,t)$ dÈpend
fortement de comment on filtre $\Psi_\IC$: rappelons que $\Psi_\IC$ est autodual, de sorte que si une 
filtration fournit pour graduÈs les $p$-extensions intermÈdiaires, une filtration \og duale \fg{} fournira
les $p+$ versions.

\item Le deuxiËme concerne le fait que pour calculer les faisceaux
de cohomologie des $\gr^K_k(\Psi_{\IC,\overline \Zm_l})$, on doit utiliser une suite spectrale
qui dÈgÈnËre en $E_2$ mais pas en $E_1$, i.e. certaines des flËches $d_1^{p,q}$ sont non nulles
sur $\overline \Qm_l$ de sorte que lorsqu'on va les considÈrer sur $\overline \Zm_l$, on devrait
non seulement contrÙler les rÈseaux, ce qui semble difficile ‡ cause de la saturation ÈvoquÈe
plus haut, mais aussi les dÈterminer ‡ homothÈtie prËs, ce qui semble totalement hors de portÈe
de part la nature mÍme de la construction des suites spectrales.
\end{itemize}

En ce qui concerne ce dernier point, il existe une explication gÈomÈtrique qui repose sur le fait
que les strates de Newton sont \og gÈomÈtriquement \fg{} induites paraboliques, cf. la proposition
\ref{prop-strate-induite}
$$X^{=h}_{\IC,\bar s}=X^{=h}_{\IC,\bar s,\overline{1_h}} \times_{P_{h,d-h}(F_v)} GL_d(F_v),$$ 
de sorte que lorsqu'on considËre toutes les
strates, on obtient naturellement, lorsqu'on calcule des faisceaux de cohomologie,
des reprÈsentations induites paraboliques lesquelles en gÈnÈral ne sont pas irrÈductibles.
Or les faisceaux de cohomologie de $\Psi_\IC$ ne font pas apparaitre d'induites paraboliques, ce qui
explique que dans \cite{boyer-invent2}, apparaissent des suites spectrales non triviales, i.e. qui
ne dÈgÈnËrent qu'en $E_2$ et non en $E_1$. 

Or il est bien connu que la thÈorie des reprÈsentations de $GL_d(F_v)$ repose, via la thÈorie des
dÈrivÈes, sur celle beaucoup
plus simple du groupe mirabolique $M_d(F_v)$ obtenu comme  le sous-groupe du parabolique 
standard $P_{1,d-1}(F_v)$ dont\footnote{Dans \cite{zelevinski1} le sous-groupe mirabolique est dÈfini
en imposant la derniËre colonne Ègale ‡ $(0,\cdots,0,1)$. On passe alors de cette dÈfinition ‡ la notre
via $g \mapsto \sigma (\lexp t g^{-1}) \sigma^{-1}$ o˘ $\sigma$ est la matrice de permutation
associÈe au cycle $(1~2~\cdots ~d)$.}
le premier coefficient en haut ‡ gauche est Ègal ‡ $1$. Ainsi
pour $\pi$ et $\pi'$ des reprÈsentations de respectivement $GL_h(F_v)$ et $GL_{d-h}(F_v)$,
la restriction au groupe mirabolique de l'induite parabolique $\pi \times \pi'$ se dÈcrit via la
suite exacte courte
$$0 \rightarrow \pi \times \pi'_{|M_{d-h}(F_v)} \longrightarrow (\pi \times \pi')_{|M_d(F_v)} \longrightarrow
\pi_{|M_h(F_v)} \times \pi' \rightarrow 0$$
o˘ le premier terme est l'induite parabolique relativement ‡ 
$\left ( \begin{array}{ccc} 1 & 0 & V_{d-h-1} \\ 0 & GL_{h} & U \\ 0 & 0 & GL_{d-h-1} \end{array} \right )$,
et le dernier est l'induite ‡ support compact relativement ‡
$\left ( \begin{array}{cc} M_{h} & U \\ 0 & GL_{d-h} \end{array} \right )$.

\vspace{.2cm}

\noindent \textbf{Fait 2} \rule[.1cm]{.6cm}{.2pt}
\textit{Il existe une version gÈomÈtrique de la suite exacte courte prÈcÈdente, cf. \cite{boyer-FT}, qui
permet de calculer directement les faisceaux de cohomologie de $\Psi_{\IC}$, i.e. sans avoir recours
‡ une suite spectrale, cf. aussi l'appendice B de ce texte.
}

\vspace{.2cm}

L'idÈe consiste ‡ considÈrer pour une strate $X^{=1}_{\IC,\bar s,c}$, 
$$\bar j_{\neq c}:\overline X_{\IC} - X^{\geq 1}_{\IC,\bar s,c} \hookrightarrow \overline X_{\IC}
:=X_{\IC} \times_{\spec \OC_v} \spec \overline \OC_v,$$
et, plus gÈnÈralement, pour tout $h \geq 1$ et toute strate \og pure \fg{} $X^{=h+1}_{\IC,\bar s,c}$
$$j^{=h}_{\neq c}:X^{\geq h}_{\IC,\bar s,\overline{1_h}} \setminus
X^{\geq h+1}_{\IC,\bar s,c} \hookrightarrow X^{\geq h}_{\IC,\bar s,\overline{1_h}}
\hookrightarrow X_{\IC,\bar s}.$$
Le point essentiel est que $\bar j_{\neq c}$ et $j^{=h}_{\neq c}$ sont affines, cf. le lemme \ref{lem-HT1}
et la suite exacte courte (\ref{eq-sec-psi}). 

La thÈorie classique des dÈrivÈes pour les reprÈsentations
du groupe mirabolique, est utilisÈe pour montrer que des extensions scindÈes entre 
$\overline \Qm_l$-faisceaux pervers d'Harris-Taylor, restent encore scindÈes sur $\overline \Zm_l$
pourvu qu'on filtre $\Psi_\IC$ dans \og le bon ordre \fg.
D'un point de vue global, on montre alors que la filtration de stratification \emph{exhaustive} de 
$\Psi_{\IC}$ dÈfinie dans \cite{boyer-torsion} ‡ l'aide des seuls morphismes d'adjonction
$j_!j^* \rightarrow \Id$ (resp. $\Id \rightarrow j_*j^*$) admet pour graduÈs les extensions
intermÈdiaires pour la $t$-structure $p$ (resp. $p+$) des systËmes locaux d'Harris-Taylor.
L'ÈnoncÈ principal qui permet de montrer cette propriÈtÈ repose sur la description du conoyau du 
bimorphisme naturel entre la $p$ et $p+$
extension intermÈdiaire d'un systËme local d'Harris-Taylor. Rappelons qu'un tel systËme local
sur la strate de Newton $X^{=h}_{\IC,\bar s}$ est associÈ ‡ une cuspidale $\pi_v$ 
de $GL_g(F_v)$ avec $g|h=tg$, et on montre, cf. la proposition \ref{prop-ext-pp}, que ce conoyau
est nul si, et seulement si, la rÈduction modulo $l$ de $\pi_v$ reste
supercuspidale. Sinon on dÈcrit, cf. la proposition \ref{prop-dec-pervers},
la $l$-torsion de ce conoyau ‡ l'aide de la rÈduction
modulo $l$ de la reprÈsentation de Steinberg $\st_t(\pi_v)$ donnÈe dans \cite{boyer-repmodl}.

En ce qui concerne l'organisation du papier,
les notations sur les reprÈsentations sont donnÈes dans l'appendice A. Les rÈsultats des lemmes
et propositions de cet article concernant les $\overline \Zm_l$-faisceaux pervers libres sont, 
d'aprËs \cite{boyer-invent2} et \cite{boyer-torsion}, dÈj‡ connus sur $\overline \Qm_l$: ceux dont
nous avons besoin, sont rappelÈs en appendice B avec quelques complÈments autour de 
l'effet des morphismes $j^{=h}_{\neq c,!} j^{=h,*}_{\neq c}$, ingrÈdient clef de cet article.
%et qui devrait permettre de simplifier grandement les arguments de \cite{boyer-invent2}, cf. les remarques suivant \ref{lem-Ql1} et \ref{prop-QL-Psi}.

Mentionnons enfin la thËse de H. Wang qui, par voie purement locale en se ramenant
aux travaux de BonnafÈ et Rouquier \cite{B-R} sur la cohomologie des variÈtÈs de 
Deligne-Lusztig,
montre la libertÈ de l'Ètage modÈrÈ de la cohomologie de la tour de Drinfeld. En utilisant
le thÈorËme de Faltings-Fargues, son rÈsultat correspond au cas des reprÈsentations
cuspidales de niveau zÈro.

\tableofcontents

\mainmatter

\renewcommand{\theequation}{\arabic{section}.\arabic{subsection}.\arabic{smfthm}}

\section{Rappels gÈomÈtriques}

Dans tout ce texte, les lettres $p \neq l$ dÈsigneront deux nombres premiers distincts, 
$d$ un entier strictement positif et $\Lambda$ au choix 
une extension algÈbrique de $\Qm_l$ comme par exemple $\overline \Qm_l$,
l'anneau des entiers d'une telle extension, comme par exemple $\Zm_l^{nr}$ ou
une extension algÈbrique de $\Fm_l$, comme par exemple $\overline \Fm_l$.

\subsection{Espaces de Lubin-Tate}
\label{para-LT}

On dÈsignera par $K$ une extension finie de $\Qm_p$, $\OC_K$ son anneau 
des entiers, d'idÈal maximal $\PC_K$, $\varpi_K$ une uniformisante  et
$\kappa=\OC_K/\PC_K$ son corps rÈsiduel de cardinal $q=p^f$. L'extension maximale non 
ramifiÈe de $K$ sera notÈe $K^{nr}$ de complÈtÈ $\hat K^{nr}$, d'anneau des entiers respectif 
$\OC_{K^{nr}}$ et $\OC_{\hat K^{nr}}$.
Soit $\Sigma_{K,d}$ le $\OC_K$-module de Barsotti-Tate formel sur
$\overline \k$ de hauteur $d$, cf. \cite{h-t} \S II. On
considËre la catÈgorie $C$ des $\OC_K$-algËbres locales artiniennes
de corps rÈsiduel $\overline \k$.

\begin{defi}
Le foncteur qui ‡ un objet $R$ de $C$
associe l'ensemble des classes d'isomorphismes des dÈformations par quasi-isogÈnies sur $R$ de $\Sigma_{K,d}$
munies d'une structure de niveau $n$ est pro-reprÈsentable par un schÈma formel
$\widehat{\MC_{LT,d,n}}=\coprod_{h \in \Zm} \widehat{\MC}_{LT,d,n}^{(h)}$ o˘
$\widehat{\MC}_{LT,d,n}^{(h)}$ reprÈsente le sous-foncteur pour des dÈformations
par des quasi-isogÈnies de hauteur $h$.
\end{defi}

\rem chacun des $\widehat{\MC}_{LT,d,n}^{(h)}$ est non-canoniquement isomorphe au schÈma formel $\widehat{\MC}_{LT,d,n}^{(0)}$
notÈ $\spf \Def_{d,n}$ dans \cite{boyer-invent2}.
On notera sans chapeau les fibres gÈnÈriques de Berkovich de ces espaces; ce sont donc des $\widehat{K^{nr}}$-espaces analytiques au sens de \cite{berk0} et on note
$\MC_{LT,n}^{d/K}:=\MC_{LT,d,n} \hat \otimes_{\hat K^{nr}} \hat{\overline K}.$

Le groupe des quasi-isogÈnies de $\Sigma_{K,d}$  s'identifie au groupe $D_{K,d}^\times$ des unitÈs de l'algËbre ‡ division 
centrale sur $K$ d'invariant $1/d$, lequel, par dÈfinition, agit sur $\MC_{LT,n}^{d/K}$.
Pour tout $n \geq 1$, on a une action naturelle de $GL_d(\OC_K/\PC_K^n)$ sur les structures de niveau et donc sur
$\MC_{LT,n}^{d/K}$; cette action se prolonge en une action de $GL_d(K)$ sur la limite projective
${\displaystyle \lim_{\leftarrow ~n}} \MC_{LT,n}^{d/K}$. Sur cette limite projective on dispose ainsi d'une action de 
$GL_d(K) \times D_{K,d}^\times$ qui se factorise par 
$\Bigl ( GL_d(K) \times D_{K,d}^\times \Bigr ) / K^\times$ o˘ $K^\times$ est plongÈ diagonalement.

\begin{defi} 
Soit
$\Psi_{K,\Lambda,d,n}^{i}\simeq H^i(\MC_{LT,d,n}^{(0)} \hat \otimes_{\hat K^{nr}} \hat{\overline K},
\Lambda)$ le $\Lambda$-module de type fini
associÈ, via la thÈorie des cycles Èvanescents de Berkovich, au morphisme structural 
$\widehat{\MC_{LT,d,n}^{(0)}} \longto \spf \hat \OC_K^{nr}.$ 
\end{defi}

On notera aussi $\UC_{K,\Lambda,d,n}^i := H^i(\MC_{LT,n}^{d/K},\Lambda)$ et on pose 
$\UC_{K,\Lambda,d}^{i}= {\DS \lim_{\atop{\longto}{n}}} ~ \UC_{K,\Lambda,d,n}^{i}$
de sorte que $\KF_{n}:=\ker (GL_d(\OC_K) \longto GL_d(\OC_K/\PC_K^n))$ Ètant pro-$p$ 
pour tout $n \geq 1$, on a
$\UC_{K,\Lambda,d,n}^{i}=(\UC_{K,\Lambda,d}^{i})^{\KF_{n}}$.

\begin{nota} \label{nota-Utau}
ConsidÈrant $\UC_{K,\overline \Zm_l,d,n}^i$ comme une reprÈsentation de $D_{K,d}^\times$, 
d'aprËs \ref{prop-scindage}, pour $\bar \tau \in \RC_{ \overline \Fm_l}(d)$, on note
$\UC^i_{\bar \tau,n}:=\UC_{K,\overline \Zm_l,d,n,\bar \tau}^i,$ puis
$\UC^i_{\bar \tau,\Nm}=\lim_{\atop{\rightarrow}{n}} ~\UC^i_{\bar \tau,n}.$
On notera aussi $\UC^i_{\bar \tau,\Nm,free}$ le quotient libre de $\UC^i_{\bar \tau,\Nm}$.
\end{nota}

%\rem concernant le thÈorËme principal ci-aprËs, on peut dÈcider de raisonner par rÈcurrence sur $h$ auquel cas pour tout $h<d$, on a $\UC^i_{\bar \tau,\Nm,free}=\UC^i_{\bar \tau,\Nm}$.

La description de la $\overline \Qm_l$-cohomologie de ces espaces de Lubin-Tate 
est donnÈe dans \cite{boyer-invent2} thÈorËme 2.3.5. Le rÈsultat principal que nous avons en 
vue est le suivant.

\begin{theo} \label{theo-LT} \phantomsection
Pour tout $h \geq 1$ et pour tout $0 \leq i \leq h-1$, le $\overline \Zm_l$-module
$\UC_{K,\overline \Zm_l,h}^i$ est sans torsion.
\end{theo}

\subsection{GÈomÈtrie de quelques variÈtÈs de Shimura unitaires simples}
\label{para-shimura}

Soient $E/\Qm$ une extension quadratique imaginaire, $F^+/\Qm$ une extension
totalement rÈelle dont on fixe un plongement rÈel $\tau:F^+ \hookrightarrow \Rm$;
on pose $F=F^+ E$ le corps CM associÈ. 

\begin{nota} Pour toute place finie $w$ de $F$, on note $F_w$ le complÈtÈ de $F$ en cette place,
$\OC_w$ son anneau des entiers d'idÈal maximal $\PC_w$ et de corps rÈsiduel $\kappa(w)$.
On notera aussi $\overline \OC_w$ l'anneau des entiers d'une clÙture algÈbrique $\overline F_w$
de $F_w$.
\end{nota}

Soit $B$ une algËbre ‡ 
division centrale sur $F$ de dimension $d^2$ telle qu'en toute place $x$ de $F$,
$B_x$ est soit dÈcomposÈe soit une algËbre ‡ division et on suppose $B$ 
munie d'une involution de
seconde espËce $*$ telle que $*_{|F}$ est la conjugaison complexe $c$. Pour
$\beta \in B^{*=-1}$, on note $\sharp_\beta$ l'involution $x \mapsto x^{\sharp_\beta}=\beta x^*
\beta^{-1}$ et $G/\Qm$ le groupe de similitudes, notÈ $G_\tau$ dans \cite{h-t}, dÈfini
pour toute $\Qm$-algËbre $R$ par 
$$
G(R)  \simeq   \{ (\lambda,g) \in R^\times \times (B^{op} \otimes_\Qm R)^\times  \hbox{ tel que } 
gg^{\sharp_\beta}=\lambda \}
$$
avec $B^{op}=B \otimes_{F,c} F$. 
Si $x$ est une place de $\Qm$ dÈcomposÈe $x=yy^c$ dans $E$ alors 
$$G(\Qm_x) \simeq (B_y^{op})^\times \times \Qm_x^\times \simeq \Qm_x^\times \times
\prod_{z_i} (B_{z_i}^{op})^\times,$$
o˘ $x=\prod_i z_i$ dans $F^+$.

Dans \cite{h-t}, les auteurs justifient l'existence d'un $G$ comme ci-dessus tel qu'en outre:
\begin{itemize}
\item si $x$ est une place de $\Qm$ qui n'est pas dÈcomposÈe dans $E$ alors
$G(\Qm_x)$ est quasi-dÈployÈ;

\item les invariants de $G(\Rm)$ sont $(1,d-1)$ pour le plongement $\tau$ et $(0,d)$ pour les
autres. 
\end{itemize}

\begin{nota} On suppose que 
\begin{itemize}
\item $p$ est dÈcomposÈe
$p=uu^c$ dans $E$ et on note $v_1,v_2,\cdots,v_r$, les places de $F$ au dessus de $u$,

\item et qu'il existe au moins une de ces places, mettons $v_1$ que l'on notera $v$, telle que 
$(B_v^{op})^\times \simeq GL_d(F_v)$.
\end{itemize}
\end{nota}

Pour tout sous-groupe compact $U^p$ de $G(\Am^{\oo,p})$ et $m=(m_1,\cdots,m_r) \in \Zm_{\geq 0}^r$, on pose
$$U^p(m)=U^p \times \Zm_p^\times \times \prod_{i=1}^r \ker ( \OC_{B_{v_i}}^\times \longto
(\OC_{B_{v_i}}/\PC_{v_i}^{m_i})^\times )$$

\begin{nota} \label{nota-m1}
On note $\IC$ l'ensemble des sous-groupes compacts ouverts $U^p(m)$ tels qu'il existe une place $x$ pour laquelle la projection de $U^p$ sur $G(\Qm_x)$ ne contienne
aucun ÈlÈment d'ordre fini autre que l'identitÈ, cf. \cite{h-t} bas de la page 90.
Pour $m$ comme ci-dessus, on a une application
$m_1: \IC \longrightarrow \Nm.$
\end{nota}

\begin{defi}
Pour tout $I \in \IC$, on note $X_I \rightarrow \spec \OC_v$ \og la variÈtÈ de Shimura
associÈe ‡ $G$\fg{} construite dans \cite{h-t} et $X_\IC=(X_I)_{I \in \IC}$
le schÈma de Hecke relativement au groupe  $G(\Am^\oo)$, au sens de \cite{boyer-invent2}
\end{defi}

\rem les morphismes de restriction du niveau $r_{J,I}:X_J \rightarrow X_I$ sont finis et plats.
et mÍme Ètales quand $m_1(J)=m_1(I)$.

\begin{notas} (cf. \cite{boyer-invent2} \S 1.3) \label{nota-strate}
Pour $I \in \IC$, on note:
\begin{itemize}
\item $X_{I,s}$ la fibre spÈciale de $X_I$ et $X_{I,\bar s}:=X_{I,s} \times \spec \overline \Fm_p$ 
la fibre spÈciale gÈomÈtrique.

\item Pour tout $1 \leq h \leq d$, $X_{I,\bar s}^{\geq h}$ (resp. $X_{I,\bar s}^{=h}$)
dÈsigne la strate fermÈe (resp. ouverte) de Newton de hauteur $h$, i.e. le sous-schÈma dont la
partie connexe du groupe de Barsotti-Tate en chacun de ses points gÈomÈtriques
est de rang $\geq h$ (resp. Ègal ‡ $h$).

\item On notera aussi $X^{\geq 0}_{I,\bar s}:=X_I$.
\end{itemize}
\end{notas}

\rem pour tout $1 \leq h \leq d$, la strate de Newton de hauteur $h$ est
de pure dimension $d-h$; le systËme projectif associÈ dÈfinit alors
un schÈma de Hecke $X_{\IC,\bar s}^{\geq h}$ (resp. $X_{\IC,\bar s}^{=h}$)
pour $\Gm=G(\Am^\oo)$, cf. \cite{h-t} III.4.4, lisse 
dans le cas de bonne rÈduction, i.e. quand $m_1=0$.

\begin{prop} (cf. \cite{h-t} p.116) \label{prop-strate-induite}
Pour tout $1 \leq h < d$, les strates $X_{I,\bar s}^{=h}$ sont gÈomÈtriquement induites sous
l'action du parabolique $P_{h,d-h}(\OC_v)$ \footnote{cf. l'appendice \ref{para-rep} pour les notations} au sens o˘
il existe un sous-schÈma fermÈ $X_{I,\bar s,\overline{1_h}}^{=h}$ tel que:
$$X_{I,\bar s}^{=h} \simeq X_{I,\bar s,\overline{1_h}}^{=h} 
\times_{P_{h,d-h}(\OC_v/\PC_v^{m_1})} GL_d(\OC_v/\PC_v^{m_1}).$$
Pour $h=0$, on ne dispose que d'une unique strate et $X^{\geq 0}_{\IC,\bar s,\overline{1_0}}$
dÈsignera encore $X_\IC$. 
\end{prop}

Soit $\GC(h)$ le groupe de Barsotti-Tate universel sur $X_{I,\bar s,\overline{1_h}}^{=h}$:
$$0 \rightarrow \GC(h)^c \longrightarrow \GC(h) \longrightarrow \GC(h)^{et} \rightarrow 0$$
o˘ $\GC(h)^c$ (resp. $\GC(h)^{et}$) est connexe (resp. Ètale) de dimension $h$ (resp. $d-h$).
Notons 
$\iota_{m_1}:(\PC_v^{-m_1}/\OC_v)^d \longrightarrow \GC(h)[p^{m_1}]$
la structure de niveau universelle. Notant $(e_i)_{1 \leq i \leq d}$ la base canonique de
$(\PC_v^{-m_1}/\OC_v)^d$, la strate de Newton 
$X_{I,\bar s,\overline{1_h}}^{=h}$ est alors dÈfinie par la propriÈtÈ
que $\bigl \{ \iota_{m_1}(e_i):~1 \leq i \leq h \bigr \}$ forme une base de Drinfeld de 
$\GC(h)^c[p^{m_1}]$.

\begin{nota} 
Dans la suite, et afin de ne pas alourdir les notations, on confondra un ÈlÈment
$a \in GL_d(F_v)/P_{h,d-h}(F_v)$ avec le sous-espace vectoriel 
$\langle a(e_1),\cdots,a(e_h) \rangle$ engendrÈ par les images par $a$ des
$h$ premiers vecteurs $e_1,\cdots, e_h$ de la base canonique de $F_v^d$.
On notera aussi $P_a:=aP_{h,d-h}a^{-1}$ le parabolique stabilisant $a \subset F_v^d$.
\end{nota}

Pour tout idÈal $I \in \IC$, l'ÈlÈment $a \in GL_d(F_v)/P_{h,d-h}(F_v)$ 
fournit un facteur direct $a_{m_1}$ de $(\PC_v^{-m_1}/\OC_v)^d$ et donc une strate $X_{I,\bar s,a}^{=h}$.
Dans le cas o˘ $a \in GL_d(\OC_v)$, la strate $X_{I,\bar s,a}^{=h}$ s'obtient aussi comme l'image
par $a$ de $X_{I,\bar s,\overline{1_h}}^{=h}$.

\rem $X_{I,\bar s,a}^{=h}$ peut se dÈfinir directement en demandant que pour $(f_1,\cdots,f_h)$
une base de $a_{m-1}$, 
$\bigl \{ \iota_{m_1}(f_1):~1 \leq i \leq h \}$ forme une base de Drinfeld de 
$\GC(h)^c[p^{m_1}]$.

\begin{nota} \phantomsection \label{nota-1h}
Pour tout $a \in GL_d(F_v) / P_{h,d-h}(F_v)$, on notera
$X_{I,\bar s,a}^{\geq h}$ l'adhÈrence dans $X_{I,\bar s}^{\geq h}$ de $X_{I,\bar s,a}^{=h}$. 
On dira d'une telle strate
qu'elle est pure.% en comparaison des situations ci-aprËs.
%Quand la situation ne prÍtera pas ‡ confusion, on notera simplement $X^{=h}_{\IC,\bar s,1}$ pour $X^{=h}_{\IC,\bar s,\overline{1_h}}$.
\end{nota}

\rem $X_{I,\bar s,a}^{\geq h} \backslash X_{I,\bar s,a}^{=h}$ s'Ècrit comme la rÈunion des
strates $X_{I,\bar s,b}^{=h'}$ pour $h < h' \leq d$ et $b \in GL_d(F_v)/P_{h',d-h'}(F_v)$ tel qu'avec les 
notations prÈcÈdentes, $a_{m_1} \subset b_{m_1}$.

\rem
Le systËme projectif $(X_{I,\bar s,\overline{1_h}}^{=h})_{I \in \IC}$ dÈfinit un schÈma de Hecke 
$X_{\IC,\bar s,\overline{1_h}}^{=h}$ pour $G(\Am^{\oo,v}) \times P_{h,d-h}(K)$ o˘ $P_{h,d-h}(K)$
agit ‡ travers le quotient $\Zm \times GL_{d-h}(K)$ de son Levi, via l'application $\left (
\begin{array}{cc} g_v^c & * \\ 0 & g_v^{et} \end{array} \right ) \mapsto (v(\det g_v^c),g_v^{et})$. 
On dira de l'action de $GL_h(F_v)$ qu'elle est \emph{infinitÈsimale}.
Le groupe de Weil $W_v$ en $v$ agit sur 
$(X_{I,\bar s,\overline{1_h}}^{=h})_{I \in \IC}$
via son quotient $-\deg: W_v \twoheadrightarrow \Zm$ o˘ $\deg$ est la composÈe du
caractËre non ramifiÈ de $W_v$, qui envoie les Frobenius gÈomÈtriques sur les uniformisantes, avec la 
valuation $v$ de $K$.

%\bigskip

Dans la suite du texte, nous utiliserons de maniËre cruciale la situation suivante relativement au choix
de $a \in GL_d(F_v)/P_{h,d-h}(F_v)$ et $c \in GL_d(F_v)/P_{h+1,d-h-1}(F_v)$ tels que
$a \subset c$.  Pour tout $g \geq 1$, on introduit les strates suivantes.
\begin{itemize}
\item Soit $X^{=h+g}_{\IC,\bar s,c}$ la rÈunion de strates pures
$X^{=h+g}_{\IC,\bar s,c}:=\displaystyle{\bigcup_{\atop{b:~ c \subset b}{\dim b=h+g}}} 
X^{=h+g}_{\IC,\bar s,b},$
et $X^{\geq h+g}_{\IC,\bar s,c}$ son adhÈrence.

\item Pour tout $I \in \IC$, on a une inclusion de strates pures
$X^{\geq h+1}_{I,\bar s,c} \hookrightarrow X^{\geq h}_{I,\bar s,a}$ et on utilisera le schÈma
$X^{\geq h}_{I,\bar s,a} \setminus X^{\geq h+1}_{I,\bar s,c}$ ainsi que pour $g \geq 1$:
$X^{\geq h+g}_{I,\bar s,a} \setminus X^{\geq h+g}_{I,\bar s,c}.$
\end{itemize}

\rem lorsque $I$ varie dans $\IC$, le systËme des 
$X^{\geq h}_{I,\bar s,a} \setminus X^{\geq h+1}_{I,\bar s,c}$
n'est pas projectif. Prenons par exemple $h=1$ avec $a=\langle e_1 \rangle$ et 
$c=\langle e_1,e_2 \rangle$. Pour $b=\langle e_1,e_2+\varpi_v^2 e_3 \rangle$ alors
$X_{I,\bar s,b}^{=3}$ est contenu dans $X^{\geq 1}_{I,\bar s,a} \setminus X^{\geq 2}_{I,\bar s,c}$
si et seulement si $m_1 \geq 3$.

\begin{notas} \label{nota-strate2} \phantomsection
Avec la convention que $i$ (resp. $j$) correspond ‡ une immersion fermÈe
(resp. une inclusion ouverte), on note
$$i^h:X_{\IC,\bar s}^{\geq h} \hookrightarrow X_{\IC,\bar s}^{\geq 1}, \qquad
i^h_a:X_{\IC,\bar s,a}^{\geq h} \hookrightarrow X_{\IC,\bar s}^{\geq 1}, \qquad
j^{\geq h}:X_{\IC,\bar s}^{=h} \hookrightarrow X_{\IC,\bar s}^{\geq h},  \qquad 
j^{\geq h}_{1}:X_{\IC,\bar s,\overline{1_h}}^{=h} \hookrightarrow X_{\IC,\bar s,\overline{1_h}}^{\geq h},$$
ainsi que 
$$j^{\geq h}_{a-c}: X^{\geq h}_{I,\bar s,a} \setminus
X^{\geq h+1}_{I,\bar s,c} \hookrightarrow X^{\geq h}_{I,\bar s,a}, \hbox{ et }
j^{\geq h+g}_{c}:X^{=h+g}_{\IC,\bar s,c} \hookrightarrow X^{\geq h+g}_{\IC,\bar s,c}.$$
Pour tout $g \geq 1$, on introduit aussi
$$j^{h+g}_{a-c}:X^{\geq h+g}_{I,\bar s,a} \setminus X^{\geq h+g}_{I,\bar s,c} \hookrightarrow 
X^{\geq h+g}_{I,\bar s,a}.$$
\rem lorsque $a=\overline{1_h}$, on utilisera $j^{h+g}_{\neq c}$ pour dÈsigner 
$j^{h+g}_{\overline{1_h}-c}$.

Enfin afin de de n'avoir que des faisceaux sur $X^{\geq 1}_{\IC,\bar s}$, on utilisera aussi la 
notation\footnote{ Contrairement aux prÈconisations prÈcÈdentes, ne sont pas des inclusions ouvertes.}
$$j^{=h}:=i^{h} \circ j^{\geq h}, \qquad j^{=h}_{\neq c}:=i^h_{\overline{1_h}} \circ 
j^{\geq h}_{\overline{1_h}-c}, \quad
j^{=h+g}_{\neq c}=i^{h+g}_{\overline{1_h}}  \circ j^{h+g}_{\neq c}.$$
\end{notas}

\rem On renvoie le lecteur au \S \ref{para-comple} pour des complÈments sur $j^{=h+g}_{\neq c}$.

%\rem l'ensemble des strates $X^{\geq h+\delta}_{\IC,\bar s,b}$ contenues dans 
%$X^{\geq h}_{\IC,\bar s,\overline{1_h}}$ est $P_{h,d-h}(F_v)/P_{h,\delta,d-h-\delta}(F_v)$.

\subsection{Filtrations de stratification d'aprËs \texorpdfstring{\cite{boyer-torsion}}{Lg}}
\label{para-rappel-filtration}

Soient $S$ le spectre d'un corps fini et $X$ un schÈma de type fini sur $S$, alors
la $t$-structure usuelle sur $\DC(X,\Lambda):=D^b_c(X,\Lambda)$ est
$$\begin{array}{l}
A \in \lexp p \DC^{\leq 0}(X,\Lambda)
\Leftrightarrow \forall x \in X,~\hi^k i_x^* A=0,~\forall k >- \dim \overline{\{ x \} } \\
A \in \lexp p \DC^{\geq 0}(X,\Lambda) \Leftrightarrow \forall x \in X,~\hi^k i_x^! A=0,~\forall k <- \dim \overline{\{ x \} }
\end{array}$$
o˘ $i_x:\spec \kappa(x) \hookrightarrow X$ et $\hi^k(K)$ dÈsigne le $k$-iËme faisceau
de cohomologie de $K$.

\begin{nota} 
On note $\lexp p \CC(X,\Lambda)$, ou simplement $\lexp p \CC$ quand le contexte est clair, le c{\oe}ur de cette $t$-structure.
Les foncteurs cohomologiques associÈs seront notÈs $\lexp p \hi^i$; pour un foncteur
$T$, on notera $\lexp p T:=\lexp p \hi^0 \circ T$.
\end{nota}

\rem  $\lexp p \CC(X,\Lambda)$ est une catÈgorie abÈlienne noethÈrienne et 
$\Lambda$-linÈaire.
Pour $\Lambda$ un corps, cette $t$-structure est autoduale pour la dualitÈ de Verdier.
Pour $\Lambda=\overline \Zm_l$, on peut munir la catÈgorie abÈlienne $\overline \Zm_l$-linÈaire 
$\lexp p \CC(X,\overline \Zm_l)$ d'une thÈorie de torsion $(\TC,\FC)$ o˘ 
$\TC$ (resp. $\FC$) est la sous-catÈgorie
pleine des objets de $l^\oo$-torsion $T$ (resp. $l$-libres $F$) , i.e. tels que $l^N 1_T$ est nul pour $N$ assez grand (resp. $l.1_F$ est un monomorphisme).

\begin{defi} Soit d'aprËs \cite{juteau}, la $t$-structure duale 
$$\begin{array}{l}
\lexp {p+} \DC^{\leq 0}(X,\overline \Zm_l):= \{ A \in \lexp p \DC^{\leq 1}(X,\overline \Zm_l):~
\lexp p \hi^1(A) \in \TC \} \\
\lexp {p+} \DC^{\geq 0}(X,\overline \Zm_l):= \{ A \in \lexp p \DC^{\geq 0}(X,\overline \Zm_l):~
\lexp p \hi^0(A) \in \FC \} \\
\end{array}$$
de c\oe ur $\lexp {p+} \CC(X,\overline \Zm_l)$  muni de sa thÈorie de torsion 
$(\FC,\TC[-1])$ \og duale \fg{} de celle de $\lexp p \CC(X,\overline \Zm_l)$.
\end{defi}

\rem pour $j:U \hookrightarrow X \hookleftarrow F:i$ avec $U$ ouvert de complÈmentaire $F$, 
la $t$-structure ainsi dÈfinie sur $X$ muni de la thÈorie de torsion prÈcÈdente, est obtenue par 
recollement ‡ partir de celles sur $U$ et $F$ selon la recette
$$\begin{array}{l}
\lexp p \DC^{\leq 0}(X,\Lambda):=\{ K \in \DC(X,\Lambda):~j^* K \in \lexp p \DC^{\leq 0}(U,\Lambda) 
\hbox{ et } i^* K \in \lexp p \DC^{\leq 0} (F,\Lambda)\} \\
\lexp p \DC^{\geq 0}(X,\Lambda):=\{ K \in \DC(X,\Lambda):~j^* K \in \lexp p \DC^{\geq 0}(U,\Lambda)
\hbox{ et } i^! K \in \lexp p \DC^{\geq 0}(F,\Lambda) \}.
\end{array}$$
o˘ les thÈories de torsion sont reliÈes par
$$\begin{array}{l}
\TC:=\{ P \in \lexp p \CC(X,\Lambda):~\lexp p i^* P \in \TC_F  \hbox{ et } j^* P \in \TC_U \} \\
\FC:=\{ P \in \lexp p \CC(X,\Lambda):~\lexp p i^! P \in \FC_F  \hbox{ et } j^* P \in \FC_U \}
\end{array}$$

\begin{defi} (cf. \cite{boyer-torsion} \S 1.3) \label{defi-FC}
Soit 
$$\FC(X,\Lambda):=\lexp p \CC(X,\Lambda) \cap \lexp {p+} \CC(X,\Lambda)
=\lexp p \DC^{\leq 0}(X,\Lambda) \cap \lexp {p+} \DC^{\geq 0}(X,\Lambda)$$
la catÈgorie quasi-abÈlienne des faisceaux pervers \og libres \fg{} sur $X$ ‡ coefficients dans
$\Lambda$. On identifiera aussi $\FC(F,\Lambda)$ avec son image dans
$\FC(X,\Lambda)$ via le foncteur $i_*=i_!=i_{!*}$.
\end{defi}

\begin{lemm} (cf. \cite{boyer-torsion} lemme 1.3.11)
Pour $j:U \hookrightarrow X$ un ouvert, on a
$$\lexp {p+} j_! \FC(U,\Lambda) \subset \FC(X,\Lambda) \quad \hbox{ et } \quad \lexp p j_* 
\FC(U,\Lambda) \subset \FC(X,\Lambda).$$
\end{lemm}

\rem si $j_!$ est $t$-exact alors, cf. \cite{boyer-torsion} proposition 1.3.14, $j_!=\lexp p j_!=\lexp {p+} j_!$
et donc $j_! (\FC(U,\Lambda)) \subset \FC(X,\Lambda)$. D'aprËs le lemme \ref{lem-jaffine},
ce sera en particulier le cas pour les $j^{=h}$ du paragraphe prÈcÈdent.

\begin{lemm} \label{lem-libre0} \phantomsection
Soit $L \in \FC(X,\Lambda)$ tel que $j_! j^* L \in \FC(X,\Lambda)$, alors 
$i_*\lexp p \hi^{-\delta} i^* L$ est nul pour tout $\delta \neq 0,1$; pour $\delta=1$
c'est un objet de $\FC(X,\Lambda)$.
\end{lemm}

\begin{proof}
Partons du triangle distinguÈ
$j_!j^* L \longrightarrow L \longrightarrow i_*i^* L \leadsto$.
En utilisant la perversitÈ de $L$ et $j_! j^* L$, la suite exacte longue de cohomologie perverse
du triangle distinguÈ prÈcÈdent s'Ècrit
$$0 \rightarrow i_* \lexp p \hi^{-1} i^* L \longrightarrow \lexp p j_! j^* L
\longrightarrow L \longrightarrow i_* \lexp p \hi^0i^* L \rightarrow 0.$$
La libertÈ de $ i_* \lexp p \hi^{-1} i^* L$ dÈcoule alors de celle,
par hypothËse, de $\lexp p j_! j^* L=j_!j^* L$.
\end{proof}

Rappelons, cf. \cite{boyer-torsion} \S 1.3, que tout morphisme $f:L \longrightarrow L'$
de $\FC(X,\Lambda)$ possËde:
\begin{itemize}
\item un noyau $\ker_\FC f$ qui est le $p$-noyau de $f$, i.e. dans $\lexp p \CC(X,\Lambda)$;

\item un conoyau $\coker_\FC f$ qui est le $p+$-conoyau de $f$, i.e. dans $\lexp {p+} \CC(X,\Lambda)$;

\item une image $\im_\FC f$ qui est la $p+$-image de $f$;

\item une coimage $\coim_\FC f$ qui est la $p$-image de $f$;
\end{itemize}
tels que 
$0 \rightarrow \ker_\FC f \longrightarrow L \longrightarrow \coim_\FC f \rightarrow 0$ et
$0 \rightarrow \im_\FC f \longrightarrow L' \longrightarrow \coker_\FC\rightarrow 0$
sont des suites strictement exactes de $\FC(X,\Lambda)$, o˘ le qualificatif strict est
rappelÈ dans la dÈfinition suivante.

\begin{defi} \label{defi-strict}
Un morphisme $f:L \longrightarrow L'$ de $\FC(X,\Lambda)$ est dit \emph{strict} et 
on note $f: L \rarrow L'$, si la flËche canonique
$\coim_\FC f \longrightarrow \im_\FC f$ est un isomorphisme.
\end{defi}

\rem un monomorphisme $f:L \hookrightarrow L'$ dans $\lexp p \CC(X,\Lambda)$ entre
faisceaux pervers libres, est strict si et seulement si son conoyau dans $\lexp p \CC(X,\Lambda)$
est libre. Cela revient aussi ‡ demander que $f$ est un monomorphisme de
$\lexp {p+} \CC(X,\Lambda)$.

\begin{defi} Un \emph{bimorphisme} de $\FC(X,\Lambda)$ est un monomorphisme qui est aussi un Èpimorphisme.
\end{defi}

\noindent \textit{Exemple}: $\coim_\FC f \longrightarrow \im_\FC f$ est un bimorphisme.

\begin{nota} \label{nota-htarrow1}
On notera $L \htarrow L'$ un bimorphisme de $\FC(X,\Lambda)$. Si en outre 
le noyau dans $\lexp {p+} \CC(X,\Lambda)$ est de dimension strictement
plus petite que celle du support de $L$, on notera
$L \htarrow_+ L'$.
\end{nota}

\rem tout morphisme $f:L \longrightarrow L'$ de $\FC(X,\Lambda)$ admet, cf. 
\cite{boyer-torsion} proposition 1.3.7, une factorisation canonique
$L \tarrow \coim_\FC f \htarrow \im_\FC f \harrow L'.$

\begin{defi}  \phantomsection \label{defi-F-filtration}
Pour $L$ un objet de $\FC(X,\Lambda)$, on dira que 
$$L_1 \subset L_2 \subset \cdots \subset L_e=L$$ 
est une $\FC$-filtration
si pour tout $1 \leq i \leq e-1$, $L_i \hookrightarrow L_{i+1}$ est un monomorphisme strict.
\end{defi} 

%\rem cela revient ‡ demander que le conoyau dans $\lexp p \CC(X,\Lambda)$ des monomorphismes
%$L_i \hookrightarrow L_{i+1}$ de $\lexp p \CC(X,\Lambda)$, sont libres.

Pour $L \in \FC(X,\Lambda)$, on considËre le diagramme suivant 
$$\xymatrix{ 
& L \ar[drr]^{\can_{*,L}} \\
\lexp {p+} j_! j^* L \ar[ur]^{\can_{!,L}} 
\ar@{->>}[r]|-{+} & \lexp p j_{!*}j^* L \ar@{^{(}->>}[r]_+ & 
\lexp {p+} j_{!*} j^* L \ar@{^{(}->}[r]|-{+} & \lexp p j_*j^* L
}$$
o˘ la ligne du bas est, cf. la remarque suivant 1.3.12 de \cite{boyer-torsion},
la factorisation canonique de
$\lexp {p+} j_! j^* L \longrightarrow \lexp {p} j_* j^* L$
et les flËches $\can_{!,L}$ et $\can_{*,L}$ donnÈes par adjonction. 

\begin{defi} \phantomsection \label{defi-filtration0}
On note 
$\PC_L:=i_*\lexp p \hi^{-1}_{libre} i^*j_* j^* L =\ker_\FC \Bigl ( \lexp {p+} j_! 
j^* L \twoheadrightarrow \lexp p j_{!*} j^* L \Bigr ).$
Avec les notations du diagramme ci-dessus, on pose
$$\Fil^0_{U,!}(L)=\im_\FC (\can_{!,L}) \quad \hbox{ et } \quad \Fil^{-1}_{U,!}(L)=
\im_\FC \Bigl ( (\can_{!,L})_{|\PC_L} \Bigr ).$$
\end{defi}

\rem d'aprËs le lemme 2.1.2 de \cite{boyer-torsion},
$\Fil^{-1}_{U,!}(L) \subset \Fil^0_{U,!}(L) \subset L$
est une $\FC$-filtration au sens de \ref{defi-F-filtration}, avec
$L/\Fil^0_{U,!}(L) \simeq i_* \lexp {p+} i^* L$ et 
$\lexp p j_{!*} j^* L \htarrow_+ \Fil^0_{U,!}(L)/\Fil^{-1}_{U,!}(L)$, ce qui d'aprËs 
le lemme 1.3.13 de \cite{boyer-torsion} donne un triangle commutatif
$$\xymatrix{
\lexp p j_{!*} j^* L \ar@{^{(}->>}[rr]_+ \ar@{^{(}->>}[drr]_+ 
& & \Fil^0_{U,!}(L)/\Fil^{-1}_{U,!}(L) \ar@{^{(}->>}[d]_+ \\
& & \lexp {p+} j_{!*} j^* L.
}$$

\begin{defi} \phantomsection \label{defi-saturee}
La filtration $\Fil_{U,!}^{-1}(L) \subset \Fil^0_{U,!}(L) \subset L$
est dite \emph{saturÈe} si $\can_{!,L}$ est strict i.e. si
$\Fil^0_{U,!}(L)=\coim_\FC(\can_{!,L}).$
\end{defi}

Soit $X$ un schÈma muni 
\begin{itemize}
\item d'une stratification
$\SF=\{ X=X^{\geq 1} \supset X^{\geq 2} \supset \cdots \supset X^{\geq e} \}$ et

\item de la donnÈe pour tout $1 \leq h \leq e$ d'un ensemble fini $\LFF(h)$ de classes
d'isomorphismes de $\overline \Qm_l$-faisceaux localement constants irrÈductibles
sur $X^{=h}:=X^{\geq h}-X^{\geq h+1}$ tels qu'en notant $j^{=h}=i^h \circ j^{\geq h}$ avec
$j^{\geq h}:X^{=h} \hookrightarrow X^{\geq h}$ et $i^h:X^{\geq h} \hookrightarrow X$,
$$\forall \LC \in \LFF(h),\quad R^nj^{=h}_* \LC \hbox{ est } (\SF,\LFF)-\hbox{constructible}$$
au sens o˘ sa restriction ‡ tous les $X^{=h'}$ est une extension itÈrÈe finie de faisceaux
irrÈductibles localement constants appartenant ‡ $\LFF(h')$.
\end{itemize}

\begin{nota}
On note $D^b_{\SF,\LFF}(X,\overline \Qm_l)$ la sous-catÈgorie pleine de $D^b(X,\overline \Qm_l)$
des complexes dont les faisceaux de cohomologie sont $(\SF,\LFF)$-constructibles au sens prÈcÈdent
ainsi que $\lexp p \CC_{\SF,\LFF}(X,\overline \Qm_l)$ la catÈgorie abÈlienne des 
$\overline \Qm_l$-faisceaux pervers $(\SF,\LFF)$-constructibles.
\end{nota}

\begin{defi} \label{defi-support}
Un faisceau pervers $P \in \lexp p \CC_{\SF,\LFF}(X,\overline \Qm_l)$ est dit ‡ support dans $X^{\geq h}$, 
si $h$ est le plus petit entier $1 \leq r \leq d$ tel que $j^{=r,*} P$ est un systËme local non nul.
Un faisceau pervers libre $P \in \FC(X,\Lambda)$ sera dit ‡ support dans $X^{\geq h}$
si $P \otimes_{\overline \Zm_l} \overline \Qm_l \in \lexp p \CC_{\SF,\LFF}(X,\overline \Qm_l)$
l'est au sens prÈcÈdent.
\end{defi}

\rem dans la suite on appliquera ces dÈfinitions pour la stratification de Newton de la fibre spÈciale
d'une variÈtÈ de Shimura de Kottwitz-Harris-Taylor, \emph{nous ferons 
alors disparaitre $\SF$ des notations}, avec $\LFF$ constituÈ des systËmes locaux
d'Harris-Taylor. La propriÈtÈ de stabilitÈ de $\LFF$ par les foncteurs $Rj^{=h}_*$
dÈcoule des rÈsultats de \cite{boyer-invent2}. Dans cette situation on dira que $P$ est ‡ support dans 
$X^{\geq h}_{\IC,\bar s,\AC}$ s'il est
‡ support dans $X^{\geq h}_{\IC,\bar s}$ et si pour toute strate pure $X^{\geq h}_{\IC,\bar s,a}$, 
$j^{=h,*}_a P$ est un systËme local non nul si et seulement si $a \in \AC$.

On cherche ‡ prÈsent ‡ construire des filtrations d'un $L \in \FC(X,\Lambda)$ tel que
$L \otimes_{\overline \Zm_l} \overline \Qm_l \in \lexp p \CC_{\SF,\LFF}(X,\overline \Qm_l)$, 
‡ l'aide de la stratification $\SF$.

\begin{defi} \label{defi-SF-filtration}
Soit $L \in \FC(X,\overline \Zm_l)$ tel que 
$L \otimes_{\overline \Zm_l} \overline \Qm_l \in \lexp p \CC_{\SF,\LFF}(X,\overline \Qm_l)$.
On dira d'une $\FC$-filtration 
$$0=\Fil^{0}_{\SF}(L) \subset \Fil^1_{\SF}(L) \subset \Fil^2_{\SF}(L) \cdots \subset 
\Fil^{e-1}_{\SF}(L) \subset \Fil^e_{\SF}(L)=L,$$
qu'elle est de type $\SF_!$ si pour tout $k$ et $h$ tel que $\gr_{\SF}^k(L)$ est ‡ support dans 
$X^{\geq h}$, le morphisme d'adjonction
$$\lexp {p+} j^{=h}_! j^{=h,*} \gr_{\SF}^k(L) \longrightarrow \gr_{\SF}^k(L)$$
est un Èpimorphisme, pas nÈcessairement strict, de $\FC(X,\Lambda)$.
\end{defi}

%Dans \cite{boyer-torsion}, on donne une construction fonctorielle d'une $\FC$-filtration de type $\SF_!$
%appelÈe filtration de $!$-stratification ou plus simplement de stratification.

%
%\rem si $L \in \FC(X,\overline \Zm_l)$ est ‡ support dans $X^{\geq h}$ alors
%$j^{=h,*} L$ est un $\overline \Zm_l$-systËme local sur $X^{=h}$. 

%Pour $F,F' \in \FC(X,\Lambda)$ dont le support est contenu dans $X^{\geq h}$ mais pas dans $X^{\geq h+1}$, on notera $F \htarrow_+ F'$ une flËche $F \htarrow F'$ au sens de \ref{nota-htarrow1} dont le conoyau est ‡ support dans $X^{\geq h+1}$.

\begin{defi} 
Pour $1 \leq h <e$, on note $X^{1 \leq h}:=X^{\geq 1}-X^{\geq h+1}$ et
$j^{1 \leq h}:X^{1†\leq h} \hookrightarrow X^{\geq 1}$.
Pour $L \in \FC(X,\overline \Zm_l)$ 
%tel que $L \otimes_{\overline \Zm_l} \overline \Qm_l \in \lexp p \CC_{\SF,\LFF}(X,\overline \Qm_l)$, 
soit
$$\Fil^r_{\SF,!}(L):=\im_\FC \Bigl ( \lexp {p+} j^{1 \leq r}_! j^{1 \leq r,*} L \longrightarrow L \Bigr ).$$
\end{defi}

\begin{prop} \phantomsection \label{prop-ss-filtration} (cf. \cite{boyer-torsion} \S 2.2)
La dÈfinition prÈcÈdente munit fonctoriellement tout objet $L$ de $\FC(X,\Lambda)$ d'une 
$\FC$-filtration de type $\SF_!$ dite de stratification
$$0=\Fil^{0}_{\SF,!}(L) \subset \Fil^1_{\SF,!}(L) \subset \Fil^2_{\SF,!}(L) \cdots \subset 
\Fil^{e-1}_{\SF,!}(L) \subset \Fil^e_{\SF,!}(L)=L.$$
\end{prop}

%\rem la dÈfinition prÈcÈdente ne requiert pas que  
%$L \otimes_{\overline \Zm_l} \overline \Qm_l \in \lexp p \CC_{\SF,\LFF}(X,\overline \Qm_l)$ mais la filtration
%obtenue ne sera Èvidemment \og efficace \fg{} que pour ces derniers.

%
%\rem dualement on peut dÈfinir
%$\Fil_{\SF,*}^{-r}(L)=\ker_\FC \bigl ( \can_{*,L}:L \longrightarrow
% \lexp {p} j^{1 \leq r}_* j^{1 \leq r,*} L \bigr )$ de faÁon ‡ obtenir une $\FC$-filtration
% $0=\Fil_{\SF,*}^{-e}(L) \subset \Fil_{\SF,*}^{1-e}(L) \subset \cdots \subset \Fil_{\SF,*}^0(L)=L$ 
% dont on note $\gr_{\SF,*}^{-r}(L)$ les graduÈs.
 
\begin{defi} \phantomsection \label{defi-SF-saturee}
On dira que $L$ est \emph{$\SF_!$-saturÈ} si pour tout  $1 \leq r \leq e-1$ le 
morphisme d'adjonction $\lexp {p+} j^{1 \leq r}_! j^{1 \leq r,*} L \longrightarrow L$ est strict, i.e. si
$$\Fil^r_{\SF,!}(L)=\coim_\FC \bigl (\lexp {p+} j^{1 \leq r}_! j^{1 \leq r,*} L \longrightarrow L \bigr ).$$
Autrement dit si pour tout $1 \leq r \leq e-1$, $\lexp p i^{r+1,*} L$ est un objet de 
$\FC(X,\Lambda)$.
\end{defi}

\rem les filtrations de type $\SF_!$ ne sont, en gÈnÈral, pas assez fines. 
Dans \cite{boyer-torsion} proposition 2.3.3, en utilisant les $\Fil^{-1}_{U,!}$, on construit
de faÁon fonctorielle la filtration exhaustive de stratification de tout objet $L$ de 
$\FC(X,\Lambda)$, 
\addtocounter{smfthm}{1}
\begin{equation} \label{eq-fill}
0=\Fill^{-2^{e-1}}_{\SF,!}(L) \subset \Fill^{-2^{e-1}+1}_{\SF,!}(L) \subset \cdots \subset
\Fill^0_{\SF,!}(L) \subset \cdots \subset \Fill^{2^{e-1}-1}_{\SF,!}(L)=L,
\end{equation}
telle que tous les graduÈs $\grr^k_\SF$ sont libres et sont, sur $\overline \Qm_l$,
simples, i.e. vÈrifient $\lexp p j^{= h}_{!*} j^{= h,*} \grr^k_\SF  \htarrow_+ \grr^k_\SF$,
o˘ $X^{\geq h}$ est le support de $\grr^k_\SF$.
DÈcrivons rapidement la construction de loc. cit.:
\begin{itemize}
\item on commence par regarder le morphisme d'adjonction 
$\lexp {p+} j^{=1}_! j^{=1,*} L \longrightarrow L$
dont le conoyau $Q_1$ donnera des graduÈs pour des indices  strictement positifs alors que 
l'image dans $L$ du noyau $P_1$ 
de $\lexp {p+} j^{=1}_! j^{=1,*} L \longrightarrow \lexp p j^{=1}_{!*} j^{=1,*} L$ donnera des 
graduÈs pour des indices strictement nÈgatifs, le graduÈ d'indice $0$ sera tel que 
$\lexp p j^{=1}_{!*} j^{=1,*} L \htarrow_+ \gr^0_{\SF,!}(L)$.

\item On passe alors ‡ la strate suivante $X^{=2}_{\IC,\bar s}$ pour $P_1$ et $Q_1$. Pour
$F:=P_1$ (resp. $F:=Q_1$), on considËre $\lexp {p+} j^{=2}_! j^{=2,*} F \longrightarrow F$. 
Le conoyau de ce morphisme donnera des graduÈs pour les indices $-2^{e-2}< k < 0$
(resp. $2^{e-2} < k < 2^{e-1}$), alors que le noyau de 
$\lexp {p+} j^{=2}_! j^{=2,*} F \longrightarrow \lexp p j^{=2}_{!*} j^{=2,*} F$ donnera des graduÈs pour les
indices strictement infÈrieurs ‡ $-2^{e-2}$ (resp. $0<k<2^{e-2}$); le graduÈ d'indice $k=-2^{e-2}$ 
(resp. $k=2^{e-2}$) vÈrifiant $\lexp p j^{=2}_{!*} j^{=2,*} F \htarrow_+ \gr^k_{\SF,!}(L)$.

\item On traite ainsi toutes les strates jusqu'‡ $h=e$.
\end{itemize}
%Dans la situation des variÈtÈs de Shimura du \S \ref{para-shimura}, la
%stratification $\SF$ sera donnÈe par celle de Newton de \ref{nota-strate}; nous ferons 
%alors disparaitre $\SF$ des notations.

%i.e. o˘ pour tout $|r|<2^{e-1}$, les graduÈs $\grr^r_{\SF,!}(L)=\Fill^{r}_{\SF,!}(L)/\Fill^{r-1}_{\SF,!}(L) \in \FC(X,\Lambda)$ sont adaptÈs ‡ $\SF$.
%
%\begin{defi} \phantomsection \label{defi-satureee}
%On dira que $L$ est \emph{exhaustivement $\SF_!$-saturÈ} (resp. 
%\emph{exhaustivement $\SF_!$-parfait})
%si dans la construction des $\Fill^\bullet_{\SF,!}(L)$ 
%tous les morphismes d'adjonction $\can_{!}$ considÈrÈs sont stricts, cf. \ref{defi-saturee}
%(resp. des Èpimorphismes pas nÈcessairement stricts de $\FC(X,\Lambda)$).
%\end{defi}
%
%\rem $L$ est exhaustivement $\SF_!$-parfait revient ‡ demander que dans la construction dÈcrite
%ci-avant tous les conoyaux sont nuls. Autrement dit cela signifie que les seuls graduÈs non nuls
%de la filtration exhaustive de stratification ont pour indice $- 2^k$ pour $0 \leq k \leq e-1$.
%
%\rem si $L$ et $L'$ ‡ support dans une mÍme strate de $\SF$, sont exhaustivement $\SF_!$-parfait,
%leur somme directe ne l'est plus ‡ priori. Pour que ce soit vrai, il faut et il suffit que les strates de leurs 
%constituants irrÈductibles soient les mÍmes.

Dualement en utilisant $\can_{*,L}$ avec
$$\CoFil_{*,r}(L)=\coim_\FC\Bigl ( L \longrightarrow \lexp {p} j^{1 \leq r}_* j^{1 \leq r,*}L \Bigr ),$$
on dÈfinit une cofiltration
%\begin{multline*}
$$L = \CoFil_{\SF,*,d}(L) \twoheadrightarrow \CoFil_{\SF,*,d-1}(L) \twoheadrightarrow % \cdots \\
\cdots
\twoheadrightarrow \CoFil_{\SF,*,1}(L) \twoheadrightarrow \CoFil_{\SF,*,0}(L)=0,$$
%\end{multline*}
et
$0=\Fil_{*}^{-d}(L) \subset \Fil_{*}^{1-d}(L) \subset \cdots \subset \Fil_{*}^0(L)=L$
o˘ $\Fil_{*}^{-r}(L):=\ker_\FC \bigl ( L \twoheadrightarrow \CoFil_{*,r}(L) \bigr ).$

\subsection{Rappels sur les faisceaux pervers de Hecke}
\label{para-schema-hecke}

Soit $\Xm_\IC=(X_I)_{I \in \IC}$ un schÈma de Hecke pour $(\Gm,\IC)$ au sens du \S 1.2.2 
de \cite{boyer-invent2}, pour $\Gm=G(\Am^{\oo,v}) \times P(F_{v})$ o˘ $P(F_{v})$ est un sous-groupe de $G(F_{v}) \simeq GL_{d}(F_{v})$. Rappelons que
\begin{itemize}
\item $\Xm_{\IC}$ est un systËme projectif de schÈmas relativement ‡ des morphismes dits de restriction du niveau
$[1]_{J,I}:\Xm_{J} \longrightarrow \Xm_{I}$, finis et plats;

\item pour tout $g \in \Gm$ et tous $J \subset I$ tels que $g^{-1}Jg \subset I$, on dispose d'un morphisme fini
$[g]_{J,I}:\Xm_{J} \longrightarrow \Xm_{I}$ vÈrifiant les propriÈtÈs suivantes
\begin{itemize}
\item pour $g \in I$ et $J \subset I$, $[g]_{J,I}=[1]_{J,I}$;

\item pour tous $g,g' \in \Gm$, et tous $K \subset J \subset I$ tels que $g^{-1} J g \subset
I$ et $(g')^{-1} K g' \subset J$, on a
$[gg']_{K,I}= [g]_{J,I} \circ [g']_{K,J}: \Xm_K \longto \Xm_J \longto \Xm_I.$

%\item pour tout $J \subset I$, $\Xm_J/\KC_I \simeq \Xm_I$, o˘ $g \in \KC_I$ agit sur $\Xm_J$ via
%$[g]_{J,J}$.
\end{itemize}
\end{itemize}
La catÈgorie
$\FPH_{\Gm}(\Xm_\IC;\Lambda)$ (resp. $\FH_{\Gm}(\Xm_\IC;\Lambda)$) des \textit{faisceaux pervers} (resp. des \textit{faisceaux})
\textit{de Hecke} sur $\Xm_\IC$ ‡ coefficients dans $\Lambda$ est dÈfinie comme la catÈgorie dont:
\begin{itemize}
\item les objets sont les systËmes $(\FC_I)_{I \in \IC}$ o˘ $\FC_I$ est un faisceau
pervers (resp. faisceau) sur $\Xm_I$ ‡ coefficients dans $\Lambda$, tels que
pour tout $g \in \Gm$ et $J \subset I$ tel que $g^{-1} J g \subset
I$, on dispose d'un morphisme de faisceaux sur $\Xm_I$, $u_{J,I}(g):\FC_I \longto [g]_{J,I,*} \FC_J$ soumis ‡ la condition de cocycle
$u_{K,I}(g'g)=[g]_{J,I,*} (u_{K,J}(g')) \circ u_{J,I}(g)$;

\item Les flËches sont les systËmes $(f_I:\FC_I \longto \FC'_I)_{I \in \IC}$ 
avec des diagrammes commutatifs:
$$\xymatrix{
\FC_I \ar[rrr]^{u_{J,I}(g)} \ar[d]^{f_I}  & & & [g]_{J,I,*} (\FC_J) \ar[d]^{[g]_{J,I,*} (f_J)} \\
\FC_I' \ar[rrr]^{u_{J,I}(g)} & & & [g]_{J,I,*} (\FC_J')
}$$
\end{itemize}

\rem par rapport ‡ \cite{boyer-invent2} \S 1.3.7, on a supprimÈ les conditions (ii) et (iii). Les propositions 6.1 et 6.2
de loc. cit. sont encore valables, i.e. $\FPH_{\Gm}(\Xm_\IC;\Lambda)$ et $\FH_{\Gm}(\Xm_\IC;\Lambda)$ sont des catÈgories abÈliennes
munies de foncteurs $j_!, i*$ (resp. $Rj_*, i_*$, resp. $j_*,Ri^!$) qui sont $t$-exacts ‡ droite (resp. $t$-exacts,
resp. $t$ exacts ‡ gauche) avec les propriÈtÈs d'adjonction habituelles, de sorte que l'on se retrouve ‡ nouveau dans une
situation de recollement.

Pour $\Lambda=\overline \Zm_{l}$, comme les $[g]_{J,I,*}$ sont $t$-exacts, les thÈories de torsion ‡ chaque Ètage munissent
$\FPH_{\Gm}(\Xm_\IC;\Lambda)$ d'un \og systËme \fg{} de thÈories de torsion et donc d'un 
\og systËme \fg{} de $t$-structure $p+$, i.e. ‡ chaque Ètage.

\begin{nota} 
On notera $\FC(\Xm_{\IC},\overline \Zm_{l})$ la catÈgorie quasi-abÈlienne des faisceaux 
pervers \og libres \fg{} de Hecke, i.e. le systËme de Hecke des $\FC(\Xm_I,\overline \Zm_l)$
pour $I \in \IC$.
\end{nota}

\section{Sur les faisceaux pervers d'Harris-Taylor}

On renvoie le lecteur ‡ l'appendice \ref{app-B1} pour les rappels sur les systËmes
locaux d'Harris-Taylor.
On considËre dans cette section, un $\overline \Zm_l$-systËme local d'Harris-Taylor 
$HT_{\overline{1_h}}(\pi_v,\Pi_t)$ associÈ ‡ une reprÈsentation irrÈductible cuspidale $\pi_v$ de 
$GL_g(F_v)$ et $\Pi_t$ une reprÈsentation quelconque, qui ne jouera aucun rÙle, de $GL_{tg}(F_v)$
o˘ on a aussi posÈ $h=tg$. On entend par l‡ qu'on a fixÈ un $\overline \Zm_l$-rÈseau stable par
l'action de $P_{h,d-h}(F_v)$. Le but de cette section est, en particulier, de prouver que les
faisceaux de cohomologie de $\lexp p j^{=tg}_{\overline{1_{tg}},!*} HT_{\overline{1_h}}(\pi_v,\Pi_t)$
sont sans torsion.

\subsection{Rappel sur les stratifications miraboliques}

Soit $F$ un faisceau pervers libre sur $X^{\geq h}_{\IC,\bar s,\overline{1_h}}$ tel que
$j^{=h,*}_{\overline{1_h}} F$ est de la forme $HT_{\overline{1_h}}(\pi_v,\Pi_t)$. On note
$P_F:=i^{h+1}_{*} \lexp p \hi^{-1} i^{h+1,*} F$ avec
\addtocounter{smfthm}{1}
\begin{equation} \label{eq-sec-fc1}
0 \rightarrow P_F \longrightarrow j^{=h}_{\overline{1_h},!}j^{=h,*}_{\overline{1_h}} F 
\longrightarrow F \rightarrow 0.
\end{equation}

On considËre alors comme dans \cite{boyer-FT} une strate pure 
$X^{\geq h+1}_{\IC,\bar s,c} \subset X^{\geq h}_{\IC,\bar s,\overline{1_h}}$, et on note pour tout
$r,\delta \geq 0$:
$$i^{h+r \leq +\delta}_{\overline{1_h},c}: X^{\geq h+r+\delta}_{\IC,\bar s,c} \subset 
X^{\geq h+r}_{\IC,\bar s,\overline{1_h}}.$$
Soit alors le triangle distinguÈ
$$j^{=h}_{\neq c,!} j^{=h,*}_{\neq c} P_F \longrightarrow P_F \longrightarrow 
i^{h+1}_{c,*} i^{h+1,*}_c P_F \leadsto$$

\begin{lemm} \label{lem-HT1}
Le complexe $i^{h+1,*}_c P_F$ est un faisceau pervers libre.
\end{lemm}

\begin{proof}
On reprend la preuve de la proposition 6.2 de \cite{boyer-FT}.
Comme $F$ est de la forme $i^h_{\overline{1_h},*} F'$, il suffit de montrer que 
$$i^{h+1 \leq +0,*}_{\overline{1_h},c} \bigl ( \lexp p \hi^{-1} i^{h\leq +1,*}_{\overline{1_h}} F' \bigr )$$
est $p$-pervers sans torsion. Pour ce faire, on utilise la suite spectrale
$$E_2^{r,s}=\lexp p \hi^r i^{h+1 \leq +0,*}_{\overline{1_h},c} \Bigl ( \lexp p \hi^s 
i^{h \leq +1,*}_{\overline{1_h}}
F' \Bigr ) \Rightarrow \lexp p \hi^{r+s} i^{h \leq +1,*}_{\overline{1_h},c} F'.$$
Comme $j^{\geq h}_{\overline{1_h}}$
est affine, d'aprËs le lemme \ref{lem-libre0}, $\lexp p \hi^s i^{h \leq +1,*}_{\overline{1_h}} F'$ est nul
pour tout $s<-1$; la surjectivitÈ $j^{\geq h}_{\overline{1_h},!} j^{\geq h,*}_{\overline{1_h}} F' 
\twoheadrightarrow F'$, implique aussi la nullitÈ pour $s=0$. Ainsi la suite spectrale prÈcÈdente
dÈgÈnËre en $E_2$ avec
$$\lexp p \hi^r i^{h \leq +1,*}_{\overline{1_h},c} F' \simeq \lexp p \hi^{r+1} i^{h+1 \leq +0,*}_{\overline{1_h},c}
\bigl ( \lexp p \hi^{-1} i^{h \leq +1,*}_{\overline{1_h}} F' \bigr ).$$

De la mÍme faÁon, comme $j^{\geq h}_{\overline{1_h}- c}:X^{\geq h}_{\IC,\bar s,\overline{1_h}}  
\setminus X^{\geq h+1}_{\IC,\bar s,c} \hookrightarrow X^{\geq h}_{\IC,\bar s,\overline{1_h}}$ est affine,
$\lexp p \hi^r i^{h \leq +1,*}_{\overline{1_h},c} F' $ est nul pour $r<-1$ et sans torsion pour $r=-1$
d'aprËs le lemme \ref{lem-libre0}, d'o˘ le rÈsultat.
\end{proof}

Ainsi le triangle distinguÈ prÈcÈdent s'Ècrit sous la forme d'une suite exacte courte
\addtocounter{smfthm}{1}
\begin{equation} \label{eq-sec-fc2}
0 \rightarrow j^{=h}_{\neq c,!} j^{=h,*}_{\neq c} P_F \longrightarrow P_F \longrightarrow 
i^{h+1}_{c,*} i^{h+1,*}_c P_F \rightarrow 0.
\end{equation}

\emph{ConsidÈrons dans un premier temps, le cas o˘ $F=\lexp p j^{=tg}_{\overline{1_{tg}},!*}
HT_{\overline{1_h}}(\pi_v,\Pi_t)$.}
Dans \cite{boyer-FT} proposition 6.4, %Au lemme \ref{lem-Ql1}, 
on montre que sur $\overline \Qm_l$, on a 
$$i^{h+1,*}_c P_F \otimes_{\overline \Zm_l} \overline \Qm_l \simeq \lexp p j^{=(t+1)g}_{c,!*}
j^{=(t+1)g,*}_c P_F \otimes_{\overline \Zm_l} \overline \Qm_l.$$
L'objectif principal de cette section est ainsi de dÈmontrer que cet isomorphisme est encore valable sur 
$\overline \Zm_l$ sachant qu'‡ priori on sait simplement que
$$\lexp p j^{=(t+1)g}_{c,!*} j^{=(t+1)g,*}_c P_F \htarrow_+ i^{h+1,*}_c P_F \htarrow_+
\lexp {p+} j^{=(t+1)g}_{c,!*} j^{=(t+1)g,*}_c P_F.$$

Des suites exactes courtes (\ref{eq-sec-fc1}) et (\ref{eq-sec-fc2}), on en dÈduit le lemme suivant qui
nous ramËne ainsi, en raisonnant par rÈcurrence, ‡ prouver les isomorphismes (\ref{eq-iso-p}).

\begin{lemm} \label{lem-trivial}
Si les faisceaux de cohomologie de $\lexp p j^{=(t+1)g}_{!*} HT(\pi_v,\Pi_{t+1})$ sont sans torsion et si 
avec les notations prÈcÈdentes pour 
$F=\lexp p j^{=tg}_{\overline{1_{tg}},!*} HT_{\overline{1_h}}(\pi_v,\Pi_t)$, on a
\addtocounter{smfthm}{1}
\begin{equation} \label{eq-iso-p}
i_{c,*}^{tg+1} i^{tg+1,*}_c P_F \simeq \lexp p j^{=(t+1)g}_{c,!*} j^{=(t+1)g,*}_c P_F,
\end{equation}
alors les faisceaux de cohomologie de $\lexp p j^{=tg}_{!*} HT(\pi_v,\Pi_{t})$ sont sans torsion.
\end{lemm}

Pour des raisons ÈlÈmentaires, cf. la proposition \ref{prop-dec-pervers}, on ne peut pas espÈrer l'ÈgalitÈ
entre les deux extensions intermÈdiaires d'un systËme local d'Harris-Taylor.
L'un des rÈsultats principaux de ce papier, cf. la proposition \ref{prop-ext-pp}
qui sera prouvÈe au \S \ref{para-ext-pp-preuve}, est que les deux extensions
intermÈdiaires d'un systËme local d'Harris-Taylor indexÈ par une cuspidale
$\pi_v$ de $GL_g(F_v)$, 
coÔncident si et seulement si la rÈduction modulo $l$ de $\pi_v$
est encore supercuspidale. Dans le cas o˘ $g=1$ et donc lorsque $\pi_v$
 est un caractËre $\chi_v$ de $F_v^\times$, on dispose d'un argument simple.
En effet avec les notations de \S \ref{app-B1},
$\LC(\chi_v[h]_D)$ est isomorphe ‡ $\overline \Zm_l$ o˘
l'action du groupe fondamental 
$\Pi_1(X^{=h}_{\IC,\bar s})$ de $X^{=h}_{\IC,\bar s}$ se factorise par son quotient 
$\Pi_1(X^{=h}_{\IC,\bar s}) \twoheadrightarrow \DC_{v,h}^\times$, o˘ l'action de
$\DC_{v,h}^\times$ est donnÈe par un caractËre $\chi_v$.

\begin{lemm} \label{lem-ext0} \phantomsection
On a
$$\lexp p j^{\geq h}_{\overline{1_h},!*}  \LC(\chi_v[h]_D)[d-h] 
\simeq \lexp {p+} j^{\geq h}_{\overline{1_h},!*} 
\LC(\chi_v[h]_D)[d-h] \simeq \overline \Zm_l [d-h].$$
\end{lemm}

\begin{proof}
Rappelons que $X^{\geq h}_{\IC,\bar s,\overline{1_h}}$ est lisse sur $\spec  \overline \Fm_p$ de sorte que
$\overline \Zm_l[d-h]$ y est pervers au sens des deux 
$t$-structures avec $i^{h \leq +1,*}_{\overline{1_h}} \overline \Zm_l[d-h] \in 
\lexp p \DC^{< 0}$ et $i^{h \leq +1,!}_{\overline{1_h}} \overline \Zm_l[d-h] \in 
\lexp {p+} \DC^{\geq 1}$, d'o˘ le rÈsultat.
\end{proof}

Revenons ‡ $P_F$ et $j^{=h}_{\neq c,!} j^{=h,*}_{\neq c} P_F$ que l'on peut
filtrer en utilisant les constructions du \S \ref{para-rappel-filtration}. En particulier,
il a un quotient
$j^{=h}_{\neq c,!} j^{=h,*}_{\neq c} P_F \twoheadrightarrow Q_{F,\neq c}$
tel que 
$$Q_{F,\neq c} \otimes_{\overline \Zm_l} \overline \Qm_l \simeq
\bigoplus_{\atop{a:~\overline{1_h} \subset a}{c \not \subset a}} 
\lexp p j^{=h+g}_{a,!*} j^{=h+g,*}_a P_F \otimes_{\overline \Zm_l} \overline \Qm_l.$$
Notons alors $T$ le conoyau
$$\bigoplus_{\atop{a:~\overline{1_h} \subset a}{c \not \subset a}} \lexp p j^{=h+g}_{a,!*} j^{=h+g,*}_a P_F \htarrow_+ Q_{F,\neq c} \twoheadrightarrow T,$$
et remarquons que pour tout point fermÈ $z$ de 
$X^{\geq h+1}_{\IC,\bar s,c}$, avec 
$i_z:\overline{ \{ z \} } \hookrightarrow X_{\IC,\bar s}$, on a
$\lexp p h^0 i_z^* Q_{F,\neq c}=0$ et donc $\lexp p h^0 i_z^* T=0$. 
Avec la notation suivante, on peut alors Ècrire 
$Q_{F,\neq c}$ sous la forme 
$$Q_{F,\neq c} \simeq \bigoplus_{\atop{a:~\overline{1_h} \subset a}{c \not \subset a}} \lexp {p(c)} j^{=h+g}_{a,!*} j^{=h+g,*}_a P_F$$
que l'on notera plus simplement encore
$Q_{F,\neq c} \simeq \lexp {p(c)} j^{=h+g}_{\neq c,!*} j^{=h+g,*}_{\neq c} P_F$.

\begin{nota} \label{nota-jpc}
Pour $\LC$ un systËme local sur une strate pure
$X^{=h+g}_{\IC,\bar s,a}$ avec $c \not \subset a$,
la notation $\lexp {p(c)} j^{=h+g}_{a,!*} \LC[d-h-g]$ dÈsignera
une extension intermÈdiaire 
$$\lexp {p} j^{=h+g}_{a,!*} \LC[d-h-g] \htarrow_+ \lexp {p(c)} j^{=h+g}_{a,!*} \LC[d-h-g]$$
dont le conoyau $T$ vÈrifie la propriÈtÈ suivante. Pour tout point $z$ de
$X^{\geq h+1}_{\IC,\bar s,c}$, avec $i_z:\overline{ \{ z \} } \hookrightarrow X_{\IC,\bar s}$, on a 
$\lexp p h^0 i_z^* T=0$. 
\end{nota}

 On introduit alors le poussÈ en avant
$$\xymatrix{
& j^{=h}_{\neq c,!} j^{=h,*}_{\neq c} P_F \ar@{^{(}->}[r] \ar@{->>}[d] & P_F \ar@{->>}[r] \ar@{-->>}[d] &
i^{h+1}_{c,*} i^{h+1,*}_c P_F \ar@{=}[d] \\
\lexp {p(c)} j^{=h}_{\neq c,!*} j^{=h,*}_{\neq c} Q_F & Q_{F,\neq c} \ar@{=}[l] \ar@{^{(}-->}[r] & Q_F \ar@{->>}[r] & 
i^{h+1}_{c,*} i^{h+1,*}_c P_F.
}$$

Du lemme  \ref{lem-trivial}, l'isomorphisme  (\ref{eq-iso-p}) dÈcoule alors
de la proposition \ref{prop-extind} qui sera montrÈe plus loin. Supposons
ainsi dans la fin de cette section que  (\ref{eq-iso-p}) est vÈrifiÈe.

\begin{prop} \label{prop-fc-fpht}
Soient $\pi_v$ une reprÈsentation irrÈductible cuspidale de $GL_g(F_v)$ et $1 \leq t \leq \frac{d}{g}$, alors
les faisceaux de cohomologie de $\lexp p j^{=tg}_{!*} HT(\pi_v,\Pi_t)$ sont sans torsion.
\end{prop}

\begin{coro} \label{coro-fill-j}
Pour toute reprÈsentation irrÈductible cuspidale $\pi_v$ de $GL_g(F_v)$ et pour tout $1 \leq t \leq d/g$,
les graduÈs de la filtration de stratification exhaustive, cf. (\ref{eq-fill}), de
$j^{=tg}_! HT(\pi_v,\Pi_t)$ sont les $\lexp p j^{=(t+i)g}_{!*} HT(\pi_v,\Pi_t \{ - \frac{i}{2} \} \times \st_i(\pi_v)
\{ \frac{t}{2} \}) (\frac{t}{2})$.
\end{coro}

\rem le rÈsultat est dÈmontrÈ dans \cite{boyer-invent2} sur $\overline \Qm_l$, l'apport de ce
nouvel ÈnoncÈ est le fait que sur $\overline \Zm_l$, ce sont les $p$-extensions intermÈdiaires qui
interviennent.

\begin{proof}
Le cas $i=0$ est trivial puisque
$j^{=tg}_! HT(\pi_v,\Pi_t) \twoheadrightarrow \lexp p j^{=tg}_! HT(\pi_v,\Pi_t)$. Dans le cas gÈnÈral,
on se ramËne trivialement ‡ $j^{=tg}_{\overline{1_{tg}},!} HT_{\overline{1_{tg}}}(\pi_v,\Pi_t)$ et ‡ la proposition
\ref{prop-extind} qui traite le cas $i=1$. ConcrËtement soit tout d'abord
$$0 \rightarrow P_i \longrightarrow j^{=tg}_{\overline{1_{tg}},!} HT_{\overline{1_{tg}}}(\pi_v,\Pi_t) \longrightarrow P'_i \rightarrow 0,$$
tel que\footnote{Pour $i=1$ on a notÈ prÈcÈdemment $P_1=P_F$.} 
$P_i \twoheadrightarrow Q_i$ o˘ 
$$Q_i \otimes_{\overline \Zm_l} \overline \Qm_l \simeq 
\lexp p j^{=(t+i)g}_{\overline{1_{tg}},!*}
HT_{\overline{1_{tg}},\overline \Qm_l}(\pi_v,\Pi_t \{ i \frac{(g-1)}{2} \} \times \st_i(\pi_v) \{ -t\frac{(g-1)}{2} \}) (\frac{t}{2}).$$
Comme dans le lemme \ref{lem-HT1}, 
$i^{h+1,*}_c P_i \hookrightarrow i^{h+1,*}_c P_1$ est un faisceau pervers libre, ce qui fournit
une suite exacte analogue ‡ (\ref{eq-sec-fc2})
$$0 \rightarrow j^{=h}_{\neq c,!} j^{=h,*}_{\neq c} P_i \longrightarrow P_i \longrightarrow 
i^{h+1}_{c,*} i^{h+1,*}_c P_i \rightarrow 0,$$
avec
$$0 \rightarrow \lexp {p(c)} j^{=(h+i)g}_{\neq c,!*} j^{=(h+i)g,*}_{\neq c} Q_i \longrightarrow  Q_i \longrightarrow
i^{h+1}_{c,*} i^{h+1,*}_c Q_i \rightarrow 0,$$
et le rÈsultat dÈcoule de la proposition \ref{prop-extind}.
\end{proof}

%On en dÈduit aussi directement la version entiËre de \ref{lem-j-c}.

\begin{coro} \label{coro-j-c}
Le lemme \ref{lem-j-c} est valable sur $\overline \Zm_l$.
\end{coro}

Il reste alors ‡ prouver  (\ref{eq-iso-p}), ce qui est le but des paragraphes suivants.

\subsection{Extensions entre $\overline \Zm_l$-faisceaux pervers libres}

Etant donnÈs deux systËmes locaux $\LC_1, \LC_2$ sur un ouvert $j:U \hookrightarrow X$, 
il est bien connu qu'il n'y a pas,
sur $\overline \Qm_l$, d'extensions entre leurs extensions intermÈdiaires, i.e. toute suite exacte courte
$$0 \rightarrow \lexp p j_{!*} \LC_1 \otimes_{\overline \Zm_l} \overline \Qm_l [\dim X] \longrightarrow 
P_{\overline \Qm_l} 
\longrightarrow \lexp p j_{!*} \LC_2\otimes_{\overline \Zm_l} \overline \Qm_l [\dim X]
\rightarrow 0$$
est nÈcessairement scindÈe. Sur $\overline \Zm_l$, la propriÈtÈ subsiste pourvu qu'on prenne
les $p$ (resp. les $p+$) extensions intermÈdiaires pour $\LC_1$ et $\LC_2$ mais sinon il peut
y avoir de telles suites exactes courtes non scindÈes. C'est ce phÈnomËne que l'on veut Ètudier
dans ce paragraphe. ConcrËtement soient $A_1$ et $A_2$ des faisceaux pervers libres et $A$ 
une extension
$$0 \rightarrow A_1 \longrightarrow A \longrightarrow A_2 \rightarrow 0,$$
que l'on suppose scindÈe sur $\overline \Qm_l$. Notons alors $A'_2$ le tirÈ en arriËre
$$\xymatrix{
A'_2 \ar@{^{(}-->}[r] \ar@{^{(}-->}[d] & A \ar@{^{(}->}[d] \\
A_2\otimes_{\overline \Zm_l} \overline \Qm_l  \ar@{^{(}->}[r] &  A  \otimes_{\overline \Zm_l} \overline \Qm_l 
}$$
de sorte que
\addtocounter{smfthm}{1}
\begin{equation} \label{eq-prop-extension}
\xymatrix{
 & A_1 \ar@{^{(}->}[d] \ar@{=}[r] & A_1 \ar@{^{(}->}[d] \\
 A'_2 \ar@{^{(}->}[r] \ar@{=}[d] & A \ar@{->>}[r] \ar@{->>}[d] & A'_1 \ar@{->>}[d]  \\
 A'_2 \ar@{^{(}->}[r] & A_2 \ar@{->>}[r] & T  \\
}
\end{equation}
En particulier $T$ est nul si et seulement si l'extension $A$ est scindÈe, i.e. $A'_i \simeq A_i$ 
pour $i=1,2$.

\begin{lemm} \label{lem-ext-scinde0}
Avec les notations et les hypothËses ci-avant, on suppose en outre que
pour $k=1, 2$, alors $A_k \simeq i_{k,*} \lexp {p+} j_{k,!*} \LC_k[d_k]$ 
o˘ $j_k: U_k \hookrightarrow \overline U_k$ est l'immersion ouverte dans son adhÈrence
de dimension $d_k$, $i_k:\overline U_k \hookrightarrow X$ et $\LC_k$ un systËme local.
Si $T \neq 0$ alors pour tout point fermÈ $i_z: z \hookrightarrow U_1$, on a $i_z^* T \neq 0$.
\end{lemm}

\begin{proof}
On reprend le diagramme prÈcÈdent avec $A_1=i_{1,*} \lexp {p+} j_{1,!*} \LC_1[d_1]$ soit
$$0 \rightarrow T[-1] \longrightarrow i_{1,*} \lexp {p+} j_{1,!*} \LC_1[d_1] 
\longrightarrow A'_1 \rightarrow 0,$$
de sorte que, par dÈfinition de la $t$-structure $p+$, $T$ ne peut pas Ítre ‡ support dans
$\overline U_1 \backslash U_1$, et nÈcessairement le conoyau $i_1^* j_1^* T$ de
$\LC_1 \hookrightarrow i_1^* j_1^* A'_1$  est non nul, d'o˘ le rÈsultat.
\end{proof}

\subsection{Faisceaux de cohomologie des $p$-faisceaux pervers d'Harris-Taylor}

La problÈmatique dÈcrite au paragraphe prÈcÈdent est en gÈnÈral difficile ‡ contrÙler. 
Dans notre situation, les arguments reposent
\begin{itemize}
\item d'une part sur l'utilisation des foncteurs $j^{=h}_{\neq c,!*} j^{=h,*}_{\neq c}$, dont le lecteur
pourra trouver au \S \ref{para-comple} les effets sur les $\overline \Qm_l$ faisceaux pervers 
d'Harris-Taylor,

\item et d'autre part sur la thÈorie des reprÈsentations du groupe mirabolique et de leurs dÈrivÈes
d'aprËs \cite{zelevinski1} et \cite{vigneras-livre} pour son adaptation aux corps finis.
\end{itemize}

\begin{prop} \label{prop-extind}
Soit $Q$ un $\overline \Zm_l$-faisceau pervers libre qui est $P_{h,d-h}(F_v)$-Èquivariant et tel que
\begin{itemize}
\item $Q \otimes_{\overline \Zm_l} \overline \Qm_l \simeq \lexp p j^{=h+g}_{\overline{1_h},!*} 
HT_{\overline{1_h}}(\pi_v,\Pi_h \otimes \pi_v)$ avec $\pi_v$ une reprÈsentation irrÈductible
cuspidale de $GL_g(F_v)$ et $\Pi_h$ une reprÈsentation quelconque de $GL_h(F_v)$;

\item $Q$ s'Ècrit comme une extension
$0 \rightarrow Q_{\neq c} \longrightarrow Q \longrightarrow Q_c \rightarrow 0$ o˘
$X^{\geq h+1}_{I,\bar s,c}$ est une strate pure contenue dans $X^{\geq h}_{I,\bar s,\overline{1_h}}$,
avec $Q_{\neq c} \simeq \lexp {p(c)} j^{=h+g}_{\neq c,!*} j^{=h+g,*}_{\neq c} Q$ et $Q_c$ libre.
\end{itemize}
Alors la suite exacte ci-dessus est scindÈe et $Q \simeq \lexp p j^{=h+g}_{\overline{1_h},!*} 
j^{=h+g,*}_{\overline{1_h}} Q$.
\end{prop}
%
%\rem notons que dËs que l'on sait que la suite exacte de l'ÈnoncÈ est scindÈe, alors nÈcessairement,
%du fait de l'Èquivariance, on en dÈduit que $Q \simeq \lexp p j^{=h+g}_{\overline{1_h},!*} 
%j^{=h+g,*}_{\overline{1_h}} Q$.

\begin{proof}
Supposons dans un premier temps que la suite exacte est scindÈe et considÈrons une strate pure 
$X^{=h+g}_{\IC,\bar s,a}$ telle que $c \not \subset a$. Notons $Q_a$ le quotient de $Q$ tel que
$\lexp {p} j^{=h+g}_{a,!*} j^{=h+g,*}_a Q \htarrow_+ Q_a$. La suite exacte de l'ÈnoncÈ Ètant scindÈe
on en dÈduit que $Q_a \simeq \lexp {p(c)} j^{=h+g}_{a,!*} j^{=h+g,*}_a Q$. Par Èquivariance,
l'isomorphisme prÈcÈdent est valable pour tout $c \not \subset a$ de sorte que
$Q_a \simeq \lexp {p} j^{=h+g}_{a,!*} j^{=h+g,*}_a Q.$
Soit alors $\AC$ maximal tel qu'il existe un Èpimorphisme
$$Q \twoheadrightarrow \bigoplus_{a \in \AC} \lexp {p} j^{=h+g}_{a,!*} j^{=h+g,*}_a Q=:Q_\AC,$$
et notons $Q^\AC$ son noyau. Supposons par l'absurde qu'il existe $a' \not \in \AC$ avec
$j^{=h+g,*}_{a'} Q \neq 0$. On a alors
$$\xymatrix{
Q^\AC \ar@{^{(}->}[r] \ar[dr] & Q \ar@{->>}[r] \ar@{->>}[d] & Q_\AC \\
& Q_{a'}.
}$$
Comme $Q^\AC \rightarrow \lexp {p} j^{=h+g}_{a',!*} j^{=h+g,*}_{a'} Q$ est non nulle et que son conoyau
est ‡ support dans $X^{\geq h+g+1}_{\IC,\bar s}$, cette application est nÈcessairement surjective
ce qui contredirait la maximalitÈ de $\AC$.

Montrons ‡ prÈsent que la suite exacte de l'ÈnoncÈ est scindÈe. On raisonne par l'absurde en supposant
que le conoyau $T$ de $\lexp p j^{=h+g}_{\overline{1_h},!*} j^{=h+g,*}_{\overline{1_h}} Q \htarrow Q$,
est non nul et donc, par hypothËse, ‡ support dans $X^{\geq h+g}_{I,\bar s,c}$. Soit
$T[l]$ sa $l$-torsion qui est donc un $\overline \Fm_l$-faisceau pervers $P_{h,d-h}(F_v)$-Èquivariant, 
et on choisit $T_0 \hookrightarrow T[l]$ de sorte que $T_0$ soit $P_{h,d-h}(F_v)$-Èquivariant et simple. 

Notons que $T \hookrightarrow \overline T$ o˘ $\bar T$ est le conoyau de
$\lexp p j^{=h+g}_{\overline{1_h},!*} j^{=h+g,*}_{\overline{1_h}} Q \htarrow 
\lexp {p+} j^{=h+g}_{\overline{1_h},!*} j^{=h+g,*}_{\overline{1_h}} Q$ 
que l'on peut aussi Ècrire sous la forme
\addtocounter{smfthm}{1}
\begin{equation} \label{eq-sec-induites}
0 \rightarrow \overline T_{\neq c} \longrightarrow \overline T \longrightarrow 
\overline T_c \rightarrow 0
\end{equation}
o˘ $\overline T_{\neq c}$ (resp. $\overline T_c$) est le conoyau de 
$$\lexp p j^{=h+g}_{\neq c,!*} j^{=h+g,*}_{\neq c} Q 
\htarrow \lexp {p+} j^{=h+g}_{\neq c,!*} j^{=h+g,*}_{\neq c} Q,$$
(resp. de $\lexp p j^{=h+g}_{c,!*} j^{=h+g,*}_{c} Q  \htarrow 
\lexp {p+} j^{=h+g}_{c,!*} j^{=h+g,*}_{c} Q$).
ConcrËtement, supposons pour simplifier les notations que $c=\overline{1_{h+1}}$, et soit
$$\lexp p j^{=h+g}_{\overline{1_{h+g}},!*} j^{=h+g,*}_{\overline{1_{h+g}}} Q \htarrow_+
\lexp {p+} j^{=h+g}_{\overline{1_{h+g}},!*} j^{=h+g,*}_{\overline{1_{h+g}}} Q \twoheadrightarrow
\overline T_1.$$
Alors $\overline T_1[l]$, muni de son action de $P_{h,g,d-h-g}(F_v)$ 
admet une filtration avec pour graduÈs 
$$\ind_{P_{h,g,\delta_k,d-h-g-\delta_k}(F_v)}^{P_{h,g,d-h-g}(F_v)} \lexp p j_{k,!*} \LC_k[d_k]$$ 
o˘ $\LC_k$ est un $\overline \Fm_l$-systËme local sur 
$U_k \subset X^{=h+g+\delta_k}_{\IC,\bar s,\overline{1_{h+g+\delta_k}}}$ 
avec $j_k:U_k \hookrightarrow \overline U_k \hookrightarrow X_{\IC,\bar s}^{\geq 1}$, 
$d_k:=\dim U_k \leq d-h-g-\delta_k$. Pour $z_k$ un point gÈnÈrique de $U_k$, 
en utilisant la notation \ref{nota-ind},
$\ind z_k^* \LC_k$ en tant que $\overline \Fm_l$-reprÈsentation du sous-groupe 
de Levi de $P_{h,g,\delta_k}(F_v)$, s'Ècrit 
sous la forme $r_l(\Pi_h) \otimes r_l(\pi_v) \{ r_k \} \otimes \bar \pi_k$.

En ce qui concerne $\overline T[l]$, il suffit alors d'induire de $P_{h,g,d-h-g}(F_v)$ ‡
$P_{h,d-h}(F_v)$. ConsidÈrons ainsi un point gÈnÈrique $z$ 
d'un des $U_k \hookrightarrow X^{=h+g+\delta}_{\IC,\bar s,\overline{1_{h+g+\delta}}}$, avec 
$\delta=\delta_k$ et tels que $d_k=\dim U_k$ est maximal.
Regardons alors simplement l'action infinitÈsimale de $P_{h,g+\delta}(F_v)$ sur la fibre en $z$, ou plutÙt avec la notation \ref{nota-ind}, sur $\ind z^*$.

La filtration de $\overline T_1[l]$ ci-avant fournit alors
une filtration de $\ind z^* \overline T[l]$ dont les graduÈs en tant que $\overline \Fm_l$-reprÈsentation
du sous-groupe de Levi de $P_{h,g+\delta_k}(F_v)$, s'Ècrivent sous la forme 
$r_l(\Pi_h) \otimes \bigl ( r_l(\pi_v) \{ r_k \} \times \bar \pi_k)$.
La suite exacte courte (\ref{eq-sec-induites}) fournit alors
sur chacun de ces graduÈs, une suite exacte courte $M_{1g+\delta}(F_v)$-Èquivariante
$$0 \rightarrow r_l(\pi_v\{ r_k \} ) \times (\bar \pi_i)_{|M_{\delta}(F_v)} \longrightarrow
\bigl ( r_l(\pi_v\{ k \} ) \times \bar \pi_i \bigr )_{|M_{g+\delta}(F_v)} \longrightarrow 
r_l(\pi_v\{ r_k \} )_{|M_{g}(F_v)} \times \bar \pi_i  \rightarrow 0$$
o˘ la premiËre induite ‡ gauche (resp. celle de droite) est l'induite de
$\left ( \begin{array}{ccc} 1 & 0 & * \\ 0 & GL_g & * \\ 0 & 0 & * \end{array} \right )$
(resp. de $M_{g,\delta}(F_v)$ dÈfini comme le sous-groupe de $P_{1,g-1,\delta}(F_v)$ dont le
coefficient en haut ‡ gauche est Ègal ‡ $1$) ‡ $M_{g+\delta}(F_v)$.
En particulier $T_0 \hookrightarrow \overline T[l]$ se factorise par un de ces graduÈs avec
$\lexp p h^0 \ind z^* T_0 \hookrightarrow r_l(\Pi_h) \otimes \bigl ( r_l(\pi_v\{r_k \})_{|M_{g}(F_v)} 
\times \bar \pi_k)$.

ConsidÈrons alors le tirÈ en arriËre $Q_0$
$$\xymatrix{
\lexp p j^{=h+g}_{\overline{1_h},!*} j^{=h+g,*}_{\overline{1_h}} Q \ar@{^{(}->}[r] & Q \ar@{->>}[r] & T \\
\lexp p j^{=h+g}_{\overline{1_h},!*} j^{=h+g,*}_{\overline{1_h} }Q_0 \ar@{=}[u] \ar@{^{(}->}[r]
& Q_0 \ar@{^{(}-->}[u]
\ar@{-->>}[r] & T_0 \ar@{^{(}->}[u]
}$$
Par composition des monomorphismes stricts, on a encore
\addtocounter{smfthm}{1}
\begin{equation} \label{eq-sec-induites0}
0 \rightarrow \lexp {p(c)} j^{=h+g}_{\neq c,!*} j^{=h+g,*}_{\neq c} Q_0 \longrightarrow Q_0
\longrightarrow Q_{0,c} \rightarrow 0
\end{equation}
avec $\lexp p h^0 i_z^* Q_0=\lexp p h^0 i_z^* Q_{0,c}=\lexp p h^0 i_z^* T_0$.
En outre $z$ Ètant choisi dans $X^{\geq h+1}_{I,\bar s,c}$, on Ècrit
$\lexp p h^0 \ind z^* T_0$ sous la forme
$r_l(\Pi_h) \otimes \theta$ o˘ $\theta$ est une reprÈsentation irrÈductible de $GL_{g+\delta}(F_v)$
$$\xymatrix{
& \theta \ar@{^{(}->}[d] \ar@{^{(}->}[dr] \\
r_l(\pi_v\{ r_k \} ) \times (\bar \pi_k)_{|M_{\delta}(F_v)} \ar@{^{(}->}[r] &
\bigl ( r_l(\pi_v\{ r_k \} ) \times \bar \pi_k \bigr )_{|M_{g+\delta}(F_v)} \ar@{->>}[r] & 
r_l(\pi_v\{ r_k \})_{|M_{g}(F_v)} \times \bar \pi_k.
}$$
On utilise ‡ prÈsent la thÈorie des reprÈsentations du groupe $M_{g+\delta}(F_v)$ de \cite{zelevinski1}
dans le cas complexe et \cite{vigneras-livre} chapitre III \S 1 sur $\overline \Fm_l$.
En particulier d'aprËs \cite{vigneras-livre} III 1.10, pour $\bar \pi_k$ irrÈductible, 
$r_l(\pi_v\{ r_k \} )_{|M_{g}(F_v)} \times \bar \pi_k$
admet une unique dÈrivÈe d'ordre $>0$ non nulle de sorte que d'aprËs \cite{vigneras-livre} III.1.5,
c'est une reprÈsentation irrÈductible de $M_{g+\delta}(F_v)$. Ainsi on a
%\begin{itemize}
%\item[(i)] d'une part 
$$\theta \simeq r_l(\pi_v\{ r_k \} )_{|M_{g}(F_v)} \times \bar \pi_k$$
et $Q_{0,c} \simeq \ind_{M_{g,\delta}(F_v)}^{M_{g+\delta}(F_v)} Q_{\overline{1_{h+g}}}$
o˘ $\lexp p j^{=h+g}_{\overline{1_{h+g}},!*} j^{=h+g,*}_{\overline{1_{h+g}}} Q_0 \htarrow_+
Q_{\overline{1_{h+g}}}$ de conoyau $T_{0,\overline{1_{h+g}}}$ non nul.
%\item[(ii)] D'autre part, si on veut que $r_l(\pi_v \{ r_k \} ) \times \bar \pi_k$ ne soit pas une reprÈsentation 
%irrÈductible, auquel cas $\theta$ ne pourrait pas Ítre stable sous l'action de $GL_{g+\delta}(F_v)$,
%le support cuspidal de $\theta$ doit Ítre un segment de Zelevinsky, nÈcessairement associÈ 
%‡ $r_l(\pi_v\{ r_k \})$. En particulier toutes les dÈrivÈes d'ordre $0<k < g$ de $\bar \pi_k$ sont nulles.
%\end{itemize}
On utilise alors que l'extension (\ref{eq-sec-induites0}) est scindÈe sur $\overline \Qm_l$ pour
Ècrire $Q_0$ sous la forme
$$0 \rightarrow \widetilde Q_{0,c} \longrightarrow Q_0 \longrightarrow
Q_{0,\neq c}  \rightarrow 0,$$
avec $\lexp p j^{=h+g}_{c,!*} j^{=h+g,*}_c Q_0 \htarrow_+ \widetilde Q_{0,c}
\htarrow_+ Q_{0,c}$. Sachant que le conoyau $T_0$ de
$\lexp p j^{=h+g}_{c,!*} j^{=h+g,*}_c Q_0 \htarrow_+ Q_{0,c}$ est simple,
on en dÈduit $\widetilde Q_{0,c} \simeq \lexp p j^{=h+g}_{c,!*} j^{=h+g,*}_c Q_0$.

%En utilisant l'Èquivariance sous $P_{h,d-h}(F_v)$, on en dÈduit que
%$Q_{0,\neq c} \htarrow_+ \ind_{GL_h(F_v) \times M_{1,g,\delta}(F_v)}^{P_{h,1,g-1,\delta}(F_v)}
%Q_{0,a}$ o˘ $a$ correspond au bloc $GL_g$ dans $M_{1,g,\delta}$ et le conoyau de
%$\lexp p j^{=h+g}_{a,!*} j^{=h+g,*}_a Q_0 \htarrow_+ Q_{0,a}$ est isomorphe ‡ $T_0$.
D'aprËs (\ref{eq-prop-extension}), si l'extension  (\ref{eq-sec-induites0}) n'est pas scindÈe,
on a une flËche $P_{h,1,g+\delta-1,d-h-g-\delta}(F_v)$-Èquivariante non nulle
$T_0 \longrightarrow T_{\neq c}$ o˘ $T_{\neq c}$ est le conoyau de 
$\lexp {p(c)} j^{=h+g}_{\neq c,!*} j^{=h+g,*}_{\neq c} Q_0 \htarrow_+ Q_{0,\neq c}$, i.e. une injection
\addtocounter{smfthm}{1}
\begin{equation} \label{eq-casi0}
r_l(\pi_v\{ r_k \})_{|M_{g}(F_v)} \times \bar \pi_k \hookrightarrow 
r_l(\pi_v\{ r_k \}) \times (\bar \pi)_{|M_{\delta}(F_v)},
\end{equation}
o˘ $\bar \pi$ est une reprÈsentation irrÈductible de $GL_\delta(F_v)$ qu'il
est inutile de prÈciser plus.
%le support cuspidal de $\bar{\pi_k}$ doit\footnote{On peut sans difficultÈ montrer que $\bar \pi$ devrait Ítre isomorphe ‡ $\bar \pi_k$.} Ítre dans la droite de Zelevinsky associÈe ‡ $r_l(\pi_v \{ r_k \} )$.

On reprend alors la preuve du thÈorËme 4.11 de \cite{zelevinski1} p458, cf. \cite{vigneras-livre}
pour la justification que les arguments de loc. cit. sont valables sur $\overline \Fm_l$. On notera aussi que notre mirabolique n'est pas celui de loc. cit., on
passe de nos conventions ‡ celle de loc. cit. en tordant les actions par
$g \mapsto \sigma (\lexp t g^{-1}) \sigma$ o˘ $\sigma$ est la
matrice de permutation associÈ au $n$-cycle $(12\cdots n)$, cf.
\cite{boyer-FT} preuve du lemme 4.4.

On introduit, d'aprËs \cite{zelevinski1} 3.2 p450, les foncteurs
$$\Psi^-: \alg(M_n(F_v)) \longrightarrow \alg(GL_{n-1}(F_v), \qquad 
\Phi^-: \alg (M_n(F_v)) \longrightarrow \alg (M_{n-1}(F_v))$$
dÈfinis par $\Psi^-=r_{V_n,1}$ (resp. $\Phi^-=r_{V_n,\theta}$) le foncteur des $V_{n}$
coinvariants (resp. $(V_{n},\psi)$-coinvariants), cf. \cite{zelevinski1} 1.8, o˘
$$V_n(F_v)=\{ (m_{i,j} \in P_n(F_v):~m_{i,j}= \delta_{i,j} \hbox{ for } j < n \}$$ 
dÈsigne le radical unipotent de $M_n(F_v)$ et  o˘ $\psi$ est un caractËre 
additif non trivial de $F_v$ qu'on prolonge sur le radical unipotent de $M_{n}(F_v)$.
%
%$$\Phi^-:=r_{V,\psi}: \alg (M_{n}(F_v)) \longrightarrow \alg (M_{n-1}(F_v))$$ 
%qui ‡ une reprÈsentation algÈbrique $V$ de $M_{n}(F_v)$ 
%associe la reprÈsentation algÈbrique de $M_{n-1}(F_v)$
%sur $V/E(V,\theta)$ o˘ $E(V,\theta)$ est le sous-espace de $V$ engendrÈ par $\pi(u) \zeta - \psi(u) \zeta$
%o˘ $u$ (resp. $\zeta$) dÈcrit $M_{n}(F_v)$ (resp. $V$), cf. loc. cit. p444, et o˘ $\psi$ est un caractËre 
%additif non trivial de $F_v$ qu'on prolonge sur le radical unipotent de $M_{n}(F_v)$.
%

\begin{prop} ( cf. \cite{zelevinski1} 4.13)
Pour $\rho \in \alg (GL_r(K))$, $\sigma \in \alg(GL_t(K))$ et 
$\tau \in \alg(M_s(K))$,
on a les propriÈtÈs suivantes.
\begin{itemize}
\item[(a)] Dans $\alg(M_{r+t}(K)$, on a la suite exacte courte dÈj‡ rencontrÈe
$$0 \rightarrow \rho \times (\sigma_{|M_t(K)}) \longrightarrow 
(\rho \times \sigma)_{|M_{r+t}(K)} \longrightarrow 
 (\rho_{|M_r(K)}) \times \sigma
\rightarrow 0.$$

\item[(b)] Pour $\Omega=\Psi^-$ ou $\Phi^-$,  on a
$\Omega( \tau \times \rho) \simeq  \Omega(\tau) \times \rho$.

\item[(c)] En ce qui concerne la \og grosse \fg{} induite, on a
$\Psi^-(\rho \times \tau) \simeq \rho \times \Psi^-(\tau)$ et
$$0 \rightarrow \rho \times \Phi^-(\tau) \longrightarrow \Phi^-(\rho \times \tau)
\longrightarrow  (\rho_{|M_r(K)}) \times \Psi^-(\tau) \rightarrow 0.$$

\item[(d)] Supposons $r>0$, alors pour tout sous-$M_{s+t}(K)$-module non nul,
$\omega \subset \rho \times \tau$, on a $\Phi^-(\omega) \neq (0)$.
\end{itemize}
\end{prop}

Rappelons en outre que pour toute reprÈsentation irrÈductible $\tau$ de
$M_n(F_v)$ alors exactement un parmi $\Phi^-(\tau)$ et $\Psi^-(\tau)$
est non nul et on calcule la dÈrivÈe $k$-Ëme de $\tau$ par la formule 
$\tau^{(k)}=\Psi^- (\Phi^-)^{k-1} (\tau)$. Rappelons que $\tau$ a une unique
dÈrivÈe non nulle qui est irrÈductible; si $k$ est l'ordre alors pour tout $i \leq k-1$,
$(\Phi^-)^i (\tau)$ est aussi irrÈductible.

\begin{lemm} Pour tout $0 \leq i \leq g-1$, on a
$$(\Phi^-)^{i} \bigl ( r_l(\pi_v\{ r_k \})_{|P_{1,g-1}(F_v)} \times \bar \pi_k \bigr ) = 
(\Phi^-)^{i} \bigl ( r_l(\pi_v\{ r_k \})_{|P_{1,g-1}(F_v)} \bigr ) \times \bar \pi_k 
\neq (0)$$
et 
$$(\Phi^-)^{i} \bigl ( r_l(\pi_v\{ r_k \})_{|P_{1,g-1}(F_v)} \bigr ) \times \bar \pi_k \bigr ) \hookrightarrow r_l(\pi_v\{ r_k \}) \times (\Phi^-)^{i} 
\bigl ( (\bar \pi)_{|P_{1,\delta-1}(F_v)} \big ).$$
\end{lemm}

\begin{proof}
La premiËre partie dÈcoule du point b) de la proposition prÈcÈdente
et du fait que $r_l(\pi_v)$ Ètant cuspidal, sa restriction ‡ $P_{1,g-1}(F_v)$
est irrÈductible, isomorphe ‡ la reprÈsentation mirabolique.

Pour ce qui concerne la deuxiËme assertion,
on raisonne par rÈcurrence sur $i$ en partant de $i=0$ qui correspond
‡ (\ref{eq-casi0}). Supposons donc le rÈsultat acquis au rang $i$ de sorte
qu'en appliquant le foncteur exact $\Phi^-$ on obtient
$$(\Phi^-)^{i+1} \bigl ( r_l(\pi_v\{ r_k \})_{|P_{1,g-1}(F_v)} \bigr ) \times \bar \pi_k \bigr ) \hookrightarrow\Phi^- \Bigl (  r_l(\pi_v\{ r_k \}) \times (\Phi^-)^{i} 
\bigl ( (\bar \pi)_{|P_{1,\delta-1}(F_v)} \big ) \Bigr ),$$
o˘ le terme de gauche est irrÈductible. 
Si la composÈe de l'injection prÈcÈdente avec la surjection du point c) de la 
proposition prÈcÈdente
$$\Phi^- \Bigl (  r_l(\pi_v\{ r_k \}) \times (\Phi^-)^{i} 
\bigl ( (\bar \pi)_{|P_{1,\delta-1}(F_v)} \big ) \Bigr ) \twoheadrightarrow 
 r_l(\pi_v\{ r_k \})_{P_{1,g-1}(F_v)} \times \Psi^- \Bigl ( (\Phi^-)^{i} 
\bigl ( (\bar \pi)_{|P_{1,\delta-1}(F_v)} \big ) \Bigr )$$
Ètait non nulle alors alors
$$(\Phi^-)^{i+1} \bigl ( r_l(\pi_v\{ r_k \})_{|P_{1,g-1}(F_v)} \bigr ) \times \bar \pi_k \bigr ) \hookrightarrow  
r_l(\pi_v\{ r_k \})_{P_{1,g-1}(F_v)} \times \Psi^- \Bigl ( (\Phi^-)^{i} 
\bigl ( (\bar \pi)_{|P_{1,\delta-1}(F_v)} \big ) \Bigr )$$
et donc par exactitude de $\Phi^-$ et $\Psi^-$, le foncteur
$\Psi^- \circ (\Phi^-)^{g-i-2} $ appliquÈ au membre de droite serait non nul, i.e.
$$\Psi^- \circ (\Phi^-)^{g-i-2} \Bigr ( r_l(\pi_v\{ r_k \})_{P_{1,g-1}(F_v)} \Bigr )
\times \Psi^- \Bigl ( (\Phi^-)^{i} \bigl ( (\bar \pi)_{|P_{1,\delta-1}(F_v)} \big ) \Bigr )
\neq (0),$$
alors que $\Psi^- \circ (\Phi^-)^{g-i-2} \Bigr ( r_l(\pi_v\{ r_k \})_{P_{1,g-1}(F_v)} 
\Bigr )=(0)$, d'o˘ la contradiction.

\end{proof}

Ainsi %d'aprËs \cite{zelevinski1} \S 3.5, cf. aussi \cite{vigneras-livre} III.1.3, on a
$(\Phi^-)^{g-1} \bigl ( r_l(\pi_v\{ r_k \})_{|P_{1,g-1}(F_v)} \times \bar \pi_k \bigr ) \neq (0)$ et
%et $(\Phi^-)^{g} \bigl ( r_l(\pi_v\{ r_k \})_{|P_{1,g-1}(F_v)} \times \bar \pi_k \bigr )=(0)$. De l'exactitude de $\Phi^-$, on en dÈduit alors que
$$
(\Phi^-)^{g-1} \bigl ( r_l(\pi_v\{ r_k \})_{|P_{1,g-1}(F_v)} \times \bar \pi_k \bigr ) \hookrightarrow 
r_l(\pi_v\{ r_k \}) \times (\Phi^-)^{g-1} \bigl ( (\bar \pi_k)_{|P_{1,\delta-1}(F_v)} \big ),
$$
D'aprËs le point (d) de la proposition prÈcÈdente
%4.13 de \cite{zelevinski1},  
on a alors
$\Phi^- \bigl ( (\Phi^-)^{g-1} \bigl ( r_l(\pi_v\{ r_k \})_{|P_{1,g-1}(F_v)} \times \bar \pi_k \bigr )=
(\Phi^-)^g \bigl ( r_l(\pi_v\{ r_k \})_{|P_{1,g-1}(F_v)} \times \bar \pi_k \bigr ) \neq 0$, ce qui n'est pas,
d'o˘ la contradiction et donc la suite de l'ÈnoncÈ est bien scindÈe.
\end{proof}

Comme notÈ plus haut, du lemme  \ref{lem-trivial} et de  la proposition prÈcÈdente
on en dÈduit l'isomorphisme  (\ref{eq-iso-p}) et donc aussi
la proposition \ref{prop-fc-fpht}, ‡ savoir que les faisceaux
de cohomologie des $p$-extensions intermÈdiaires des systËmes locaux d'Harris-Taylor, sont
sans torsion. Afin d'Ètudier les faisceaux de cohomologie de $\Psi_\IC$, cf. le \S \ref{para-hyp2}, 
nous utiliserons l'ÈnoncÈ suivant qui n'est qu'une version lÈgËrement modifiÈe de la proposition 
prÈcÈdente. 

\begin{prop} \label{prop-extind2}
Soit $X^{\geq 1}_{\IC,\bar s,c}$ une strate pure et $Q_{\neq c}$
un $\overline \Zm_l$-faisceau pervers libre et $P_c(F_v)$-Èquivariant tel que $Q_{\neq c} 
\otimes_{\overline \Zm_l} \overline \Qm_l \simeq \lexp p j^{=tg}_{\neq c,!*} HT_{\neq c}
(\pi_v,\tau)$ pour $\pi_v$ une reprÈsentation irrÈductible cuspidale de $GL_g(F_v)$ et o˘
$\tau$ est la reprÈsentation mirabolique de $P_{1,tg-1}(F_v)$, i.e. celle dont la seule dÈrivÈe
non nulle est celle d'ordre $tg$. On considËre alors une extension de faisceaux
pervers libre $P_{1,d-1}(F_v)$-Èquivariants:
$$0 \rightarrow Q_{\neq c} \longrightarrow X \longrightarrow P_c \rightarrow 0$$
telle que
\begin{itemize}
\item l'extension est scindÈe aprËs extension des scalaires ‡ $\overline \Qm_l$;

\item avec les notations prÈcÈdentes, 
$Q_{\neq c} \simeq \lexp {p(c)} j^{=tg}_{\neq c,!*} j^{=tg,*}_{\neq c} Q_{\neq c}$;

\item $P_c$ est supportÈ sur une strate $X^{\geq h}_{\IC,\bar s,c}$ avec $h\leq  tg$.
\end{itemize}
Alors l'extension ci-avant est scindÈe.
\end{prop}

\begin{proof}
%La preuve est strictement identique ‡ celle de la proposition \ref{prop-extind}.
On raisonne comme dans la preuve de la proposition prÈcÈdente en partant de
$$0 \rightarrow P'_c \longrightarrow X \longrightarrow Q'_{\neq c} \rightarrow 0$$
avec $Q_{\neq c} \htarrow_+ Q'_{\neq c}$. Le conoyau $T'_{\neq c}$ de 
$\lexp p j^{=tg}_{\neq c,!*} j^{=tg,*}_{\neq c} Q_{\neq c} \htarrow_+ Q'_{\neq c}$,
d'aprËs (\ref{eq-prop-extension}), s'identifie avec le conoyau $T_c$ 
de $P'_c \htarrow P_c$.
On dÈcrit alors $T_c[l]$ ‡ l'aide d'une filtration comme ci-avant
o˘ les graduÈs sont de la forme
$i_* \lexp pj_{!*} \LC_{\overline \Fm_l} \otimes ( \bar \pi \times \sigma_{P_{1,h-1}(F_v)})$
et on observe que $ \bar \pi \times \sigma_{P_{1,h-1}(F_v)}$ est de dÈrivÈe
infÈrieure ou Ègale ‡ $h\leq tg$, alors que la dÈrivÈe de tout sous-espace de 
$\bar \pi_{|P_{1,\delta}(F_v)} \times \tau$ est d'ordre strictement supÈrieur ‡ celui de $\tau$,
i.e. $tg$ de sorte que $T_c[l]$ est nÈcessairement trivial, i.e. la suite exacte courte de l'ÈnoncÈ
est scindÈe.
%
%On note comme ci-avant $\Fil^k(T'_{\neq c}[l])$ la filtration de $T'_{\neq c}[l]$ hÈritÈe
%de celle de $\overline T_{\neq c}[l]$. On s'intÈresse alors aux graduÈs
%$\gr^k(\overline T_{\neq c}[l]) \simeq i_* \lexp p j_{!*} \LC_{\overline \Fm_l} \otimes 
%\bigl ( r_l(\st_t(\pi_v) \times  \bar \pi_{|P_{1,h-tg-1}(F_v)} \bigr )$ o˘
%$\bar \pi_{|P_{1,h-tg-1}(F_v)}$ est de dÈrivÈe $\geq tg$ auquel cas, comme par hypothËse
%$\gr^k(T_{\neq c}[l])=0$, alors 
%$\gr^k(T_{c}[l]) \hookrightarrow \gr^k(\overline T_{\neq c}[l])$. Par ailleurs, 
%un sous-faisceau pervers simple de $\gr^k(T_{c}[l])$ est de la forme 
%$i_* \lexp p j_{!*} \LC_{\overline \Fm_l} \otimes 
%\bigl (\bar \pi' \times \bar \tau_{|P_{1,h-1}(F_v)} \bigr )$
%avec donc $\bar \pi' \times \bar \tau_{|P_{1,h-1}(F_v)}$ de dÈrivÈ non nulle d'ordre
%$r< tg$. On a alors 
%$$(0) \neq (\Phi^-)^{r-1} \bigl (\bar \pi' \times \bar \tau_{|P_{1,h-1}(F_v)} \bigr )
%\subset r_l(\st_t(\pi_v)) \times (\Phi^-)^{r-1} (\bar \pi_{|P_{1,h-tg-1}(F_v)}),$$
%o˘ on utilise donc que $\bar \pi_{|P_{1,h-tg-1}(F_v)}$ est de dÈrivÈe $\geq tg > r$.
%Ainsi d'aprËs \cite{zelevinski1} proposition 4.13 (d), on a
%$(\Phi^-)^{r} \bigl (\bar \pi' \times \bar \tau_{|P_{1,h-1}(F_v)} \bigr ) \neq (0)$ ce qui n'est 
%pas.
\end{proof}

En particulier dans le cas o˘ $h=tg$, on obtient le rÈsultat suivant.

\begin{coro} \label{coro-extind2}
Soit $X^{\geq 1}_{\IC,\bar s,c}$ une strate pure et soit $Q$ un $\overline \Zm_l$-faisceau 
pervers libre, $GL_d(F_v)$-Èquivariant tel que:
\begin{itemize}
\item $Q \otimes_{\overline \Zm_l} \overline \Qm_l \simeq \lexp p j^{=tg}_{!*} 
HT(\pi_v,\st_t(\pi_v))$ pour $\pi_v$ une reprÈsentation irrÈductible
cuspidale de $GL_g(F_v)$;

\item $Q$ s'Ècrit comme une extension de faisceaux pervers libres 
$P_c(F_v)$-Èquivariants
$$0 \rightarrow Q_{\neq c} \longrightarrow Q \longrightarrow Q_c \rightarrow 0$$
avec $Q_{\neq c} \simeq \lexp {p(c)} j^{=tg}_{\neq c,!*} 
j^{=tg,*}_{\neq c} Q_{\neq c}$.
\end{itemize}
Alors la suite exacte ci-dessus est scindÈe.
\end{coro}

\subsection{Sur les extensions intermÈdiaires des systËmes locaux d'Harris-Taylor}
\label{para-ext-pp}
 
Au \S \ref{para-ext-pp-preuve}, nous montrerons le rÈsultat suivant.

\begin{prop} \label{prop-ext-pp}
Soit $\pi_{v,-1}$ une reprÈsentation irrÈductible cuspidale de $GL_g(F_v)$
dont la rÈduction modulo $l$ est supercuspidale, alors, pour tout $1 \leq t \leq \frac{d}{g}$ et pour toute 
reprÈsentation $\Pi_t$ de $GL_{tg}(F_v)$,  le bimorphisme naturel
$$\lexp p j^{=tg}_{!*} HT(\pi_{v,-1},\Pi_t) \htarrow_+
\lexp {p+} j^{=tg}_{!*} HT(\pi_{v,-1},\Pi_t),$$
est un isomorphisme.
\end{prop}

Notons $\varrho$ la rÈduction modulo $l$ d'une telle $\pi_{v,-1}$ 
qui est donc par hypothËse, supercuspidale.
Avec les notations de \ref{nota-taui}, $\pi_{v,-1} \in \scusp_{-1}(\bar \tau)=\scusp_{-1}(\varrho)$
o˘ $\bar \tau$ est la rÈduction modulo $l$ de $\pi_{v,-1}[1]_D$.
Dans ce paragraphe, nous voulons en dÈduire le calcul de la $p$-torsion
du conoyau du bimorphisme ci-avant lorsque $\pi_{v,u} \in \scusp_u(\varrho)$ pour $u \geq 0$.
Avec les notations de \ref{nota-rhoi}, cela signifie que la rÈduction modulo $l$ de $\pi_{v,u}$
est isomorphe ‡ $\rho_u$.

%
%
%\begin{prop} Avec les notations ci-avant, pour $t \geq 1$, on a
%alors l'ÈgalitÈ suivante dans le groupe de Grothendieck
%\addtocounter{smfthm}{1}
%\begin{equation} \label{eq-redmodl}
%r_l(\pi_{v,u}[t]_D)=l^u \sum_{i=0}^{m(\varrho)-1} r_l(\pi_{v,-1}[tm(\varrho)l^u] \nu^i.
%\end{equation}
%\end{prop}
%
%\rem dans le cas o˘ $\epsilon(\varrho)=1$, la formule
%(\ref{eq-redmodl}) s'Ècrit $r_l( \pi_{v,u}[t]_D)=l^{u+1} r_l(\pi_{v,-1}[t]_D)$.
%
%\begin{nota} Pour $\tau_v$ une reprÈsentation irrÈductible
%de $D_{v,h}^\times$, on note
%$$\FC_{\tau_v} \otimes \Pi_h \otimes \Xi^k$$
%le systËme local d'Harris-Taylor sur $X^{=h}_{\IC,\bar s}$ induit ‡ partir de sa restriction
%‡ $X^{=h}_{\IC,\bar s,1}$ laquelle est munie d'une action de
%$$G(\Am^{\oo,p}) \times \Qm_p^\times \times (GL_h(F_v) \times GL_{d-h}(F_v))
%\times \prod_{i=2}^r (B_{v_i}^{op})^\times \times \Zm$$
%o˘ $\Pi_h$ (resp. $\Xi$) correspond ‡ l'action infinitÈsimale de $GL_h(F_v)$ 
%(resp. de $\Zm$) o˘
%$$\Xi:\frac{1}{2}\Zm \longto \overline \Qm_l^\times$$ 
%est le caractËre dÈfini par $\Xi(\frac{1}{2})=q^{1/2}$.
%\end{nota}
%

Soit $\Fm(\bullet):= \bullet \otimes^\Lm_{\overline \Zm_l} \overline \Fm_l$,
le foncteur de rÈduction modulo $l$. Rappelons que 
ce dernier ne commute pas aux foncteurs de troncations et que d'aprËs les Èquations 
2.54-2.61 de \cite{juteau}, on a
$$\Fm \lexp p j_{!*} \rightarrow \lexp p j_{!*} \Fm \rightarrow \hi^{-1} \Fm \lexp p i_*
\lexp p \hi^0_{tors} i^* j_*[1] \leadsto$$
$$\lexp p j_{!*} \Fm \rightarrow \Fm \lexp {p+} j_{!*} \rightarrow \hi^0 \Fm \lexp p i_*
\lexp p \hi^0_{tors} i^*j_* \leadsto$$
En revanche, dans le cas o˘ $\lexp p j_!=\lexp {p+} j_!$, en utilisant 
$$\lexp p j_! \rightarrow \lexp {p+} j_! \rightarrow \lexp p i_* \hi^{-1}_{tors} i^* j_* [1] \leadsto$$
quand $ \lexp p i_* \hi^{-1} i^* j_*$ est libre, alors le triangle distinguÈ
$$\Fm  \lexp p j_! \rightarrow \lexp p j_! \Fm \rightarrow \hi^{-1} \Fm \lexp p i_* \lexp p
\hi^{-1}_{tors}i^*j_* [2] \leadsto$$
nous donne que $\Fm$ et $\lexp p j_!$ commutent.

\begin{prop} \phantomsection \label{prop-dec-pervers} 
Avec les notations prÈcÈdentes et celles de la proposition \ref{prop-defi-Vk},
dans le groupe de Grothendieck des
$\overline \Fm_l$-faisceaux pervers Èquivariants sur $X_{\IC,\bar s}$,  on a l'ÈgalitÈ
\begin{multline*}
\Fm \Bigl ( \lexp p j_{!*}^{= tg_{u}(\varrho)} HT(\pi_{v,u},\Pi_t) \Bigr )=  m(\varrho)l^{u} 
\sum_{r=0}^{s-tm(\varrho)l^u} \lexp p j_{!*}^{= tg_{u}(\varrho)+rg_{-1}(\varrho)} \\
HT \bigl (\varrho,r_l(\Pi_{t}) \overrightarrow{\times} V_{\varrho}(r+tm(\varrho)l^u ,
< \underline{\delta_u}) \bigr ) \otimes \Xi^{r\frac{g-1}{2}}.
\end{multline*}
\end{prop}

\rem dans le groupe de Grothendieck, l'induite $r_l(\Pi_t)  \overrightarrow{\times} 
V_{\varrho_{-1}}(r+tm(\varrho_{-1})l^u ,< \underline{\delta_u})$ n'intervient que par 
sa semi-simplifiÈe. Pour les mÍmes raisons, il est inutile de prÈciser les rÈseaux stables
utilisÈs pour les systËmes locaux de la formule prÈcÈdente.

\begin{proof}
Les cas $tg_u \geq d$ Ètant triviaux, on raisonne par rÈcurrence en supposant le rÈsultat
acquis pour tout $t<t' $ et on traite le cas de $t$. L'idÈe est de partir de la commutation
entre $\Fm$ et les $j^{= tg}_!$:
$$\Fm \bigl ( j_!^{= tg_u(\varrho)} HT(\pi_{v,u},\Pi_{t}) \bigr )= 
\sum_{t'=t}^{s_u} \Fm \bigl ( \lexp p j_{!*}^{=t'g_u(\varrho)} HT(\pi_{v,u},\Pi_t \overrightarrow{\times}
\st_{t'-t}(\pi_{v,u}) ) \bigr ) \otimes \Xi^{\frac{(t'-t)(g_u(\varrho)-1)}{2}}),$$
et, en posant $t(u)=tm(\varrho)l^u$
\addtocounter{smfthm}{1}
\begin{multline} \label{eq-dec-1}
\Fm \bigl ( j_!^{= t(u)g_{-1}(\varrho)} HT(\pi_{v,-1},\Pi_{t}) \bigr )= \\
\sum_{t'=t(u)}^s \Fm \bigl ( \lexp p j_{!*}^{= t'g_{-1}(\varrho)} HT(\pi_{v,-1}, \Pi_{t}) \overrightarrow{\times} 
\st_{t'-t(u)}(\pi_{-1}) \bigr ) \otimes \Xi^{\frac{(t'-t(u))(g_{-1}(\varrho)-1)}{2}}.
\end{multline}
Or d'aprËs \ref{prop-red-tau}, on a
$$\Fm j_!^{= tg_u(\varrho)} HT(\pi_{v,u},\Pi_t) =j_!^{= tg_u(\varrho)} \Fm HT(\pi_{v,u},\Pi_t) =
m(\varrho) l^u j_!^{=t(u)g_{-1}(\varrho)} \Fm HT(\pi_{v,-1},\Pi_t)$$
et d'aprËs l'hypothËse de rÈcurrence, on a
\addtocounter{smfthm}{1}
\begin{multline} \label{eq-dec-2}
\sum_{t'=t+1}^{s_u}
\Fm \Bigl ( \lexp p j_{!*}^{= t'g_u(\varrho)} HT(\pi_{v,u},\Pi_t \overrightarrow{\times} \st_{t'-t}(\pi_u)) \Bigr ) 
\otimes \Xi^{\frac{(t'-t)(g_u(\varrho)-1)}{2}} \\ 
= \sum_{t'=t(u)+1}^{s(u)} \lexp p j_{!*}^{\geq t'g_{-1}(\varrho)}
HT(\varrho,r_l(\Pi_{t}) \overrightarrow{\times} V_{\varrho_{-1}}(t'-t(u),\geq \underline{\delta_u}) ) 
\otimes \Xi^{\frac{(t'-t(u))(g_{-1}(\varrho)-1)}{2}}
\end{multline}
En soustrayant (\ref{eq-dec-2}) ‡ (\ref{eq-dec-1}), on obtient le rÈsultat.
\end{proof}

\rem on notera en particulier que si $\Pi_t$ est un sous-quotient irrÈductible de $\st_t(\pi_{v,u})$ alors 
pour tout sous-quotient irrÈductible $\sigma$ de $V_{\varrho_{-1}}(s(u)-t(u),\geq \underline{\delta_u})$,
l'induite $\Pi_t \times \sigma$ est irrÈductible.

\rem la surjection
$$\lexp p j_{!}^{= t'g_u(\varrho)} ( \Fm HT(\pi_{v,u},\Pi_t{'}))  \twoheadrightarrow
\Fm ( \lexp p j_{!*}^{= t'g_u(\varrho)} HT(\pi_{v,u},\Pi_{t'}))$$
nous donne en outre que la suite des dimensions des
graduÈs de la filtration de stratification exhaustive est strictement croissante ce qui fixe
complËtement cette filtration. En particulier, en notant
$T:=i^{tg}_* \bigl ( \lexp p \hi^0_{tor} i^{tg,*}j^{\geq tg}_*[1] \bigr )$, alors pour tout sous-faisceau pervers 
$T_0$ de la $l$-torsion $T[l]$ de $T$,
il existe un sous-quotient irrÈductible $\sigma$ de 
$V_{\varrho_{-1}}(s(u)-t(u),\geq \underline{\delta_u})$
tel que $T_0$ admet $\lexp p j_{!*}^{= s(u)g_{-1}(\varrho)}
HT(\varrho,r_l(\Pi_{t}) \overrightarrow{\times} \sigma) \otimes \Xi^{\frac{(s(u)-t(u))(g_{-1}(\varrho)-1)}{2}}$
comme sous-objet.

\rem la $l$-torsion du quotient des $p+$ faisceaux pervers d'Harris-Taylor par
leur $p$ version, est complËtement dÈcrit par la combinatoire
de la rÈduction modulo $l$ des reprÈsentations de $GL_d(F_v)$ et de $D_{v,d}^\times$.
L'Ètude de la torsion d'ordre supÈrieure dÈcoulerait selon le mÍme schÈma de dÈmonstration,
de l'Ètude de la rÈduction modulo $l^n$ des reprÈsentations irrÈductibles
de $GL_d(F_v)$ et $D_{v,d}^\times$.

\section{Cycles Èvanescents}
\label{para-ce}

Pour tout $I \in \IC$, le faisceaux pervers des cycles Èvanescents
$R\Psi_{\eta_v,I}(\Lambda[d-1])(\frac{d-1}{2})$ sur $X_{I,\bar s}$ sera notÈ $\Psi_{I,\Lambda}$. Soit alors
$\Psi_{\IC,\Lambda}$ le faisceau pervers de Hecke associÈ sur $X_{\IC,\bar s}$.
Pour $\Lambda=\overline \Zm_l$, on notera plus simplement $\Psi_I$
et $\Psi_\IC$. Avec ces dÈcalages, $\Psi_{\IC,\Lambda}$ est autodual pour
la dualitÈ de Grothendieck-Verdier.

\rem soit, cf. \cite{h-t} III.2, $\LC_\xi$ le systËme local  
attachÈ ‡ une reprÈsentation irrÈductible algÈbrique $\xi$ de $G$ sur $\Lambda$. Alors
$R\Psi_{\eta_v,\IC}(\LC_\xi) \simeq R \Psi_{\eta_v,\IC}(\Lambda) \otimes \LC_\xi$, i.e.
d'un point de vue faisceautique, le rÙle de $\LC_\xi$ est transparent ce qui justifie
de n'Ètudier que le cas $\xi$ trivial.

D'aprËs \cite{boyer-invent2} thÈorËme 2.2.4, cf. aussi la proposition \ref{prop-psi-fil},
les graduÈs de la filtration par les poids de $\Psi_{\IC, \overline \Qm_l}$ sont les
$\PC(t,\pi_v)(\frac{1-t+2k}{2})$ o˘
\begin{itemize}
\item $\pi_v$ dÈcrit les classes d'Èquivalence inertielle des reprÈsentations irrÈductibles cuspidales
de $GL_g(F_v)$ pour $g$ variant de $1$ ‡ $d$,

\item $t$ varie de $1$ ‡ $\lfloor \frac{d}{g} \rfloor$ et $k$ de $0$ ‡ $t-1$.
\end{itemize}

\begin{defi} \label{defi-p-rho}
On dira d'un $\overline \Zm_l$-faisceau pervers libre $P$ tel que
$$P \otimes_{\overline \Zm_l} \overline \Qm_l \simeq \PC(t,\pi_v)(\frac{1-t+2k}{2})$$
comme ci-avant avec $\pi_v \in \scusp_i(\varrho)$, qu'il est un faisceau pervers
d'Harris-Taylor de $\varrho$-type $i$, ou simplement de type $\varrho$ quand on ne souhaitera
pas prÈciser l'indice $i$.
\end{defi}

%
%
%Au \S \ref{para-rap-psi}, on rappelle les rÈsultats de \cite{boyer-invent2} sur la description des faisceaux
%pervers simples de $\Psi_\IC \otimes_{\overline \Zm_l} \overline \Qm_l$ en termes des faisceaux pervers
%d'Harris-Taylor $\PC(\pi_v,t)(\frac{1-t+2k}{2})$ o˘ $\pi_v$ est une reprÈsentation irrÈductible cuspidale
%de $GL_g(F_v)$ avec $1 \leq t \leq d/g$ et $0 \leq k \leq t-1$. 
%
%\begin{defi} \label{defi-p-rho}
%On note $\varrho$ la 
%$\overline \Fm_l$-reprÈsentation irrÈductible supercuspidale telle que, au sens de \ref{nota-rhoi}, $\pi_v$
%est de type $\varrho$, et on dira dans la suite 
%que $\PC(\pi_v,t)(\frac{1-t+2k}{2})$ est un faisceau pervers de type $\varrho$.
%\end{defi}

Le but de cette section est de montrer que les faisceaux de cohomologie $\hi^i \Psi_\IC$
de $\Psi_\IC$ sont
sans torsion. Pour ce faire, comme indiquÈ dans l'introduction, nous allons utiliser une filtration
de $\Psi_\IC$ dont la suite spectrale calculant les $\hi^i \Psi_\IC$ ‡ partir des faisceaux
de cohomologie de ses graduÈs, dÈgÈnËre en $E_1$.

\subsection{DÈcomposition supercuspidale}
\label{para-psi}

\`A l'aide des variÈtÈs d'Igusa de premiËre et seconde espËce, les auteurs de \cite{h-t} p136, associent
‡ toute $\Lambda$-reprÈsentation admissible $\rho_v$  de $\DC_{v,h}^\times$, 
un $\Lambda$-systËme local $\LC_{\Lambda,\overline{1_h}}(\rho_v)$ sur 
$X^{=h}_{\IC,\bar s,\overline{1_h}}$ muni 
d'une action de $G(\Am^{\oo,v}) \times P_{h,d-h}(\OC_v)$, o˘ le deuxiËme facteur agit via
la projection $P_{h,d-h}(\OC_v) \longrightarrow \Zm \times GL_{d-h}(\OC_v)$ comme dans la remarque
prÈcÈdant \ref{nota-1h}. On note alors
$$\LC_{\Lambda}(\rho_v):=\LC_{\Lambda,\overline{1_h}}(\rho_v) \times_{P_{h,d-h}(\OC_v)} 
GL_d(\OC_v)$$
sa version induite sur $X^{=h}_{\IC,\bar s}$, cf. les notations du \S \ref{app-B1}.
Pour $\rho_v$ une reprÈsentation de $D_{v,h}^\times$, on notera $\LC_{\Lambda}(\rho_v)$
pour $\LC_{\Lambda}(\rho_{v,|\DC_{v,h}^\times})$. 

\rem l'action du facteur $GL_h(F_v)$ du sous-groupe de Levi $P_{h,d-h}(F_v)$
est dite \og infinitÈsimale \fg. Pour toute strate pure $X^{=h}_{\IC,\bar s,a}$ et pour
$\LC_{\Lambda,a}(\rho_v)$ le systËme local associÈ, on a aussi une action dite infinitÈsimale
du facteur $GL_h(F_v)$ du sous-groupe de Levi de $P_a(F_v)=aP_{h,d-h}(F_v)a^{-1}$, cf.
le deuxiËme tiret de \ref{nota-va}.

Le dÈcoupage (\ref{eq-decoupageQl}) de $\Psi_{\IC,\overline \Qm_l}$ selon les classes d'Èquivalence 
inertielles des reprÈsentations irrÈductibles cuspidales $\pi_v$ de $GL_g(F_v)$ pour $g$ variant de 
$1$ ‡ $d$, n'est plus valable sur $\overline \Zm_l$. L'idÈe est alors d'utiliser 
la proposition \ref{prop-scindage}. Ainsi ‡ la dÈcomposition
$\rho_v \simeq \bigoplus_{\bar \tau \in \RC_{\bar \Fm_l}(h)} \rho_{v,\bar \tau}$
d'une $\overline \Qm_l$-reprÈsentation entiËre de $D_{v,h}^\times$, selon ses $\bar \tau$-composantes,
on associe
$$\LC_{\overline \Zm_l}(\rho_v) \simeq \bigoplus_{\bar \tau \in \RC_{\bar \Fm_l}(h)} 
\LC_{\overline \Zm_l}(\rho_{v,\bar \tau}).$$

\begin{prop} \label{prop-fbartau} \phantomsection
(cf. \cite{h-t} proposition IV.2.2 et le \S 2.4 de \cite{boyer-invent2}) \\
On a un isomorphisme\footnote{Noter le dÈcalage $[d-1]$ dans la dÈfinition de 
$\Psi_{\IC,\overline \Zm_l}$.} 
$G(\Am^{\oo,v}) \times P_{h,d-h}(F_v) \times W_v$-Èquivariant
$$\ind_{(D_{v,h}^\times)^0 \varpi_v^\Zm}^{D_{v,h}^\times} 
\bigl ( \hi^{h-d-i} \Psi_{\IC,\overline \Zm_l} \bigr )_{|X^{=h}_{\IC,\bar s,\overline{1_h}}} \simeq
\bigoplus_{\bar \tau \in \RC_{ \overline \Fm_l}(h)} \LC_{\overline \Zm_l,\overline{1_h}}
(\UC^{h-1-i}_{\bar \tau,\Nm}),$$
o˘ $\UC^\bullet_{\bar \tau,\Nm}$ est dÈfini en \ref{nota-Utau} et la correspondance entre le
systËme indexÈ par $\IC$ et $\Nm$ est donnÈs par l'application $m_1$ de \ref{nota-m1}.
\end{prop}

\begin{nota} \label{nota-L5}
Pour $\bar \tau \in \RC_{ \overline \Fm_l}(h)$,
on notera $\LC_{\overline \Zm_l,\overline{1_h}}(\bar \tau)$ pour 
$\LC_{\overline \Zm_l,\overline{1_h}} (\UC^{h-1}_{\bar \tau,\Nm,free})$
et $\LC_{\overline \Zm_l}(\bar \tau)$ pour la version induite. 
\end{nota}

\rem on retrouve ces systËmes locaux dans les graduÈs $\gr^h_!(\Psi_\IC)$ de la filtration de stratification
$$j^{=h,*} \gr^h_!(\Psi_{\IC}) \simeq \bigoplus_{\bar \tau \in \RC_{\overline \Fm_l}(h)} 
\LC_{\overline \Zm_l}(\bar \tau).$$

\begin{prop}
On a une dÈcomposition
$$\Psi_\IC \simeq \bigoplus_{g=1}^d \bigoplus_{\varrho \in \scusp_{\overline \Fm_l}(g)} \Psi_{\varrho}$$
o˘ les graduÈs $\gr^h_!(\Psi_{\varrho})$
de la filtration de stratification de $\Psi_{\varrho}$ vÈrifient, cf. la notation \ref{nota-type},
$$j^{=h,*} \gr^h_!(\Psi_{\varrho}) \simeq 
\left \{ \begin{array}{ll} 0 & \hbox{si } g \nmid h \\
\LC_{\overline \Zm_l}(\varrho[t]_D) & \hbox{pour } h=tg, \end{array} \right.$$
et o˘ les constituants irrÈductibles de $\Psi_\varrho \otimes_{\overline \Zm_l} \overline \Qm_l$
sont exactement ceux de $\Psi_\IC  \otimes_{\overline \Zm_l} \overline \Qm_l$ qui sont de type
$\varrho$ au sens de la dÈfinition \ref{defi-p-rho}.
\end{prop}

\begin{proof}
Raisonnons par rÈcurrence sur la filtration de stratification de $\Psi_\IC$
$$0=\Fil^0_!(\Psi_{\IC}) \subset \Fil^1_!(\Psi_{\IC}) \subset \cdots \subset
\Fil^d_{!}(\Psi_{\IC})=\Psi_{\IC}$$
en supposant l'existence d'une dÈcomposition 
$$\Fill^r_!(\Psi_\IC)= \bigoplus_{g=1}^d \bigoplus_{\varrho \in \scusp_{\overline \Fm_l}(g)}  \Fil^r_{!,\varrho}
(\Psi_\IC).$$
Le cas de $r=0$ Ètant clair, supposons le rÈsultat acquis
pour $r-1$ et montrons le pour $r$. De la dÈcomposition 
$j^{= r,*} \gr^r_!(\Psi_\IC) \simeq  \bigoplus_{g|r=tg}
\bigoplus_{\varrho \in \scusp_{F_v}(g)} \LC_{\overline \Zm_l} (\varrho[t]_D)$,
on obtient
$$\gr^r_!(\Psi_\IC) \simeq \bigoplus_{g|r}\bigoplus_{\varrho \in \scusp_{\overline \Fm_l}(g)} 
\gr^r_{!,\varrho} (\Psi_\IC)$$
avec $j^{= r}_! \LC_{\overline \Zm_l}(\varrho[t]_D)[d-r] \twoheadrightarrow \gr^r_{!,\varrho} (\Psi_\IC)$
et o˘ tous les constituants simples de 
$\gr^r_{!,\varrho} (\Psi_\IC)\otimes_{\overline \Zm_l} \overline \Qm_l$ sont de type $\varrho$.

ConsidÈrons alors un diagramme comme (\ref{eq-prop-extension}) o˘ $A_1$ (resp. $A_2$)
est un faisceau pervers d'Harris-Taylor de type $\varrho_1$ (resp. $\varrho_2$) o˘ on suppose
que $\varrho_1$ et $\varrho_2$ ne sont pas dans la mÍme droite de Zelevinsky. L'action
de $W_v$ sur $T[l]$ vu comme quotient de $A'_1$ (resp. de $A_2$) est alors isotypique
relativement ‡ la reprÈsentation galoisienne associÈe ‡ $\varrho_1$ (resp. $\varrho_2$)
par la correspondance de Langlands-Vigneras, d'o˘ la contradiction. Ainsi $\gr^r_{!,\varrho_2}(\Psi_\IC)$
est en somme directe avec $\Fil^{r-1}_{!,\varrho_1}(\Psi_\IC)$ ce qui donne la propriÈtÈ au rang $r$
en faisant varier $\varrho_1$ et $\varrho_2$.
\end{proof}

\rem on peut donc ainsi traiter sÈparÈment chacun des $\Psi_\varrho$, ce qui nous amËne ‡ fixer
pour la suite une telle $\overline \Fm_l$-reprÈsentation irrÈductible supercuspidale $\varrho$.

\marque \textit{Notations}: dans la fin de cette section, nous allons introduire quelques notations utiles
dans les paragraphes suivants.
Avec les notations de \ref{nota-D0}, rappelons que
$g_{-1}(\varrho) ~| ~ g_0(\varrho) ~|~ \cdots | ~g_{s(\varrho)}$. On note alors
$i_{\varrho}(h)$ le plus grand entier $i \geq -1$ tel que $g_i(\varrho)$ divise $h=g_i(\varrho)t_i(\varrho,h)$.
De la formule (\ref{eq-psi-dh}), on a alors
\addtocounter{smfthm}{1}
\begin{equation} \label{eq-Ql}
j^{=h,*} \gr^h_!(\Psi_\varrho) \otimes_{\overline \Zm_l} \overline \Qm_l  \simeq
 \bigoplus_{i=-1}^{i_{\varrho}(h)} \bigoplus_{\pi_v \in \scusp_i(\varrho)}
 j^{=h,*} \PC (t_i(\varrho,h),\pi_v)(\frac{1-t_i(\varrho,h)}{2}).
\end{equation}
Notons alors, pour $k=-i_\varrho(h),\cdots, 1$,
$$\Fil^k_\varrho \bigl ( j^{=h,*} \gr^h_!(\Psi_\varrho) \otimes_{\overline \Zm_l} \overline \Qm_l \bigr ):=
\bigoplus_{i=-i_\varrho(h)}^{k}  
\bigoplus_{\pi_v \in \scusp_i(\varrho)} j^{=h,*} \PC
 (t_i(\varrho,h),\pi_v)(\frac{1-t_i(\varrho,h)}{2})$$
et soit 
$$\xymatrix{
\Fil_\varrho^k \bigl ( j^{=h,*} \gr^h_!(\Psi_\varrho) \bigr ) \ar@{^{(}-->}[r] \ar@{^{(}-->}[d] &
\Fil_\varrho^k \bigl ( j^{=h,*} \gr^h_!(\Psi_\varrho) \otimes_{\overline \Zm_l} \overline \Qm_l \bigr ) 
\ar@{^{(}->}[d] \\
j^{=h,*} \gr^h_!(\Psi_\varrho) \ar@{^{(}->}[r] & 
j^{=h,*} \gr^h_!(\Psi_\varrho) \otimes_{\overline \Zm_l} \overline \Qm_l,
}$$
et $\gr_{\varrho}^k( j^{=h,*} \gr^h_!(\Psi_\varrho) )$ les graduÈs associÈs ‡ cette
$\varrho$-filtration naÔve.
Avec les notations de \ref{nota-grmoins}, on note
$$0 \rightarrow \gr^{h,-}_!(\Psi_\varrho) \longrightarrow \gr^h_!(\Psi_\varrho) \longrightarrow
\gr^{h,+}_!(\Psi_\varrho) \rightarrow 0,$$
tel que $\gr^{h,+}_!(\Psi_\varrho) \otimes_{\overline \Zm_l} \overline \Qm_l \simeq \lexp p j^{=h}_{!*}
j^{=h,*} \gr^h_!(\Psi_\varrho \otimes_{\overline \Zm_l} \overline \Qm_l)$ et on introduit
$$\xymatrix{
& \gr^{h,+}_{!,\geq 0} (\Psi_\varrho) \ar@{^{(}->}[d] \\
j^{=h}_! j^{=h,*} \gr^h_!(\Psi_\varrho) \ar@{->>}[r] \ar@{->>}[d] & \gr^{h,+}_!(\Psi_\varrho) \ar@{-->>}[d] \\
j^{=h}_! \gr^{-1}_\varrho \bigl (j^{=h,*} \gr^h_!(\Psi_\varrho) \bigr ) \ar@{-->>}[r] &
\gr^{h,+}_{!,-1} (\Psi_\varrho).
}$$
Enfin on peut filtrer chacun des $\gr_{\varrho}^k( j^{=h,*} \gr^h_!(\Psi_\varrho) )$ de sorte que
les graduÈs soient des systËmes locaux
d'Harris-Taylor, i.e. des rÈseaux stables de chacun des 
$j^{=h,*} \PC(t_i(\varrho,h),\pi_v)(\frac{1-t_i(\varrho,h)}{2})$,
lesquels dÈpendent, ‡ priori, de tous les choix non naturels faits pour construire cette filtration. 
Ainsi pour $k=-1$, on obtient alors une filtration de $\gr^{h,+}_{!,-1} (\Psi_\varrho)$ dont les graduÈs
seront notÈs $\gr^{h,+}_{!,\pi_v}(\Psi_\varrho)$, pour $\pi_v \in \scusp_{-1}(\varrho)$.

On introduira aussi, pour $i \geq 0$ et $\pi_v \in \scusp_i(\varrho)$, les notations 
$\gr^{h,+}_{!,i}(\Psi_\varrho)$ et $\gr^{h,+}_{!,\pi_v}(\Psi_\varrho)$, qui dÈpendent de
tous les choix faits.

%
%
%Revenant ‡ notre motivation principale qui est de montrer que les fibres des faisceaux de cohomologie
%de $\Psi_{\IC,\overline \Zm_l}$ sont sans torsion, tout ce qui nous importe est de construire une filtration
%de $\Psi_{\IC,\overline \Zm_l}$ de sorte que 
%\begin{itemize}
%\item d'une part les fibres des faisceaux de cohomologie des graduÈs soient sans torsion et
%
%\item que d'autre part la suite spectrale calculant les fibres des faisceaux de cohomologie de 
%$\Psi_{\IC,\overline \Qm_l}$ ‡ partir de celles de ses graduÈs, dÈgÈnËre en $E_1$ sur $\overline \Qm_l$.
%\end{itemize}
%De ce point de vue, les rÈseaux des systËmes locaux d'Harris-Taylor construits par tous les choix
%nÈcessaires pour se ramener ‡ des $HT(\pi_v,\Pi_t)$ dans toute la suite du texte, n'interviennent pas.
%

\subsection{Preuve de la proposition \ref{prop-ext-pp}}
\label{para-ext-pp-preuve}

Notons
$\bar j: X_{\IC,\bar \eta} \hookrightarrow X_{\IC} \hookleftarrow X_{\IC,\bar s}: \bar i,$
et considÈrons la $t$-structure $p$ sur 
$\overline X_{\IC}:=X_{\IC} \times_{\spec \OC_v} \spec \overline \OC_v$ obtenue en recollant 
$$\Bigl ( \lexp p D^{\leq -1}(X_{\IC,\overline \eta},\overline \Zm_l), 
\lexp p D^{\geq -1}(X_{\IC,\overline \eta},\overline \Zm_l) \Bigr ) \quad \hbox{et} \quad
\Bigl ( \lexp p D^{\leq 0}(X_{\IC,\overline s},\overline \Zm_l), 
\lexp p D^{\geq 0}(X_{\IC,\overline s},\overline \Zm_l) \Bigr ).$$
Les foncteurs $\bar j_!$ et $\bar j_*=\lexp p {\bar j_{!*}}$ sont alors $t$-exacts avec
$$0 \rightarrow \Psi_{\IC} \longrightarrow \bar j_! \overline \Zm_l[d-1](\frac{d-1}{2}) \longrightarrow
\bar j_* \overline \Zm_l[d-1](\frac{d-1}{2}) \rightarrow 0$$
o˘ $\Psi_{\IC}=\lexp p \hi^{-1} \bar i^* \bar j^* \overline \Zm_l[d-1](\frac{d-1}{2})$.

\begin{lemm} Pour $\Lambda= \overline \Zm_l$, 
$\Psi_{\IC, \overline \Zm_l}$ est un objet de $\FC(X_{\IC,\bar s}, \overline \Zm_l)$.
\end{lemm}

\begin{proof}
D'aprËs \cite{ast} proposition 4.4.2,
$\Psi_{\IC, \overline \Zm_l}$ est un objet de $\lexp p \DC^{\leq 0}(X_{\IC,\bar s},\overline \Zm_l)$. D'aprËs
\cite{ill} variante 4.4 du thÈorËme 4.2, on a $D \Psi_{\IC, \overline \Zm_l} \simeq \Psi_{\IC, \overline \Zm_l}$
de sorte que
$$\Psi_{\IC, \overline \Zm_l} \in \lexp p \DC^{\leq 0}(X_{\IC,\bar s},\overline \Zm_l) \cap 
\lexp {p+} \DC^{\geq 0}(X_{\IC,\bar s},\overline \Zm_l)=\FC(X_{\IC,\bar s}, \overline \Zm_l).$$
\end{proof}

Le morphisme
$\bar j_{\neq c}:\overline X_{\IC} \setminus X^{\geq 1}_{\IC,\bar s,c} \hookrightarrow \overline X_{\IC}$
Ètant affine, on peut reprendre la preuve du lemme \ref{lem-HT1} et conclure ‡ la libertÈ de
$\lexp p \hi^0 i^{1,*}_{c} \bigl ( \Psi_{\IC} \bigr )\simeq 
\bigoplus_{\varrho \in \scusp_{\overline \Fm_l}(g)} 
\lexp p \hi^0 i^{1,*}_{c} \bigl ( \Psi_{\varrho} \bigr )$ avec
\addtocounter{smfthm}{1}
\begin{equation} \label{eq-sec-psi}
0 \rightarrow \bar j_{\neq c,!} \bar j^{*}_{\neq c} \Psi_{\varrho}
\longrightarrow \Psi_{\varrho} \longrightarrow i^1_{c,*} \lexp p \hi^0 i^{1,*}_{c} 
\bigl ( \Psi_{\varrho} \bigr ) \rightarrow 0.
\end{equation}
De l'exactitude de $\bar j_{\neq c,!} \bar j^{*}_{\neq c}$, les filtrations de $\Psi_\varrho$ introduites
plus avant, fournissent des filtrations de $\bar j_{\neq c,!} \bar j^{*}_{\neq c} \Psi_\varrho$.
Ainsi la filtration de stratification fournit les graduÈs $\gr^h_{!,\neq c}(\Psi_\varrho)$ avec
$$0 \rightarrow \gr^{h,-}_!(\Psi_\varrho) \longrightarrow \gr^h_{!,\neq c}(\Psi_\varrho) \longrightarrow
\gr^{h,+}_{!,\neq c}(\Psi_\varrho) \rightarrow 0,$$
o˘, cf la notation \ref{nota-jpc}, $\gr^{h,+}_{!,\neq c,\pi_v}(\Psi_\varrho) \simeq
\lexp {p(c)} j^{=h}_{\neq c,!*} j^{=h,*} _{\neq c} \gr^h_!(\Psi_\varrho)$.

\begin{nota}
Partant des $\gr^{h,+}_{!,\neq c}(\Psi_\varrho)$, on introduit comme prÈcÈdemment les 
$\gr^{h,+}_{!,\neq c,-1}(\Psi_\varrho)$, 
$\gr^{h,+}_{!,\neq c,\geq 0}(\Psi_\varrho)$ et $\gr^{h,+}_{!,\neq c,\pi_v}(\Psi_\varrho)$ 
pour tout $\pi_v \in \scusp(\varrho)$.
\end{nota}

Comme rappelÈ ‡ la proposition \ref{prop-QL-Psi}, le faisceau pervers libre
$i^1_{c,*} \lexp p \hi^0 i^{1,*}_{c} \bigl ( \Psi_{\varrho} \bigr )$ admet une filtration de 
stratification\footnote{La formulation sur $\overline \Qm_l$ est plus agrÈable car on peut y sÈparer
les contributions selon les cuspidales de $GL_g(F_v)$ pour $g$ variant de $1$ ‡ $d$.}
$$\Fil^0_c(\Psi_\varrho)=0 \subset \Fil^1_c(\Psi_\varrho) \subset \cdots \subset \Fil^d_c(\Psi_\varrho)
=i^1_{c,*} \lexp p \hi^0 i^{1,*}_{c} \bigl ( \Psi_{\varrho} \bigr ),$$ 
dont les graduÈs $\gr^h_c(\Psi_\varrho)$ vÈrifient, en supposant
pour simplifier les notations que $c=\overline{1_1}$, 
$$\gr^h_c(\Psi_\varrho) \otimes_{\overline \Zm_l} \overline \Qm_l \simeq
\ind_{P_{1,h-1,d-h}(F_v)}^{P_{1,d-1}(F_v)} \lexp p j^{=h}_{\overline{1_h},!*}
j^{=h,*}_{\overline{1_h}} \gr^h_!(\Psi_\varrho)  \otimes_{\overline \Zm_l} \overline \Qm_l.$$

\begin{nota} Comme prÈcÈdemment, on introduit les notations
$\gr^h_{c,-1}(\Psi_\varrho)$ et $\gr^h_{c,\geq 0}(\Psi_\varrho)$.
\end{nota}

Pour prouver le thÈorËme \ref{theo-glob2}, il nous suffit de montrer que la filtration de
stratification de $i^1_{c,*} \lexp p \hi^0 i^{1,*}_{c}  \Psi_{\varrho}$,
dont la $\overline \Qm_l$-version est donnÈe ‡ la proposition \ref{prop-QL-Psi}, admet pour 
graduÈs les $p$-extensions intermÈdiaires de \ref{prop-QL-Psi}.
En effet soit $z$ un point gÈomÈtrique
et $X^{\geq 1}_{\IC,\bar s,c}$ une strate pure contenant $z$. De la suite exacte courte 
$$0 \rightarrow \bar j_{\neq c,!} \bar j^{*}_{\neq c} \Psi_{\varrho}
\longrightarrow \Psi_{\varrho} \longrightarrow \lexp p \hi^0 i^{1,*}_{c} 
\bigl ( \Psi_{\varrho} \bigr ) \rightarrow 0,$$
on en dÈduit que le germe en $z$ de $\hi^i \Psi_{\varrho}$ est donnÈ par celui de 
$\lexp p \hi^0 i^{1,*}_{c}  \bigl ( \Psi_{\varrho} \bigr )$. Or 
\begin{itemize}
\item si $\lexp p \hi^0 i^{1,*}_{c}  \bigl ( \Psi_{\IC,\overline \Zm_l} \bigr )$ admet 
une filtration dont les 
graduÈs sont des $p$-extensions intermÈdiaires de systËmes locaux d'Harris-Taylor,

\item et comme les faisceaux de cohomologie de ceux-ci sont, d'aprËs \ref{prop-fc-fpht}, 
sans torsion, de sorte que les termes initiaux de la suite spectrale associÈe ‡ cette filtration et
calculant les germes en $z$ des faisceaux de cohomologie de 
$\lexp p \hi^0 i^{1,*}_{c}  \bigl ( \Psi_{\IC,\overline \Zm_l} \bigr )$, sont sans torsion.

\item Ainsi le rÈsultat dÈcoulerait du fait que sur $\overline \Qm_l$, cette suite spectrale 
dÈgÈnËre en $E_1$, cf. la derniËre remarque de cet article.
\end{itemize}

\begin{prop} \label{prop-psi-p}
Avec les notations prÈcÈdentes, pour tout $\pi_v \in \scusp_{-1}(\varrho)$,
la filtration de stratification de $\gr^{g_{-1}(\varrho)}_{!,\pi_v}(\Psi_\varrho)$ admet pour graduÈs les
$\lexp p j^{=tg_{-1}(\varrho)}_{!*} HT(\pi_v,\st_t(\pi_v))(\frac{1-t}{2})$.
\end{prop}

\begin{proof}
Notons que $\gr^{g_{-1}(\varrho)}_{!,\pi_v} \hookrightarrow \Psi_\varrho$, de sorte que
$$\lexp p i^1_{c,*} \hi^0 i^{1,*}_c \bigl ( \gr^{g_{-1}(\varrho)}_{!,\pi_v} \bigr ) \hookrightarrow
\lexp p i^1_{c,*} \hi^0 i^{1,*}_c \Psi_\varrho$$
est libre. Ainsi le rÈsultat pour $t>1$ (resp. $t=1$) se dÈduit de la suite exacte courte
$$0 \rightarrow \bar j_{\neq c,!} \bar j^{*}_{\neq c}  \bigl ( \gr^{g_{-1}(\varrho)}_{!,\pi_v} \bigr )
\longrightarrow  \gr^{g_{-1}(\varrho)}_{!,\pi_v} \longrightarrow \lexp p \hi^0 i^{1,*}_{c} 
\bigl ( \gr^{g_{-1}(\varrho)}_{!,\pi_v} \bigr ) \rightarrow 0,$$
et du corollaire \ref{coro-fill-j} (resp. de la proposition \ref{prop-extind}).
\end{proof}

\rem on dÈduit du rÈsultat prÈcÈdent que les graduÈs de la filtration de stratification de
$\gr^{g_{-1}(\varrho)}_! (\Psi_\varrho)$ sont les $\bigoplus_{\pi_v \in \scusp_{-1}(\varrho)}
\lexp p j^{=tg_{-1}(\varrho)}_{!*}  HT(\pi_v,\st_t(\pi_v))(\frac{1-t}{2})$.

En utilisant le fait que $\Psi_\IC$ est autodual, on en dÈduit le corollaire suivant.

\begin{coro} \label{coro-psi-p}
Pour tout $\pi_v \in \scusp_{-1}(\varrho)$, on a une surjection $\Psi_\varrho \twoheadrightarrow P(\pi_v)$,
o˘ les graduÈs de la filtration de stratification de $P(\pi_v)$ sont les
$\lexp {p+} j^{=tg_{-1}(\varrho)}_{!*} HT(\pi_v,\st_t(\pi_v))(\frac{1-t}{2})$, pour 
$t=1,\cdots, \lfloor \frac{d}{g_{-1}(\varrho)} \rfloor$.
\end{coro}

\rem on a de mÍme $\Psi_\varrho \twoheadrightarrow P_\varrho$, o˘ les graduÈs de la filtration de 
stratification de $P_\varrho$ sont les $\lexp {p+} j^{=tg_{-1}(\varrho)}_{!*} \LC$ o˘
$\LC$ est un rÈseau stable de $\bigoplus_{\pi_v \in \scusp_{-1}(\varrho)} HT(\pi_v,\st_t(\pi_v))(\frac{1-t}{2})$.

Dans la suite nous aurons besoin d'une version lÈgËrement amÈliorÈe de la proposition prÈcÈdente, et
donc par dualitÈ de son corollaire.

\begin{prop} \label{prop-psi-p2}
Soit $s_\varrho=\frac{d}{g_{-1}(\varrho)}$ que l'on suppose entier et $\geq 4$.
Il existe un quotient $\Psi_\varrho \twoheadrightarrow Q_\varrho$ o˘ les graduÈs $\gr^k_!(Q_\varrho)$
de la filtration de stratification de $Q_\varrho$ sont 
\begin{itemize}
\item nuls pour $k \neq tg_{-1}(\varrho)$ avec $1 \leq t \leq s_\varrho$, 

\item pour $k=tg_{-1}(\varrho)$ avec $1 \leq t \leq s_\varrho$ et $t \neq s_\varrho-1$,
ils sont isomorphes ‡ $\lexp {p+} j^{=tg_{-1}(\varrho)}_{!*} j^{=tg_{-1}(\varrho),*}
\gr^{tg_{-1}(\varrho)}_!(Q_\varrho)$ o˘ $j^{=tg_{-1}(\varrho),*} \gr^{tg_{-1}(\varrho)}_!(Q_\varrho)$
est un rÈseau stable de $\bigoplus_{\pi_v \in \scusp_{-1}(\varrho)} HT(\pi_v,\st_t(\pi_v))(\frac{1-t}{2})$ 
pour $t \neq s_\varrho-1$,

\item pour $k=(s_\varrho-1)g_{-1}(\varrho)$, on a
$$0 \rightarrow  \lexp {p+} j^{=s_\varrho g_{-1}(\varrho)}_{!*} \LC_1  \longrightarrow 
\gr^{s_\varrho-1}_!(Q_\varrho) \longrightarrow \lexp {p+} j^{=(s_\varrho-1) g_{-1}(\varrho)}_{!*}  \LC_2
 \rightarrow 0,$$
o˘ $\LC_1$ (resp. $\LC_2$) est un rÈseau stable de 
$\bigoplus_{\pi_v \in \scusp_{-1}(\varrho)} 
HT(\pi_v,\st_{s_\varrho} (\pi_v)) (\frac{3-s_\varrho}{2})$ 
(resp. de $\bigoplus_{\pi_v \in \scusp_{-1}(\varrho)} HT(\pi_v,\st_{s_\varrho-1}(\pi_v)) (\frac{2-s_\varrho}{2})$).
\end{itemize}
\end{prop}

\rem on va raisonner comme prÈcÈdemment par dualitÈ, de sorte qu'il
s'agit de faire \og remonter \fg{} le rÈseau de
$$\bigoplus_{\pi_v \in \scusp_{-1}(\varrho)} \lexp {p} j^{=s_\varrho g_{-1}(\varrho)}_{!*} 
HT(\pi_v,\st_{s_\varrho} (\pi_v)) (\frac{s_\varrho-3}{2})$$ 
vu comme sous-espace du conoyau de 
$\Fil^{g_{-1}(\varrho)}_!(\Psi_\varrho) \harrow \Psi_\varrho$, le long de la flËche de la figure \ref{fig-0}, 
en utilisant 
le diagramme \ref{eq-prop-extension}. Contrairement ‡ ce que l'on fera par la suite, on ne peut pas
le faire directement, on va alors utiliser le foncteur exact $j^{\geq 1}_{\neq c,!} j^{\geq 1,*}_{\neq c}$
ce qui nÈcessitera $s_\varrho \geq 4$ pour qu'on puisse en tirer une information utile.

\begin{proof}
Comme prÈcÈdemment on va utiliser la dualitÈ de Grothendieck-Verdier, et donc 
construire un monomorphisme strict $P \hookrightarrow \Psi_\varrho$ 
avec 
$$0 \rightarrow P_0 \longrightarrow P \longrightarrow P_1 \rightarrow 0$$
o˘ 
\begin{itemize}
\item les graduÈs de la filtration de stratification exhaustive de $P_1$ sont,  pour $t=1,\cdots,s_\varrho-2$, 
les $\lexp p j^{=tg_{-1}(\varrho)}_{!*}  j^{=tg_{-1}(\varrho),*} \gr^t(P_1)$ o˘ $ j^{=tg_{-1}(\varrho),*} \gr^t(P_1)$
est un rÈseau stable de $\bigoplus_{\pi_v \in \scusp_{-1}(\varrho)} HT(\pi_v,\st_t(\pi_v))(\frac{t-1}{2})$.

\item $P_0$ se dÈvisse en
$$0 \rightarrow P_{0,1} \longrightarrow P_0 \longrightarrow P_{0,2} \rightarrow 0$$
avec $P_{0,1} \simeq \lexp {p} j^{=s_\varrho g_{-1}(\varrho)}_{!*} \LC$ o˘ $\LC$ est un rÈseau
stable de $\bigoplus_{\pi_v \in \scusp_{-1}(\varrho)} HT(\pi_v,\st_{s_\varrho}(\pi_v))(\frac{s_\varrho-1}{2})$ et 
$$0 \rightarrow \lexp {p} j^{=(s_\varrho-1) g_{-1}(\varrho)}_{!*} \LC'_1
\longrightarrow P_{0,2} \longrightarrow \lexp {p} j^{=s_\varrho g_{-1}(\varrho)}_{!*} \LC'_2 \rightarrow 0,$$
o˘ $\LC'_1$ (resp. $\LC'_2$) est un rÈseau stable de
$\bigoplus_{\pi_v \in \scusp_{-1}(\varrho)} HT(\pi_v,\st_{s_\varrho-1}(\pi_v)) (\frac{s_\varrho-2}{2})$
(resp. $\bigoplus_{\pi_v \in \scusp_{-1}(\varrho)}HT(\pi_v,\st_{s_\varrho} (\pi_v)) (\frac{s_\varrho-3}{2})$).
\end{itemize}

\begin{figure}[htbp]
\begin{center}
%\documentclass[10pt]{article}
%\usepackage{pstricks-add}
%\pagestyle{empty}
%\begin{document}
\newrgbcolor{rvwvcq}{0.08235294117647059 0.396078431372549 0.7529411764705882}
\newrgbcolor{wrwrwr}{0.3803921568627451 0.3803921568627451 0.3803921568627451}
\psset{xunit=1cm,yunit=1cm,algebraic=true,dimen=middle,dotstyle=o,dotsize=5pt 0,linewidth=1.6pt,arrowsize=3pt 2,arrowinset=0.25}
\begin{pspicture*}(-5.0685714285714285,-3.0242857142857074)(9.350476190476183,7.318571428571425)
\pscircle[linewidth=2pt,linecolor=wrwrwr](-3.1,-0.2){0.5903388857258178}
\pscircle[linewidth=2pt,linecolor=wrwrwr](2.6,5.4){0.6133421349984373}
\pscircle[linewidth=2pt,linecolor=wrwrwr](-2.1028325829942993,0.7796732517950744){0.5070542263673541}
\pscircle[linewidth=2pt,linecolor=wrwrwr](1.337783927865424,4.15992806948182){0.45686014641488615}
\psline[linewidth=2pt,linecolor=wrwrwr](1.337783927865424,4.15992806948182)(3.08,2.87)
\pscircle[linewidth=2pt,linestyle=dotted,linecolor=wrwrwr](3.08,2.87){0.4386342439892262}
\psline[linewidth=2pt,linecolor=wrwrwr](-3.1,-0.2)(-2.1028325829942993,0.7796732517950744)
\psline[linewidth=2pt,linestyle=dotted,linecolor=wrwrwr](-2.1028325829942993,0.7796732517950744)(1.337783927865424,4.15992806948182)
\psline[linewidth=2pt,linecolor=wrwrwr](1.337783927865424,4.15992806948182)(2.6,5.4)
\rput[tl](-3.8876190476190486,0.48047619047619367){p}
\rput[tl](-2.8019047619047632,1.5090476190476212){p}
\rput[tl](1.0266666666666635,5.109047619047618){p}
\rput[tl](1.8076190476190441,6.2138095238095215){p}
\rput[tl](-0.6,3.5){p}
\rput{44.1}(-1.646,1.25){\psellipse[linewidth=2pt,linecolor=wrwrwr](0,0)(3.4,1.58)}
\rput[tl](-3.7733333333333343,2.9947619047619054){$P_1=$}
\rput[tl](-2.6,0.1){$g_{-1}(\varrho)$}
\rput[tl](-1.43,1.35){$2g_{-1}(\varrho)$}
\rput{81.1420}(3.08,4.48){\psellipse[linewidth=2pt,linecolor=wrwrwr](0,0)(2.44,2.31)}
\rput[tl](-0.23,5.7){$P_0=$}
\rput[tl](2.9885714285714244,4.0){$P_{0,2}$}
\rput[tl](3.2,5.35){$s_\varrho g_{-1}(\varrho)$}
\pscircle[linewidth=2pt,linecolor=wrwrwr](4.2,0.4){0.4539214372573928}
\pscircle[linewidth=2pt,linecolor=wrwrwr](0.2844698512137831,3.125093187157401){0.42722753570788424}
\psline[linewidth=2pt,linecolor=wrwrwr](0.2844698512137831,3.125093187157401)(1.337783927865424,4.15992806948182)
\psline[linewidth=2pt,linecolor=wrwrwr](4.2,0.4)(5.3,1.4)
\pscircle[linewidth=2pt,linecolor=wrwrwr](5.3,1.4){0.52}
%\rput{-27.66}(3.02,3.14){\psellipse[linewidth=2pt,linecolor=wrwrwr](0,0)(4.66,2.8)}
\psline[linewidth=2pt,linestyle=dotted,linecolor=wrwrwr](4.2,0.4)(1.68,-2.11)
\psline[linewidth=2pt,linecolor=wrwrwr]{->}(5.1,1.6)(3.2,2.7)
%\parametricplot[linewidth=2pt,linecolor=wrwrwr]{0.9283387040706426}{1.267109694711186}{1*5.441828924621192*cos(t)+0*5.441828924621192*sin(t)+1.9980952380952346|0*5.441828924621192*cos(t)+1*5.441828924621192*sin(t)+-2.433809523809519}
\psline[linewidth=3.2pt,linestyle=dotted,linecolor=wrwrwr](1.2933333333333332,-0.5480952380952349)(6.95047619047619,4.556666666666666)
\rput[tl](-1.2971428571428592,-1.7480952380952326){$Q_2 \harrow \Psi_\varrho$}
\rput{58.97}(4.979,0.95){\psellipse[linewidth=2pt,linecolor=wrwrwr](0,0)(1.8,1.11)}
\rput[tl](5.5,-0.4){$=Q_{2,1}$}
%\rput[tl](-2.35,5.8){$Q_{2}^-=$}
\psline[linewidth=2pt,linecolor=wrwrwr](-4.84,-1.1576190476190429)(-2.687619047619049,-1.97666666666666)
\psline[linewidth=2pt,linecolor=wrwrwr](-2.687619047619049,-1.97666666666666)(2.4933333333333296,-0.5861904761904708)
\psline[linewidth=2pt,linecolor=wrwrwr](2.4933333333333296,-0.5861904761904708)(7.769523809523803,-1.481428571428565)
\rput[tl](5.7,2.4){$s_\varrho g_{-1}(\varrho)$}
\rput[tl](4.721904761904757,0.8042857142857172){$(s_\varrho-1)g_{-1}(\varrho)$}
\psline[linewidth=2pt,linecolor=wrwrwr](-4.84,-1.1576190476190429)(-4.801904761904762,2.004285714285717)
\psline[linewidth=2pt,linecolor=wrwrwr](7.769523809523803,-1.481428571428565)(8.112380952380946,1.7947619047619077)
\begin{scriptsize}
\psdots[dotstyle=*,linecolor=rvwvcq](-3.1,-0.2)
\psdots[dotstyle=*,linecolor=rvwvcq](2.6,5.4)
\psdots[dotstyle=*,linecolor=rvwvcq](-2.1028325829942993,0.7796732517950744)
\psdots[dotstyle=*,linecolor=rvwvcq](1.337783927865424,4.15992806948182)
\psdots[dotstyle=*,linecolor=rvwvcq](3.08,2.87)
\psdots[dotstyle=*,linecolor=rvwvcq](4.2,0.4)
\psdots[dotstyle=*,linecolor=rvwvcq](0.2844698512137831,3.125093187157401)
\psdots[dotstyle=*,linecolor=rvwvcq](5.3,1.4)
\end{scriptsize}
\end{pspicture*}
%\end{document}
 \\
%\documentclass[10pt]{article}
%\usepackage{pstricks-add}
%\pagestyle{empty}
%\begin{document}
\newrgbcolor{rvwvcq}{0.08235294117647059 0.396078431372549 0.7529411764705882}
\newrgbcolor{wrwrwr}{0.3803921568627451 0.3803921568627451 0.3803921568627451}
\psset{xunit=1cm,yunit=1cm,algebraic=true,dimen=middle,dotstyle=o,dotsize=5pt 0,linewidth=1.6pt,arrowsize=3pt 2,arrowinset=0.25}
\begin{pspicture*}(-5.0685714285714285,-3.0242857142857074)(9.350476190476183,7.318571428571425)
\pscircle[linewidth=2pt,linecolor=wrwrwr](-3.1,-0.2){0.5903388857258178}
\pscircle[linewidth=2pt,linecolor=wrwrwr](2.6,5.4){0.6133421349984373}
\pscircle[linewidth=2pt,linecolor=wrwrwr](-2.1028325829942993,0.7796732517950744){0.5070542263673541}
\pscircle[linewidth=2pt,linecolor=wrwrwr](1.337783927865424,4.15992806948182){0.45686014641488615}
\psline[linewidth=2pt,linecolor=wrwrwr](1.337783927865424,4.15992806948182)(3.08,2.87)
\pscircle[linewidth=2pt,linestyle=dotted,linecolor=wrwrwr](3.08,2.87){0.4386342439892262}
\psline[linewidth=2pt,linecolor=wrwrwr](-3.1,-0.2)(-2.1028325829942993,0.7796732517950744)
\psline[linewidth=2pt,linestyle=dotted,linecolor=wrwrwr](-2.1028325829942993,0.7796732517950744)(1.337783927865424,4.15992806948182)
\psline[linewidth=2pt,linecolor=wrwrwr](1.337783927865424,4.15992806948182)(2.6,5.4)
%\rput[tl](-3.8876190476190486,0.48047619047619367){p}
%\rput[tl](-2.8019047619047632,1.5090476190476212){p}
%\rput[tl](1.02,5.1){p}
%\rput[tl](1.8076190476190441,6.2138095238095215){p}
%\rput[tl](1.1,3.5){p}
\rput{44.1}(-1.646,1.25){\psellipse[linewidth=2pt,linecolor=wrwrwr](-1,0)(2.4,1.58)}
\rput[tl](-3.9,2.5){$Q_2^+=$}
\rput[tl](-2.6,0.1){$g_{-1}(\varrho)$}
\rput[tl](-1.43,1.35){$2g_{-1}(\varrho)$}
%\rput{81.1420}(3.08,4.48){\psellipse[linewidth=2pt,linecolor=wrwrwr](0,0)(2.44,2.31)}
%\rput[tl](-0.23,5.7){$P_0=$}
\rput[tl](2.9885714285714244,4.0){$P_{0,2}$}
\rput[tl](3.2,5.35){$s_\varrho g_{-1}(\varrho)$}
\pscircle[linewidth=2pt,linecolor=wrwrwr](4.2,0.4){0.4539214372573928}
\pscircle[linewidth=2pt,linecolor=wrwrwr](0.2844698512137831,3.125093187157401){0.42722753570788424}
\psline[linewidth=2pt,linecolor=wrwrwr](0.2844698512137831,3.125093187157401)(1.337783927865424,4.15992806948182)
\psline[linewidth=2pt,linecolor=wrwrwr](4.2,0.4)(5.3,1.4)
\pscircle[linewidth=2pt,linecolor=wrwrwr](5.3,1.4){0.52}
\rput{-27.66}(3.02,3.14){\psellipse[linewidth=2pt,linecolor=wrwrwr](0,0)(4.66,2.8)}
%\psline[linewidth=2pt,linestyle=dotted,linecolor=wrwrwr](4.2,0.4)(1.68,-2.11)
\psline[linewidth=2pt,linecolor=wrwrwr]{->}(5.1,1.6)(3.2,2.7)
%\parametricplot[linewidth=2pt,linecolor=wrwrwr]{0.9283387040706426}{1.267109694711186}{1*5.441828924621192*cos(t)+0*5.441828924621192*sin(t)+1.9980952380952346|0*5.441828924621192*cos(t)+1*5.441828924621192*sin(t)+-2.433809523809519}
\psline[linewidth=3.2pt,linestyle=dotted,linecolor=wrwrwr](-2,3.45)(5,-2.4)
\rput[tl](-1.2971428571428592,-1.7480952380952326){$Q_2 \harrow \Psi_\varrho$}
\rput{58.97}(4.979,0.95){\psellipse[linewidth=2pt,linecolor=wrwrwr](0,0)(1.8,1.11)}
\rput[tl](5.5,-0.4){$=Q_{2,1}$}
\rput[tl](-2.35,5.8){$Q_{2}^-=$}
\psline[linewidth=2pt,linecolor=wrwrwr](-4.84,-1.1576190476190429)(-2.687619047619049,-1.97666666666666)
\psline[linewidth=2pt,linecolor=wrwrwr](-2.687619047619049,-1.97666666666666)(2.4933333333333296,-0.5861904761904708)
\psline[linewidth=2pt,linecolor=wrwrwr](2.4933333333333296,-0.5861904761904708)(7.769523809523803,-1.481428571428565)
\rput[tl](5.7,2.4){$s_\varrho g_{-1}(\varrho)$}
\rput[tl](4.721904761904757,0.8042857142857172){$(s_\varrho-1)g_{-1}(\varrho)$}
\psline[linewidth=2pt,linecolor=wrwrwr](-4.84,-1.1576190476190429)(-4.801904761904762,2.004285714285717)
\psline[linewidth=2pt,linecolor=wrwrwr](7.769523809523803,-1.481428571428565)(8.112380952380946,1.7947619047619077)
\psline[linewidth=2pt,linecolor=wrwrwr]{->}(4,.3)(2,-1.7)
\begin{scriptsize}
\psdots[dotstyle=*,linecolor=rvwvcq](-3.1,-0.2)
\psdots[dotstyle=*,linecolor=rvwvcq](2.6,5.4)
\psdots[dotstyle=*,linecolor=rvwvcq](-2.1028325829942993,0.7796732517950744)
\psdots[dotstyle=*,linecolor=rvwvcq](1.337783927865424,4.15992806948182)
\psdots[dotstyle=*,linecolor=rvwvcq](3.08,2.87)
\psdots[dotstyle=*,linecolor=rvwvcq](4.2,0.4)
\psdots[dotstyle=*,linecolor=rvwvcq](0.2844698512137831,3.125093187157401)
\psdots[dotstyle=*,linecolor=rvwvcq](5.3,1.4)
\end{scriptsize}
\end{pspicture*}
%\end{document}
\end{center}
\caption{\label{fig-0} Illustration graphique de la preuve de \ref{prop-psi-p2}}
\end{figure}

Soit tout d'abord
$\widetilde \gr^{2g_{-1}(\varrho)}_{!,-1}(\Psi_\varrho) \hookrightarrow  \gr^{2g_{-1}(\varrho)}_{!}
(\Psi_\varrho)$ le sous-espace strict isomorphe sur $\overline \Qm_l$ ‡  
$\gr^{2g_{-1}(\varrho)}_{!,-1}(\Psi_\varrho)$, i.e. par rapport ‡ la $\varrho$-filtration naÔve utilisÈe jusqu'‡
prÈsent, on commence par filtrer avec $\scusp_{-1}(\varrho)$. Soit alors
$$\xymatrix{
\Fil^{g_{-1}(\varrho)}_!(\Psi_\varrho) \ar@{^{(}->}[r] \ar@{=}[d] & Q_2 \ar@{-->>}[r] \ar@{^{(}-->}[d] &
Q_{2,1} \ar@{^{(}->}[d] \\
\Fil^{g_{-1}(\varrho)}_!(\Psi_\varrho) \ar@{^{(}->}[r] \ar@{=}[d] & Q_1 \ar@{-->>}[r] \ar@{^{(}-->}[d] &
\widetilde \gr^{2g_{-1}(\varrho)}_{!,-1}(\Psi_\varrho) \ar@{^{(}->}[d] \\
\Fil^{g_{-1}(\varrho)}_!(\Psi_\varrho) \ar@{^{(}->}[r]  & \Fil^{2g_{-1}(\varrho)}_!(\Psi_\varrho) \ar@{->>}[r]
&   \gr^{2g_{-1}(\varrho)}_{!,}(\Psi_\varrho),
}$$
o˘ $Q_{2,1} \otimes_{\overline \Zm_l} \overline \Qm_l \simeq \bigoplus_{\pi_v \in \scusp_{-1}(\varrho)}
Q_{2,1,\pi_v}$ avec
$$0 \rightarrow \PC(\pi_v,s_\varrho)(\frac{s_\varrho-3}{2}) \longrightarrow Q_{2,1,\pi_v} \longrightarrow
\PC(\pi_v,s_\varrho-1)(\frac{s_\varrho-4}{2}) \rightarrow 0.$$
On considËre alors 
$$0 \rightarrow \Fil_*^{3g_{-1}(\varrho)-d}(Q_{2}) \longrightarrow Q_{2} \longrightarrow 
\CoFil_*^{3g_{-1}(\varrho)-d}(Q_{2}) \rightarrow 0,$$
puis les filtrations de stratification exhaustives de $Q_{2}^-:= \Fil_*^{3g_{-1}(\varrho)-d}(Q_{2})$
et $Q_{2}^+:=\CoFil_*^{3g_{-1}(\varrho)-d}(Q_{2})$. On obtient ainsi une filtration
$$0=\Fil^0(Q_{2}) \hookrightarrow \Fil^1(Q_{2}) \hookrightarrow \cdots \hookrightarrow 
\Fil^{s_\varrho+2}(Q_{2})=Q_{2}$$
dont les graduÈs $\gr^k(Q_{2})$
aprËs tensorisation par $\overline \Qm_l$ s'Ècrivent comme une somme directe
sur les $\pi_v \in \scusp_{-1}(\varrho)$
\begin{itemize}
\item des $\PC(\pi_v,s_\varrho-k+1)(\frac{s_\varrho-k}{2})$ pour $k=1,2,3$;

\item $\PC(\pi_v,s_\varrho-k+4)(\frac{s_\varrho-k+1}{2})$ pour $k=4,5$;

\item $\PC(\pi_v,s_\varrho-k+3)(\frac{s_\varrho-k+2}{2})$ pour $k=6,\cdots, s_\varrho+2$.
\end{itemize}

¿ ce stade on ne peut pas affirmer que les structures entiËres correspondent ‡ des $p$-extensions
intermÈdiaires, pour ce faire nous allons appliquer le foncteur exact
$j^{\geq 1}_{\neq c,!} j^{\geq 1,*}_{\neq c}$ au monomorphisme strict $Q_{2} \harrow \Psi_\varrho$.
La filtration $\Fil^\bullet(Q_{2})$ de $Q_{2}$ fournit alors une filtration 
$$
0=\Fil_{\neq c}^0(Q_{2}) \hookrightarrow \Fil_{\neq c}^1(Q_{2}) \hookrightarrow \cdots 
\hookrightarrow \Fil_{\neq c}^{2 s_\varrho+2}(Q_{2})= j^{\geq 1}_{\neq c,!} j^{\geq 1,*}_{\neq c} Q_{2} 
\harrow Q_{2} \harrow \Psi_\varrho
$$
dont les graduÈs $\gr^k_{\neq c}(Q_{2})$ vÈrifiant les propriÈtÈs suivantes\footnote{On utilise en
particulier que les faisceaux de cohomologie des faisceaux pervers d'Harris-Taylor
sont sans torsion, cf. aussi le corollaire \ref{coro-j-c}.}:
\begin{itemize}
\item pour $k=1$ c'est un rÈseau stable de $\bigoplus_{\pi_v \in \scusp_{-1}(\varrho)} \PC(\pi_v,s_\varrho)
(\frac{s_\varrho-1}{2})$;

\item pour $k=2$, il s'Ècrit sous la forme $\lexp {p(c)} j^{=(s_\varrho-1)g_{-1}(\varrho)}_{\neq c,!*} \QC_2$
o˘ $\QC_2$ est un rÈseau stable de $\bigoplus_{\pi_v \in \scusp_{-1}(\varrho)}
HT_{\neq c}(\pi_v,\st_{s_\varrho-1}(\pi_v))(\frac{s_\varrho-2}{2})$.
On dira plus simplement que $\gr^k_{\neq c}(Q_{2,1})$ est une $p(c)$-extension intermÈdiaire
des $\PC_{\neq c}(\pi_v,s_\varrho-1)(\frac{s_\varrho-2}{2})$;

\item pour $k=3$, avec le vocabulaire du tiret prÈcÈdent en remplaÁant $p(c)$ par $p$, 
c'est une $p$-extension intermÈdiaire des $\PC_{c}(\pi_v,s_\varrho-1)(\frac{s_\varrho-2}{2})$;

\item pour $k=4$, c'est une $p(c)$-extension intermÈdiaire des 
$\PC_{\neq c}(\pi_v,s_\varrho-2)(\frac{s_\varrho-3}{2})$;

\item pour $k=5$ c'est un rÈseau stable des $\bigoplus_{\pi_v \in \scusp_{-1}(\varrho)} \PC(\pi_v,s_\varrho)
(\frac{s_\varrho-3}{2})$;

\item pour $k=6$, c'est une $p(c)$-extension intermÈdiaires des 
$\PC_{\neq c}(\pi_v,s_\varrho-1)(\frac{s_\varrho-4}{2})$;

\item pour $k=7$ et comme on a supposÈ $s_\varrho \geq 4$, c'est une $p$-extension intermÈdiaire
$\PC_{c}(\pi_v,s_\varrho-2)(\frac{s_\varrho-3}{2})$;

\item pour $k=9+2i$ (resp. $10+2i$) et $0 \leq i \leq s_\varrho-4$, 
c'est une $p(c)$-extension intermÈdiaire (resp. une $p$-extension
intermÈdiaire) des $\PC_{\neq c}(\pi_v,s_\varrho-3-i)(\frac{s_\varrho-3-i}{2})$
(resp. de $\PC_{c}(\pi_v,s_\varrho-2-i)(\frac{s_\varrho-3-i}{2})$).
\end{itemize}

\rem on dÈduit de la proposition \ref{prop-psi-p} que $\gr^2_{\neq c}(Q_{2,1})$ et $\gr^4_{\neq c}(Q_{2,1})$
sont en fait des $p$-extensions intermÈdiaires.

Il s'agit ‡ prÈsent de rÈorganiser la filtration selon les deux flËches de la figure du bas de \ref{fig-0}.
\begin{itemize}
\item[a)] On fait alors passer $\gr_{\neq c}^6(Q_{2,1})$ en quotient de 
$j^{\geq 1}_{\neq c,!}j^{\geq 1,*}_{\neq c} Q_{2,1}$; en utilisant le diagramme \ref{eq-prop-extension}
et comme les rÈseaux des systËmes locaux ne sont pas modifiÈs, lors de l'Èchange
les graduÈs $\gr^k_{\neq c}(Q_{2,1})$ qui deviennent
$\widetilde \gr^k_{\neq c}(Q_{2,1})$ pour $7 \leq k \leq 2s_\varrho+2$ restent toujours des $p$
et $p(c)$ extensions intermÈdiaires. En particulier $\widetilde \gr^k_{\neq c}(Q_{2,1})$ est toujours
une $p$-extension intermÈdiaire.

\item[b)] On Èchange ensuite l'ordre entre $\gr^4_{\neq c}(Q_{2,1})$ et $\gr^5_{\neq c}(Q_{2,1})$ ce qui 
fournit un nouveau $\widetilde \gr^4_{\neq c}(Q_{2,1})$ avec
$$0 \rightarrow \widetilde \gr^4_{\neq c}(Q_{2,1}) \longrightarrow X \longrightarrow 
\widetilde \gr^7_{\neq c}(Q_{2,1}) \rightarrow 0$$
o˘ $X$ est $GL_d(F_v)$-Èquivariant et $X \otimes_{\overline \Zm_l} \overline \Qm_l \simeq
\bigoplus_{\pi_v \in \scusp_{-1}(\varrho)} \PC(\pi_v,s_\varrho-2)(\frac{s_\varrho-3}{2})$.
Par Èquivariance et comme $\widetilde \gr^7_{\neq c}(Q_{2,1})$ est une $p$-extension intermÈdiaire
on en dÈduit qu'il en est de mÍme pour $\widetilde \gr^4_{\neq c}(Q_{2,1})$. 
\end{itemize}
¿ ce stade on a presque fini de prouver le rÈsultat annoncÈ sauf qu'‡ priori on a simplement un
dernier quotient $\gr^{s_\varrho-2}(P_1)$ vÈrifiant
$$0 \rightarrow \lexp {p(c)} j^{=g_{-1}(\varrho)}_{\neq c,!*} j^{=g_{-1}(\varrho),*}_{\neq c} P_1
\longrightarrow \gr^{s_\varrho-2}(P_1) \longrightarrow A \rightarrow 0$$
o˘ $A \otimes_{\overline \Zm_l} \overline \Qm_l \simeq \lexp p j^{=g_{-1}(\varrho)}_{c,!*}
j^{=g_{-1}(\varrho),*}_c P_1$. On conclut alors que $\gr^{s_\varrho-2}(P_1) \simeq 
 \lexp {p} j^{=g_{-1}(\varrho)}_{\neq c,!*} j^{=g_{-1}(\varrho),*}_{\neq c} P_1$ en utilisant la proposition
\ref{prop-extind2}.

\end{proof}

\emph{Revenons ‡ prÈsent ‡ la preuve de la proposition \ref{prop-ext-pp} que l'on va
dÈmontrer en raisonnant  par l'absurde}. ConsidÈrons donc un point gÈomÈtrique de dimension maximale
de sorte qu'il existe un entier $k$ tel que $\lexp p h^0 i_z^* \gr^k_c(\Psi_\varrho) \neq (0)$.
Soit alors $h_z$ maximal tel que $z$ est contenu
dans $X^{\geq h_z}_{\IC,\bar s}$ et on note 
$$i_z: \overline{\{ z \} }_{|X^{1 \leq h_z}_{\IC,\bar s}} \hookrightarrow X^{1 \leq h_z}_{\IC,\bar s}:=
X^{\geq 1}_{\IC,\bar s} \setminus X^{\geq h_z+1}_{\IC,\bar s}$$
ainsi que $j_z:X^{1 \leq h_z}_{\IC,\bar s} \setminus \overline{\{ z \} }_{|X^{1 \leq h_z}_{\IC,\bar s}}  
\hookrightarrow X^{1\leq h_z}_{\IC,\bar s}$.

\rem pour simplifier les notations, on supposera dans la suite que $z$ est un point 
supersingulier, i.e. $h_z=d$. Les modifications pour retrouver le cas gÈnÈral se font
en appliquant simplement le foncteur exact $j^{1 \leq h_z,*}_{\overline{1_{h_z}}}$ o˘
$$j^{1 \leq h}_{\overline{1_{h_z}}}:X^{1 \leq h}_{\IC,\bar s,\overline{1_{h_z}}} \hookrightarrow 
X^{\geq 1}_{\IC,\bar s},$$
‡ tous les 
faisceaux pervers intervenant et en utilisant l'Èquivariance sous $P_{h_z,d-h_z}(F_v)$ 
(resp. $P_{1,h_z-1,d-h_z}(F_v)$) en lieu et place de celle de $GL_d(F_v)$
(resp. $P_{1,d-1}(F_v)$).

On introduit les hypothËses suivantes dont le lecteur pourra trouver une illustration graphique ‡ la figure 
\ref{fig-1} et que nous allons montrer par rÈcurrence 
sur $h$ de $d-h_z$ (supposÈ Ègal ‡ $d$ d'aprËs la remarque prÈcÈdente), ‡ $1$. On notera encore
$s_\varrho:=\lfloor \frac{d}{g_{-1}(\varrho)} \rfloor$.

\begin{figure}[htbp]
\input{LT-pic2-fin.tex}
%\caption{Illustration}
%\end{figure}

%\end{document}

%\begin{figure}[ht]
\begin{center}
%\documentclass[10pt]{article}
%\usepackage[utf8]{inputenc}
%\usepackage{pstricks-add}
%\pagestyle{empty}
%\begin{document}
\psset{xunit=1cm,yunit=1cm,algebraic=true,dimen=middle,dotstyle=o,dotsize=5pt 0,linewidth=1.6pt,arrowsize=3pt 2,arrowinset=0.25}
\begin{pspicture*}(-4.7,1)(8,9)
\pscircle[linewidth=2.pt,hatchcolor=black,fillstyle=crosshatch,hatchangle=45.0,hatchsep=0.242](-3.,5.){0.4204759208325727}
\rput[tl](-4.6902,7.828){$\left \{ \begin{array}{l} p \\ p+ \end{array} \right.$}
\rput[tl](-4.0368,6.497){$\neq c$}
\rput[tl](-2.5,7.8){$g_{-1}(\varrho) | h$}
\rput[tl](-3.8,4.5){$\mathrm{gr}^{h,+}_!(\Psi_{\varrho})$}
\rput[tl](-1.8588,4.9){$=$}
\pscircle[linewidth=2.pt,hatchcolor=black,fillstyle=hlines,hatchangle=45.0,hatchsep=0.242](0.,5.){0.42}
\pscircle[linewidth=2.pt](2.,5.){0.42}
\rput[tl](0.7548,5.1176){$/$}
\rput[tl](1.4,4.5){$\mathrm{gr}^{h,+}_{!,-1}(\Psi_\varrho)$}
\rput[tl](-0.8,4.5){$\mathrm{gr}^{h,+}_{!,\geq 0}(\Psi_\varrho)$}
\pscircle[linewidth=2.pt,hatchcolor=black,fillstyle=crosshatch,hatchangle=45.0,hatchsep=0.242](-3.,2.){0.38}
\rput{40.05}(-1.77,3.08){\psellipse[linewidth=2.pt](0,0)(1.03,0.5)}
\pscircle[linewidth=2.pt,hatchcolor=black,fillstyle=crosshatch,hatchangle=45.0,hatchsep=0.242](-2.24,2.68){0.30}
\psline[linewidth=2.pt,linestyle=dashed,dash=6pt 6pt](-1.86,3.)(-1.22,3.62)
\psline[linewidth=2.pt](-3.,2.)(-2.24,2.68)
\rput[tl](-1.1328,3.0848){$\mathrm{gr}^{h,-}_!(\Psi_\varrho)$}
\rput[tl](-2.609,1.8506){$\mathrm{gr}^{h,+}_!(\Psi_\varrho)$}
\rput[tl](0.295,2.1894){$\mathbf{=}$}
\rput[tl](1.0936,2.3588){$\mathrm{gr}^h_!(\Psi_\varrho)=\mathrm{gr}^{h,-}_!(\Psi_\varrho) / \mathrm{gr}^{h,+}_!(\Psi_\varrho)$}
\rput[tl](-2.73,5.8){$h$}
\pscircle[linewidth=2.pt,hatchcolor=black,fillstyle=crosshatch,hatchangle=45.0,hatchsep=0.242](-2.9962,6.7632){0.3872}
\rput[tl](-1.8,7){$=\mathrm{gr}^{h,+}_{!,\neq c}(\Psi_\varrho)$}
\pscircle[linewidth=2.pt,hatchcolor=black,fillstyle=crosshatch,hatchangle=45.0,hatchsep=0.242](3.199,6.739){0.4362717043311427}
\rput[tl](2.9328,7.7312){$h$}
\rput[tl](2.5,6.3){$c$}
\rput[tl](4,7){$=\mathrm{gr}_c^h(\Psi_\varrho)$}
\end{pspicture*}
%\end{document}
\end{center}
\caption{\label{fig-1} Illustration graphique de $HR_\Psi(h+1)$}
\end{figure}
%\includepdf[scale=.7]{LT-pic1.pdf}

%\clearpage

\noindent - \fbox{$HR_\Psi(h+1)$}: $\Psi_{\varrho}$ admet une filtration
$$0=\Fil^0(\Psi_{\varrho},h) \subset \Fil^1(\Psi_{\varrho},h) \subset \Fil^2(\Psi_{\varrho},h)
\subset \Fil^3(\Psi_{\varrho},h) \subset \Fil^4(\Psi_{\varrho},h) \subset 
\Fil^5(\Psi_{\varrho},h)=\Psi_{\varrho}$$
vÈrifiant les propriÈtÈs suivantes:
\begin{itemize}
\item $\Fil^1(\Psi_{\varrho},h)=j^{=1}_{\neq c,!} j^{=1,*}_{\neq c} \Fil^{h-1}_!(\Psi_{\varrho})$;

\item $\gr^2(\Psi_{\varrho},h)$ admet une filtration dont les graduÈs sont\footnote{On rappelle que ces
graduÈs sont nuls si $k$ n'est pas un multiple de $g_{-1}(\varrho)$.} si $s_\varrho g_{-1}(\varrho) \neq d$, 
(resp. si $s_\varrho g_{-1}(\varrho)= d$) 
les $\gr^{k}_{!,\neq c,\geq 0}(\Psi_{\varrho})$ et $\gr^{k,-}_{!,-1}(\Psi_{\varrho})$
pour $k$ variant de $h$ ‡ $d$ (resp. on enlËve le terme $\gr^{(s_\varrho-1)g_{-1}(\varrho),-}_{!,-1}
(\Psi_\varrho)$ par rapport au cas non respÈ). On notera en particulier que la partie libre de
$\lexp p \hi^0 i_z^* \Fil^2(\Psi_\varrho,h)$ est triviale et on demande qu'il en soit de mÍme pour sa torsion;

\item $\gr^3(\Psi_{\varrho},h) \simeq \gr^h_{!,\neq c,-1}(\Psi_{\varrho})$; 
%on demande en outre qu'en tout point gÈomÈtrique de 
%$X^{\geq 1}_{\IC,\bar s,c}$, les fibres de $\Fil^3(\Psi_{\varrho},h)$ sont sans torsion.
%
%\rem comme par ailleurs 
%\begin{itemize}
%\item les faisceaux de cohomologie de $\gr^3(\Psi_{\varrho},h)$ sont sans torsion;
%
%\item pour tout $i$ la flËche $h^i \gr^3(\Psi_{\varrho},h) \longrightarrow h^{i+1}
%\Fil^2(\Psi_{\varrho},h)$ est injective sur $\overline \Qm_l$,
%\end{itemize}
%on en dÈduit de mÍme que les fibres de $\Fil^2(\Psi_{\varrho},h)$ sont sans torsion.
%

\item %$\gr^4(\Psi_{\varrho},h) \simeq \Fil^h_!(\lexp p h^0 i_c^* \Psi_{\varrho})$. On
$\Fil^4(\Psi_{\varrho},h)$ est $GL_d(F_v)$ Èquivariant 
%et on a une suite exacte courte
%$$0 \rightarrow \gr^{4}_{\geq 0}(\Psi_{\varrho},h) \longrightarrow \gr^4(\Psi_{\varrho},h) 
%\longrightarrow \gr^{4}_{-1}(\Psi_{\varrho},h) \rightarrow 0$$ 
et la filtration de stratification de $\gr^{4}_{\geq 0}(\Psi_{\varrho},h)$ 
%(resp. de $\gr^{4}_{-1}(\Psi_{\varrho},h)$) 
admet pour graduÈs successifs les $\gr^k_{c}(\Psi_\varrho)$ pour $1 \leq k \leq h$
puis les $\gr^k_{c,\geq 0}(\Psi_\varrho)$ pour $h+1 \leq k \leq h$.

\item Si $s_\varrho$ ne divise pas $h_z$ (supposÈ Ègal ‡ $d$) ou si $s_\varrho \leq 2$,
la filtration de stratification de $\gr^5(\Psi_{\varrho},h)$ admet pour graduÈs
successifs $\gr^k_!(\gr^5(\Psi_\varrho,h))=\lexp p j^{=k}_{!*} j^{=k,*} \gr^k_{!,-1}(\Psi_\varrho)$ 
pour $h+1 \leq k \leq d$. Dans les autres cas, on demande la mÍme condition sauf pour
\begin{multline*}
0 \rightarrow \lexp p j^{=d}_{!*} j^{=d,*} \gr^{d-g_{-1}(\varrho),-}_{!,-1}(\Psi_\varrho) \\ \longrightarrow
\gr^{(s_\varrho-1)g_{-1}(\varrho)}_!(\gr^5(\Psi_\varrho,h)) \longrightarrow \\ 
\lexp p j^{=d-g_{-1}(\varrho)}_{!*} j^{=d-g_{-1}(\varrho),*} \gr^{d-g_{-1}(\varrho)}_{!,-1}(\Psi_\varrho) \rightarrow 0.
\end{multline*}
\end{itemize}

\rem la modification de la filtration pour les cas o˘ $s_\varrho \geq 3$ divise $h_z=d$, 
provient du fait qu'on
cherche ‡ avoir la partie libre de $\lexp p \hi^0 i_z^* \Fil^2(\Psi_\varrho,h)$ nulle, ce qui sera utile 
pour montrer l'inductivitÈ. Cela implique en particulier que l'on doive alors utiliser la proposition
\ref{prop-psi-p2} au lieu du corollaire \ref{coro-psi-p}. Notons enfin que cette nuance n'est pas illustrÈe
dans la figure \ref{fig-1}  afin de ne pas alourdir encore le dessin.

\medskip

\noindent - \fbox{$HR_p(h)$}: pour tout $k>h$, on a
$\lexp p j^{=k}_{!*} j^{=k,*} \gr^k_{!,-1}(\Psi_\varrho)=\lexp {p+} j^{=k}_{!*} j^{=k,*} \gr^k_{!,-1}(\Psi_\varrho)$.

\rem on notera que $HR_p(1)$ implique le rÈsultat cherchÈ, i.e. la proposition \ref{prop-ext-pp}.
%\medskip

CommenÁons par noter que tout systËme local
sur $X^{=d}_{\IC,\bar s}=X^{=d}_{\IC,\bar s,c}$ est pervers et donc $HR_p(d)$
est trivialement vÈrifiÈ. En ce qui concerne $HR_\Psi(d)$, il suffit simplement de poser
$\Fil^1(\Psi_{\varrho},d)=\Fil^2(\Psi_{\varrho},d)=\Fil^3(\Psi_{\varrho},d)=j^{=1}_{\neq c,!}j^{=1,*}_{\neq c}
\Psi_\varrho$ et $\Fil^4(\Psi_\varrho,d)=\Fil^5(\Psi_\varrho,d)$.
%
%\begin{lemm}
%L'hypothËse $HR_\Psi(d)$ est vÈrifiÈe.
%\end{lemm}
%
%\begin{proof}
%Notons que
%$\Fil^1(\Psi_{\varrho},d)=\Fil^2(\Psi_{\varrho},d)=\Fil^3(\Psi_{\varrho},d)$. Le seul
%point non trivial ‡ vÈrifier est que pour tout $\pi_v \in \scusp_{-1}(\varrho)$, on a une surjection
%$\lexp p h^0 i_c^* \Psi_{\varrho} \twoheadrightarrow \Fil^4_{-1}(\Psi_{\varrho},d)$, o˘ les
%graduÈs de la filtration de stratification de $\Fil^4_{-1}(\Psi_{\varrho},d)$ sont les 
%$\bigoplus_{\pi_v \in \scusp_{-1}(\varrho)} 
%\lexp {p+} j^{=tg_{-1}(\varrho)}_{c,!*} HT_c(\pi_v,\st_t(\pi_v))(\frac{t-s}{2})$.
%ConsidÈrons $\Fil^{g_{-1}(\varrho)}_!(\Psi_{\varrho})$ obtenu en prenant l'image par
%$j^{=1}_!$ de $\bigoplus_{\pi_v \in \scusp_{-1}(\varrho)} HT(\pi_v,\pi_v) \hookrightarrow
%j^{=1,*} \Psi_{\varrho}$. D'aprËs le corollaire \ref{coro-resolution1} celui-ci admet 
%une filtration dont les graduÈs sont
%les $\bigoplus_{\pi_v \in \scusp_{-1}(\varrho)} 
%\lexp p j^{=tg_{-1}(\varrho)}_{!*} HT(\pi_v,\st_t(\pi_v))(\frac{t-1}{2})$. Le rÈsultat
%dÈcoule simplement du fait que $\Psi$ est autodual, et que la dualitÈ Èchange les 
%$\lexp p j_{!*}$ avec $\lexp {p+} j_{!*}$, et par applications rÈpÈtÈes du lemme
%\ref{lem-ext-scinde0}.
%\end{proof}
%
%\rem en suivant le mÍme raisonnement, et en fixant $\pi_v \in \scusp_{-1}(\varrho)$,
%on peut montrer que 
%$\Fil^4_{-1}(\Psi_{\varrho},d)$ admet pour quotient un faisceau pervers libre dont les
%graduÈs successifs de la filtration de stratification sont les
%$\lexp {p+} j^{=tg_{-1}(\varrho)}_{c,!*} HT_c(\pi_v,\st_t(\pi_v))(\frac{t-s}{2})$.
%

\begin{lemm} \label{lem-preuve1}
Si $HR_\Psi(h+1)$ et $HR_p(h+1)$ sont vÈrifiÈes, alors 
$HR_p(h)$ aussi.
\end{lemm}

\begin{proof}
Visuellement ‡ l'aide de la figure \ref{fig-1}, il s'agit de montrer qu'on peut faire passer le terme
$\gr^{h,+}_{!,\neq c,-1}(\Psi_\varrho)$ de l'autre cÙtÈ de la flËche notÈe $\heartsuit$, pour former
un quotient 
$\Fil^4(\Psi_\varrho,h) \twoheadrightarrow \gr^h_!(\Psi_\varrho)$ vÈrifiant
$$0 \rightarrow \lexp {p(c)} j^{=h}_{\neq c,!*} j^{=h,*}_{\neq c} \gr^h_!(\Psi_\varrho) \longrightarrow
 \gr^h_!(\Psi_\varrho) \longrightarrow \lexp {p+}  j^{=h}_{c,!*} j^{=h,*}_{c} \gr^h_!(\Psi_\varrho)
\rightarrow 0.$$
On utilise ensuite le corollaire \ref{coro-psi-p} pour conclure.

ConsidÈrons ainsi
$\Fil^4(\Psi_{\varrho},h)$ et $\pi_v \in \scusp_{-1}(\varrho)$. S'il n'existe pas d'entier $t$
tel que $tg_{-1}=h$ alors il n'y a rien ‡ dÈmontrer puisque $HR_p(h)=HR_p(h+1)$.
Notons alors $t$ tel que $tg_{-1}(\varrho)=h$ et considÈrons
$$0 \rightarrow \gr^{3,\pi_v}(\Psi_{\varrho},h) \longrightarrow \gr^3(\Psi_{\varrho},h)\longrightarrow \gr^3_{\pi_v}(\Psi_{\varrho},h) \rightarrow 0$$
o˘ $ \gr^3_{\pi_v}(\Psi_{\varrho},h)  \simeq 
\lexp {p(c)} j^{=h}_{!*,\neq c} HT(\pi_v,\st_t(\pi_v)(\frac{1-t}{2}).$
Introduisons
$$\xymatrix{
\Fil^2(\Psi_{\varrho},h) \ar@{=}[d] \ar@{^{(}->}[r] & \Fil^{3,\pi_v}(\Psi_{\varrho},h)
\ar@{-->>}[r] \ar@{^{(}-->}[d] & \gr^{3,\pi_v}(\Psi_{\varrho},h) \ar@{^{(}->}[d] \\
\Fil^2(\Psi_{\varrho},h) \ar@{^{(}->}[r] & \Fil^3(\Psi_{\varrho},h) \ar@{->>}[r] &
\gr^3(\Psi_{\varrho},h),
}$$
et on Ècrit
$$0 \rightarrow \gr^3_{\pi_v}(\Psi_{\varrho},h) \longrightarrow \Fil^4(\Psi_{\varrho},h)/\Fil^{3,\pi_v}(\Psi_{\varrho},h) \longrightarrow \gr^4(\Psi_{\varrho},h) \rightarrow 0.$$

\begin{lemm} Si $h_z \neq 3g_{-1}(\varrho)$,
on a une surjection 
$$\Fil^4(\Psi_{\varrho},h)/\Fil^{3,\pi_v}(\Psi_{\varrho},h) 
\twoheadrightarrow \lexp {p+} j^{=h}_{!*} HT(\pi_v,\st_t(\pi_v))(\frac{1-t}{2}).$$
\end{lemm}

\begin{proof}
Le cas o˘ $s_\varrho$ ne divise pas $h_z$ (supposÈ Ègal ‡ $d$), dÈcoule du corollaire
\ref{coro-psi-p}, de l'hypothËse de rÈcurrence et la forme de $\gr^5(\Psi,h)$.
Supposons donc que $s_\varrho g_{-1}(\varrho)=h_z=d$: si $s_\varrho \leq 2$, il n'y a rien ‡ montrer
et pour $s_\varrho \geq 4$, on conclut comme ci-avant en remplaÁant l'appel ‡ \ref{coro-psi-p} par
\ref{prop-psi-p2}.
\end{proof}

Notons alors $\gr^4_{\pi_v}(\Psi_{\varrho},h)$ le quotient 
$\Fil^4(\Psi_\varrho,h) \twoheadrightarrow \lexp {p+} j^{=h}_{c,!*} HT_c(\pi_v,\st_t(\pi_v))(\frac{1-t}{2})$
puis
$$0 \rightarrow \gr^{4,\pi_v}(\Psi_{\varrho},h) \longrightarrow \gr^4(\Psi_{\varrho},h)
\longrightarrow \gr^4_{\pi_v}(\Psi_{\varrho},h) \rightarrow 0.$$
Soit alors $X$ le tirÈ en arriËre
$$\xymatrix{
\lexp {p(c)} j^{=h}_{\neq c,!*} HT_{\neq c}(\pi_v,\st_t(\pi_v))(\frac{t-1}{2}) \ar@{^{(}->}[r] & 
X \ar@{^{(}-->}[d] \ar@{-->>}[r] & \gr^{4,\pi_v}(\Psi_{\varrho},h) \ar@{^{(}->}[d] \\
\gr^3_{\pi_v}(\Psi_{\varrho},h) \ar@{^{(}->}[r] \ar@{=}[u] & 
\Fil^4(\Psi_{\varrho},h)/\Fil^{3,\pi_v}(\Psi_{\varrho},h) \ar@{->>}[r] & \gr^4(\Psi_{\varrho},h). 
}$$
Ainsi en Èchangeant l'ordre des termes de l'extension dÈfinissant $X$, on obtient
$$0 \rightarrow \gr^{4,\pi_v}(\Psi_{\varrho},h)' \longrightarrow X \longrightarrow
\lexp {p+} j^{=h}_{\neq c,!*} HT_{\neq c}(\pi_v,\st_t(\pi_v))(\frac{t-1}{2}) \rightarrow 0.$$

\rem on rappelle que, aprËs application du
foncteur exact $j^{=h,*}$, les rÈseaux obtenus en partant de $\gr^3(\Psi_{\varrho},h))$
ou du quotient $\Fil^4(\Psi_{\varrho},h) \twoheadrightarrow \lexp {p+} j^{=h}_{!*} 
HT(\pi_v,\st_t(\pi_v)(\frac{t-1}{2})$, sont donnÈs par $\lexp p h^0 j^{=h,*} \Psi_{\varrho}$
de sorte que le rÈseau de $HT(\pi_v,\st_t(\pi_v))$ n'est pas modifiÈ. 

Supposons, par l'absurde que le conoyau 
$$\begin{array}{ll}
T& =\coker \bigl ( \gr^{4,\pi_v}(\Psi_{\varrho},h)'  \htarrow_+ \gr^{4,\pi_v}(\Psi_{\varrho},h) \bigr ) \\
& =\coker \bigl (
\lexp {p(c)} j^{=h}_{\neq c,!*} HT_{\neq c}(\pi_v,\st_t(\pi_v))(\frac{t-1}{2}) \htarrow_+
\lexp {p+} j^{=h}_{\neq c,!*} HT_{\neq c}(\pi_v,\st_t(\pi_v))(\frac{t-1}{2}) \bigr ),
\end{array}$$ 
est non nul.

(a) En voyant $T$ comme quotient de $\gr^{4,\pi_v}(\Psi_{\varrho},h)$ comme ci-dessus,
d'aprËs (\ref{eq-prop-extension}), $T$ admettrait une filtration
$$(0)=\Fil^0(T) \subset \Fil^1(T) \subset \cdots \subset \Fil^r(T)=T$$
dont les graduÈs $\gr^k(T)$ s'Ècriraient comme le conoyau $T_\delta$ d'un bimorphisme
$$\lexp p j^{=\delta g_{i_\delta}(\varrho)}_{c,!*} HT(\pi_{v,\delta},\st_{\delta}(\pi_{v,\delta}))(\frac{\delta-1}{2})
\htarrow_+ \lexp {p+} 
j^{=\delta g_{i_\delta}(\varrho)}_{c,!*} HT(\pi_{v,\delta},\st_{\delta}(\pi_{v,\delta}))(\frac{\delta-1}{2})
\twoheadrightarrow T_\delta,$$
associÈ ‡ un sous-quotient irrÈductible de $\gr^{4,\pi_v}(\Psi_{\varrho},h)$, o˘
$\pi_{v,\delta} \in \scusp_{i_\delta}(\varrho)$ avec lorsque $i_\delta=-1$, $\delta < t$.
Or un tel $T_\delta[l]$ peut s'Ècrire comme extensions successives
de $P_{1,d-1}(F_v)$-reprÈsentations de la forme $\bar \pi \times 
\sigma_{P_{1,\delta g_{i_\delta}(\varrho)-1}(F_v)}$,
dont la dÈrivÈe vÈrifie la propriÈtÈ suivante: pour $i_\delta=-1$, cette dÈrivÈe est d'ordre 
$< tg_{-1}(\varrho)$ et pour $i_\delta \geq 0$, cet ordre est divisible par $g_0(\varrho)$.

(b) Comme conoyau de $\lexp {p(c)} j^{=h}_{\neq c,!*} HT_{\neq c}(\pi_v,\st_t(\pi_v))(\frac{t-1}{2}) 
\htarrow_+ \lexp {p+} j^{=h}_{\neq c,!*} HT_{\neq c}(\pi_v,\st_t(\pi_v))(\frac{t-1}{2}) \bigr )$,
$T[l]$ admet un sous-espace de la forme
$\tilde \pi_{|P_{1,d-tg_{-1}-1}(F_v)} \times \sigma_0$ o˘ on peut prendre $\sigma_0$ de dÈrivÈe
d'ordre $tg_{-1}(\varrho)$ ou $(t-1)g_{-1}(\varrho)$. de sorte que 
$\tilde \pi_{|P_{1,d-tg_{-1}-1}(F_v)} \times \sigma_0$ admet une sous-reprÈsentation de dÈrivÈe
d'ordre strictement plus grand que $(t-1)g_{-1}(\varrho)$ et non divisible par $g_0(\varrho)$,
ce qui contredit la propriÈtÈ obtenue en (a).

Ainsi on a $\lexp {p+} j^{=h}_{\neq c,!*} HT_{\neq c}(\pi_v,\st_t(\pi_v))=
\lexp {p(c)} j^{=h}_{\neq c,!*} HT_{\neq c}(\pi_v,\st_t(\pi_v))$, d'o˘ $HR_p(h)$ 
en faisant varier $c$.

\medskip

Il nous reste dÈsormais ‡ traiter le cas o˘ $h_z$ (supposÈ Ègal ‡ $d$) est Ègal ‡ $3g_{-1}(\varrho)$
et $h=g_{-1}(\varrho)$, i.e. ‡ montrer que $\lexp p j^{=g_{-1}(\varrho)}_{!*} j^{=g_{-1}(\varrho),*}
\gr^{g_{-1}(\varrho)}_!(\Psi_\varrho) \simeq \lexp {p+} j^{=g_{-1}(\varrho)}_{!*}  j^{=g_{-1}(\varrho),*}
\gr^{g_{-1}(\varrho)}_!(\Psi_\varrho)$. On part de la filtration donnÈe par $HR_\Psi(2g_{-1}(\varrho)+1)$
et le fait dÈj‡ dÈmontrÈ que $\gr^{2g_{-1}(\varrho),+}_{!,-1}(\Psi_\varrho)$ passe en suivant la flËche notÈe $\heartsuit$
pour former $\gr^5(\Psi_\varrho,g_{-1}(\varrho)+1)$ avec
$$0 \rightarrow X \longrightarrow
\Fil^4(\Psi_\varrho,g_{-1}(\varrho)+1) \longrightarrow \lexp {p+} j^{=g_{-1}(\varrho)}_{c,!*}
j^{=g_{-1}(\varrho),*}_c \gr^{g_{-1}(\varrho)}_{c,-1} \rightarrow 0.$$
Il suffit alors de montrer que l'extension $P$ dÈfinie ci-dessous
$$\xymatrix{
X \ar@{^{(}->}[r] \ar@{->>}[d] & \Fil^4(\Psi_\varrho,g_{-1}(\varrho)+1) \ar@{->>}[r] \ar@{-->>}[d] &
\lexp {p+} j^{=g_{-1}(\varrho)}_{c,!*} j^{=g_{-1}(\varrho),*}_c \gr^{g_{-1}(\varrho)}_{c,-1} \ar@{=}[d] \\
X/ \Fil^1(\Psi_\varrho,g_{-1}(\varrho)+1) \ar@{^{(}-->}[r] & P \ar@{->>}[r] & 
\lexp {p+} j^{=g_{-1}(\varrho)}_{c,!*} j^{=g_{-1}(\varrho),*}_c \gr^{g_{-1}(\varrho)}_{c,-1}
}$$
est scindÈe. On raisonne alors comme prÈcÈdemment en notant que 
\begin{itemize}
\item d'un cÙtÈ tout quotient de $l$-torsion de $\lexp {p+} j^{=g_{-1}(\varrho)}_{c,!*} j^{=g_{-1}(\varrho),*}_c 
\gr^{g_{-1}(\varrho)}_{c,-1}$ ne fait intervenir que des reprÈsentations du groupe mirabolique dont la 
dÈrivÈe est d'ordre $g_{-1}(\varrho)$,

\item et que de l'autre cÙtÈ tous les quotients de torsion associÈs aux extensions intermÈdiaires
de $X/ \Fil^1(\Psi_\varrho,g_{-1}(\varrho)+1)$, ont nÈcessairement une dÈrivÈe d'ordre strictement
plus grande que $g_{-1}(\varrho)$.
\end{itemize}
En ce qui concerne ce dernier point, il est clair pour toutes les extensions intermÈdiaires associÈes
‡ des $\pi_v \in \scusp_i(\varrho)$ pour $i \geq 0$ et il ne reste alors plus qu'‡ le vÈrifier pour le
systËme local concentrÈ aux points supersinguliers $\gr^{2g_{-1}(\varrho),-}_{!,-1} (\Psi_\varrho)$.
La fibre en un point supersingulier $z$ de 
$$0 \rightarrow  \gr^{2g_{-1}(\varrho),-}_{!,-1} (\Psi_\varrho) \longrightarrow
 \gr^{2g_{-1}(\varrho)}_{!,-1} (\Psi_\varrho) \longrightarrow  \gr^{2g_{-1}(\varrho),+}_{!,-1} (\Psi_\varrho)
 \rightarrow 0$$
 fournit, d'aprËs la proposition \ref{prop-fc-fpht}, pour chaque $\pi_v \in \scusp_{-1}(\varrho)$, 
 un diagramme
 $$\xymatrix{
 & \hi^{-1} i_z^* \gr^{2g_{-1}(\varrho)}_{!,\pi_v} (\Psi_\varrho)  \ar@{^{(}->}[d] \ar[dr]^\sim \\
\st_2(\pi_v\{\frac{-1}{2} \} ) \times (\pi_v \{ 1†\})_{M} \ar@{^{(}->}[r] \ar[dr]^\sim &
\bigl ( \st_2(\pi_v\{\frac{-1}{2} \} ) \times (\pi_v \{ 1†\}) \bigr)_{M} \ar@{->>}[r] 
\ar@{->>}[d]&
 \st_2(\pi_v\{\frac{-1}{2} \} )_{M}  \times (\pi_v \{ 1†\}) \\
& \hi^0 i_z^*  \gr^{2g_{-1}(\varrho),-}_{!,\pi_v} (\Psi_\varrho) 
}$$
o˘ on a notÈ simplement $M$ pour le mirabolique associÈ et on a omis les termes galoisiens.
Ainsi le rÈseau de $\st_3(\pi_v)$ de $\hi^0 i_z^*  \gr^{2g_{-1}(\varrho),-}_{!,\pi_v} (\Psi_\varrho)$
vÈrifie la propriÈtÈ suivante d'aprËs \cite{zelevinski1} proposition 4.13 (d): toute sous-reprÈsentation
irrÈductible de sa rÈduction modulo $l$ a une dÈrivÈe d'ordre $2g_{-1}(\varrho)$ ou $3g_{-1}(\varrho)$.

\end{proof}

\begin{lemm}
Si $HR_\Psi(h+1)$ et $HR_p(h)$ sont vÈrifiÈes alors $HR_\Psi(h)$ aussi.
\end{lemm}

\begin{proof}
On filtre $\gr^3(\Psi_\varrho,h)_{-1}$ de sorte que chacun des graduÈs soit de la forme
$\lexp p j^{=tg_{-1}(\varrho)}_{!*} HT(\pi_v,\st_t(\pi_v))(\frac{1-t}{2})$ pour $\pi_v$ dÈcrivant
$\scusp_{-1}(\varrho)$. Comme dans la preuve du lemme prÈcÈdant, chacun de ces graduÈs
\og passe ‡ travers \fg{} la flËche notÈe $\heartsuit$ dans la figure (\ref{fig-1}),
ce qui permet de construire $\Fil^4(\Psi_\varrho,h-1)$ avec $\gr^5(\Psi_\varrho,h-1)$
vÈrifiant les hypothËses demandÈes.

Il s'agit alors de faire remonter $\gr^2(\Psi_\varrho,h)$ le long de la flËche notÈe
$\diamondsuit$ dans la figure \ref{fig-1}.
Soit alors
$$\xymatrix{
\Fil^1(\Psi_{\varrho},h) \ar@{^{(}->}[r] \ar@{->>}[d] & \Fil^2(\Psi_{\varrho},h) \ar@{->>}[r] 
\ar@{-->>}[d] & \gr^2(\Psi_{\varrho},h)\ar@{=}[d] \\
\gr^{h-1,+}_{!,\neq c,-1}(\Psi_{\varrho})\ar@{^{(}-->}[r] & P \ar@{->>}[r] & 
\gr^2(\Psi_{\varrho},h)
}$$
et il s'agit de montrer que
\addtocounter{smfthm}{1}
\begin{equation} \label{eq-sec-scinde}
0 \rightarrow  \gr^{h-1,+}_{!,\neq c,-1}(\Psi_{\varrho}) \longrightarrow P 
\longrightarrow \gr^2(\Psi_{\varrho},h) \rightarrow 0
\end{equation}
est scindÈe. Dans le cas contraire, on aurait
$$0 \rightarrow \widetilde \gr^2(\Psi_{\varrho},h) \longrightarrow P \longrightarrow
\widetilde \gr^{h-1,+}_{!,\neq c,-1}(\Psi_{\varrho}) \rightarrow 0$$
avec, puisque les rÈseaux ne sont pas modifiÈs,
$$\gr^{h-1,+}_{!,\neq c,-1}(\Psi_{\varrho}) \htarrow_+ \widetilde \gr^{h-1,+}_{!,\neq c,-1}
(\Psi_{\varrho}) \twoheadrightarrow T \neq 0.$$
Or rappelons que $\Fil^2(\Psi_{\varrho},h)$ a ÈtÈ construit de sorte que la partie libre de son
$\lexp p \hi^0 i_z^* \Fil^2(\Psi_\varrho,h)$ est triviale. Ainsi si $T$ Ètait non nul, on en dÈduirait que
$\lexp p \hi^0 i_z^* \Fil^2(\Psi_\varrho,h)$ serait de torsion, ce qui n'est pas par hypothËse.

%
%Pour ce faire nous allons utiliser la construction suivante appliquÈe ‡
%$\Fil^4(\Psi_{\varrho},h-1)$.
%
%\begin{lemm} \label{lem-technique1}
%Soit $Q$ un faisceau pervers libre $G$-Èquivariant sur un schÈma $X$ et soit
%$i_Z:Z \hookrightarrow X$ une immersion fermÈe avec $Z$ stable par $G$. On note
%$j_Z:X \backslash Z \hookrightarrow X$. On suppose que
%\begin{itemize}
%\item $\lexp p h^0 i_Z^* Q \otimes_{\overline \Zm_l} \overline \Qm_l$ est nul,
%
%\item et qu'on a une suite exacte courte $0 \rightarrow A \longrightarrow Q \longrightarrow B
%\rightarrow 0$,
%
%\item ainsi qu'une surjection $\lexp p j_{Z,!} j_Z^* A \twoheadrightarrow A$.
%\end{itemize}
%Alors l'image $Q^Z$ du morphisme d'adjonction $\lexp p j_{Z,!} j_Z^* Q \longrightarrow Q$
%vÈrifie que $\lexp p h^0 i_z^* Q^z=0$, s'inscrit dans un diagramme commutatif
%$$\xymatrix{
%& \lexp p h^0 i_Z^* Q \ar@{=}[r] &  \lexp p h^0 i_Z^* Q \\
%A \ar@{^{(}->}[r] \ar@{^{(}-->}[dr] \ar@{->}[ur]^0 & Q \ar@{->>}[u] \ar@{->>}[r] & 
%B \ar@{->>}[u] \\
%& Q^Z \ar@{-->>}[dr] \ar@{^{(}->}[u] \\
%& & B' \ar@{^{(}->}[uu]
%}$$
% En outre $Q^z$ est encore $G$-Èquivariant.
%\end{lemm}
%
%Soit alors $Q:=\Fil^4(\Psi_{\varrho},h-1)^{Z}$ avec $Z=X^{=d}_{\IC,\bar s}$ o˘
%$A:=\Fil^2(\Psi_{\varrho},h)$ et $B:=\widetilde \gr^4(\Psi_{\varrho},h-1)$
%$$0 \rightarrow A \longrightarrow Q^Z \longrightarrow B' \rightarrow 0,$$
%avec $\lexp p h^0 i_d^* B'=0$ et donc en particulier
%$$B' \twoheadrightarrow \lexp p j^{=t'g}_{c,*!} HT(\pi_v,\st_{t'}(\pi_v))(\frac{t'-1}{2})$$
%pour tout $\pi_v \in \scusp_{-1}(\varrho)$ et $t'$ maximal tel que $t'g_{-1}(\varrho) \leq h-1$.
\end{proof}

\subsection{Preuve du thÈorËme principal}
\label{para-hyp2}

D'aprËs le thÈorËme de comparaison de Berkovich, cf. \cite{berk2}, le thÈorËme \ref{theo-LT}
dÈcoule de l'ÈnoncÈ suivant.

\begin{theo} \label{theo-glob2}\phantomsection
Pour tout $i$, les $\hi^i\Psi_{\varrho}$ sont sans torsion.
\end{theo}

Sur $\overline \Qm_l$, cf. la proposition \ref{prop-QL-Psi}, les graduÈs de la filtration de stratification
de $i^1_{c,*} \lexp p \hi^0 i^{1,*}_{c} \bigl ( \Psi_{\IC,\overline \Zm_l} \bigr )$ sont des faisceaux pervers
d'Harris-Taylor et la suite spectrale calculant ses faisceaux de cohomologie ‡ partir de ceux de ses
graduÈs, dÈgÈnËre en $E_1$. Ainsi sur $\overline \Zm_l$, d'aprËs \ref{prop-fc-fpht}, il suffit de montrer
que ces graduÈs sont les $p$-faisceaux pervers d'Harris-Taylor.
Ainsi pour $d=2$, sachant que les seuls faisceaux pervers d'Harris-Taylor non ponctuels sont associÈs
‡ des caractËres, le rÈsultat dÈcoule trivialement du lemme \ref{lem-ext0}.

On reprend la suite exacte courte (\ref{eq-sec-psi}) et la filtration du faisceau pervers
libre $\lexp p \hi^0 i^{1,*}_{c} \bigl ( \Psi_{\varrho} \bigr )$ dont les graduÈs $\gr^k_c$
vÈrifient
$$\lexp p j^{=tg_i(\varrho)}_{c,!*} HT_c(\pi_v,\st_t(\pi_v))(\frac{1-t}{2}) \htarrow_+ \gr^k_c 
\twoheadrightarrow T^k_c$$
pour $i \geq -1$ et $\pi_v \in \scusp_i(\varrho)$. Supposons par l'absurde qu'il existe un tel graduÈ
tel que le bimorphisme ci-dessus ne soit pas un isomorphisme et notons $z$ un point gÈnÈrique
de dimension maximale tel que la fibre en $z$ d'un tel $T^k_c$ soit non nulle.\footnote{On peut
pour simplifier considÈrer comme dans le paragraphe prÈcÈdent que $z$ est un point supersingulier.}
Soit alors $i$ minimal tel la fibre en $z$ d'un $T^k_c$ soit non nulle. D'aprËs le paragraphe prÈcÈdent,
on a nÈcessairement $i \geq 0$. En outre on en dÈduit que $\lexp p h^0 i_z^* \Psi_{\bar \tau}$
admet alors un sous-quotient de la forme $\bar \pi \times \sigma_{|P_{1,tg_i(\bar \tau)}(F_v)}$ o˘ 
\begin{itemize}
\item $\sigma$ est un sous-quotient
de la rÈduction modulo $l$ de $\st_t(\pi_v)$ et 

\item $\bar \pi$ est, d'aprËs la proposition \ref{prop-dec-pervers}, un sous-quotient 
irrÈductible de $V_{\bar \rho}(r+tm(\varrho)l^i ,< \underline{\delta_i})$.

\item En outre tous les sous-quotients irrÈductibles de $\lexp p h^0 i_z^* \Psi_{\bar \tau}$, ont alors
une dÈrivÈe d'ordre $\geq g_i(\varrho)$.
\end{itemize}
Rappelons la suite exacte courte
$$0 \rightarrow \bar \pi_{|P_{1,h-1}(F_v)} \times \sigma \longrightarrow 
(\bar \pi \times \sigma)_{|P_{1,h+tg_i(\varrho)-1}(F_v)} \longrightarrow 
\bar \pi \times \sigma_{|P_{1,tg_i(\varrho)}(F_v)} \rightarrow 0,$$
o˘, puisque les supports cuspidaux sont disjoints, $\bar \pi \times \sigma$ est irrÈductible.
Ainsi comme $\lexp p h^0 i_z^* \Psi_{\varrho}$ est $GL_{h+tg_i(\varrho)}(F_v)$ Èquivariant, on 
en dÈduit qu'il admet aussi $ \bar \pi_{|P_{1,h-1}(F_v)} \times \sigma$ comme sous-quotient
et que donc il admet un sous-quotient irrÈductible ayant une dÈrivÈe d'ordre $<g_i(\bar \tau)$, d'o˘
la contradiction.

\appendix

\renewcommand{\theequation}{\Alph{section}.\arabic{subsection}.\arabic{equation}}

\section{Rappels sur les reprÈsentations}
\label{para-rep}

\subsection{de $GL_d(K)$ ‡ coefficients dans $\overline \Qm_l$}
\label{para-rappel-rep}

Notons $K$ un corps local non archimÈdien dont le corps rÈsiduel est de cardinal $q$ une puissance de $p$.
Une racine carrÈe $q^{\frac{1}{2}}$ de $q$ dans $\overline \Qm_l$ Ètant fixÈe,
pour $k \in \frac{1}{2} \Zm$, nous noterons
$\pi\{ k \}$ la reprÈsentation tordue de $\pi$ o˘ l'action de $g \in GL_n(K)$ est donnÈe par $\pi(g) \nu(g)^k$ avec
$\nu: g \in GL_n(K) \mapsto q^{-\val (\det g)}$.

\begin{defi}
Pour $V$ un sous-espace vectoriel de dimension $h$ de $K^n$, soit $P_V$
le sous-groupe parabolique associÈ et $N_V$ son radical unipotent, i.e. l'ensemble
des $g \in GL_n(K)$ tels que le noyau de $g-\Id$ contienne $V$ et son image soit contenue
dans $V$. 
\end{defi}

\begin{notas} \label{nota-va} \phantomsection
\begin{itemize}
\item Dans le cas o˘ dans la dÈfinition prÈcÈdente, $V$ est engendrÈ par les $h$ premiers
vecteurs de la base canonique, $V$ sera notÈ $\overline{1_h}$ et $P_V$ par $P_{h,d-h}$.

\item Pour tout $a \in GL_d(K) / P_{h,d-h}(K)$, on rappelle qu'on note encore 
$a$ le sous-espace vectoriel obtenu comme l'image par
$a$ de $\Vect (e_1,\cdots,e_h)$, o˘ $(e_i)_{1†\leq i \leq d}$ dÈsigne
la base canonique de $K^d$. On notera aussi $P_a(K)=aP_{h,d-h}(K)a^{-1}$ 
le parabolique associÈ et $N_a=a N_{h,d}(K) a^{-1}$
son sous-groupe unipotent. Le facteur $GL_h(K)$ de $P_a(K)$ est appelÈ \og son facteur
infinitÈsimal \fg.
%
%\item Plus gÈnÈralement pour un drapeau 
%$\Delta= \{ (0) \subsetneq a_{1} \subsetneq a_{2} \subsetneq \cdots \subsetneq a_{r} \subset K^d=:a_{r+1}\}$, 
%de graduÈs $a_i/a_{i-1}$, on note $P_\Delta(F_v)$ le sous-groupe parabolique associÈ
%et $U_\Delta(F_v)$ son radical unipotent.
%
%\item Enfin pour $\Delta= \{ (0) \subsetneq a_{1} \subsetneq a_{2} \subsetneq \cdots \subsetneq a_{r} \subset K^d\}$ 
%un drapeau et $a_r \subset a \subset K^d$, on notera 
%$$\Delta(a):= \{ (0) \subsetneq a_{1} \subsetneq a_{2} \subsetneq \cdots \subsetneq a_{r} \subset a\subset K^d\}.$$
%Pour $a_r \subset a \subset b \subset K^d$, on notera $\Delta(a \subset b):= \bigl ( \Delta(a) \bigr ) (b)$.
\end{itemize}
\end{notas}

\rem afin d'allÈger les notations, on notera $\dim a$ pour $\dim V_a$.

\begin{defi}
Soit $P=MN$ un parabolique standard de $GL_n$ de LÈvi $M$ et de radical unipotent $N$.
On note $\delta_P:P(K) \rightarrow \overline \Qm_l^\times$ l'application dÈfinie par
$\delta_P(h)=|\det (\ad(h)_{|\lie N})|^{-1}.$
Pour $(\pi_1,V_1)$ et $(\pi_2,V_2)$ des reprÈsentations de respectivement $GL_{n_1}(K)$ 
et $GL_{n_2}(K)$, et $P_{n_1,n_2}$ le parabolique standard de $GL_{n_1+n_2}$ de Levi 
$M=GL_{n_1} \times GL_{n_2}$ et de radical unipotent $N$, $\pi_1 \times \pi_2$
dÈsigne l'induite parabolique normalisÈe de $P_{n_1,n_2}(K)$ ‡ $GL_{n_1+n_2}(K)$ de 
$\pi_1 \otimes \pi_2$ c'est ‡ dire
l'espace des fonctions $f:GL_{n_1+n_2}(K) \rightarrow V_1 \otimes V_2$ telles que
$$f(nmg)=\delta_{P_{n_1,n_2}}^{-1/2}(m) (\pi_1 \otimes \pi_2)(m) \Bigl ( f(g) \Bigr ),
\quad \forall n \in N, ~\forall m \in M, ~ \forall g \in GL_{n_1+n_2}(K).$$
\end{defi}

Rappelons qu'une reprÈsentation irrÈductible $\pi$ de $GL_n(K)$ est dite \textit{cuspidale} 
(resp. \textit{supercuspidale}) si elle n'est pas isomorphe ‡
un sous-quotient (resp. ‡ un sous-espace)
d'une induite parabolique propre. D'aprËs \cite{vigneras-induced} \S V.4, 
la rÈduction modulo $l$ d'une reprÈsentation irrÈductible cuspidale est encore irrÈductible cuspidale 
mais pas nÈcessairement supercuspidale.

\begin{notas}
Soient $g$ un diviseur de $d=sg$ et $\pi$ une reprÈsentation cuspidale
irrÈductible de $GL_g(K)$. 
\begin{itemize}
\item L'unique quotient (resp. sous-reprÈsentation) irrÈductible de
$\pi\{ \frac{1-s}{2} \} \times \pi\{\frac{3-s}{2} \} \times \cdots \times \pi\{ \frac{s-1}{2} \}$
est notÈ $\st_s(\pi)$ (resp. $\speh_s(\pi)$).

\item L'unique sous-espace (resp. quotient) irrÈductible de $\st_t(\pi_v \{ \frac{-r}{2} \} )
\times \speh_r (\pi_v \{ \frac{t}{2} \} )$ (resp. de $\st_{t-1}(\pi_v \{ \frac{-r-1}{2} \} ) \times \speh_{r+1}
(\pi_v \{ \frac{t-1}{2} \} )$) est notÈ $LT_{\pi_v}(t-1,r)$.
\end{itemize}
\end{notas}

%
%\begin{defi} Soit $P=MN$ un parabolique de $GL_n$ de LÈvi $M$ avec pour radical
%unipotent $N$. Pour $\pi$ une $R$-reprÈsentation admissible de $GL_n(K)$,
%l'espace des vecteurs $N(K)$-coinvariants est stable sous l'action de
%$M(K) \simeq P(K)/N(K)$. On notera $J_P(\pi)$ cette reprÈsentation tordue par $\delta_P^{-1/2}$.
%\end{defi}
%
%\begin{lemm} \label{lem-fond-rep}
%Soient $\bar \pi$ une $\overline \Fm_l$-reprÈsentation irrÈductible de $GL_h(K)$ telle qu'il existe
%un ÈlÈment unipotent non trivial $n \in GL_h(K)$ agissant trivialement. Alors $\bar \pi$ est un caractËre.
%\end{lemm}
%
%\rem en particulier si $\bar \pi$ est obtenu comme un sous-quotient de la rÈduction modulo $l$
%d'une reprÈsentation irrÈductible entiËre $\Pi$ alors $\Pi$ a pour support cuspidal des caractËres.
%
%\begin{proof}
%L'adhÈrence de la classe de conjugaison de $n$ contient un unipotent $n_1$, agissant nÈcessairement
%trivialement, dont le sous-espace propre est de codimension $1$. Ainsi tous les unipotents dont
%le sous-espace propre est de codimension $1$ agissent trivialement et donc, par combinaison linÈaire,
%le sous-groupe unipotent du Borel $B(K)$ de $GL_h(K)$ agit lui aussi trivialement.
%Ainsi le foncteur de Jacquet associÈ ‡ $B(K)$ est trivial sur $\bar \pi$
%et donc $\bar \pi$ est une reprÈsentation irrÈductible admissible de $B(K)$ et donc du tore. Ainsi
%$\bar \pi$ est de dimension finie et donc un caractËre.
%\end{proof}

\begin{nota}
Pour $n \geq 2$, on note
$M_n(K)$ le sous-groupe de $P_{1,n-1}(K)$ dont le premier coefficient en haut ‡ gauche
est Ègal ‡ $1$. 
\end{nota}

\rem Quitte ‡ tordre les actions par $g \mapsto \sigma (\lexp t g^{-1} ) \sigma^{-1}$, 
o˘ $\sigma$ est la matrice de permutation associÈ au cycle $(12\cdots n)$,
on reconnait le traditionnel groupe mirabolique dont le radical unipotent $V_{n-1}(K)$ 
est abÈlien isomorphe ‡ $(K^{\times})^{n-1}$. 

%
%On termine ces rappels par un lemme certainement bien connu des experts dont nous prÈsentons une
%version limitÈe afin d'Èviter d'introduire trop de notations inutiles.

\begin{lemm} \label{lem-mirabolique}
Soit $\pi$ une reprÈsentation irrÈductible cuspidale de $GL_g(K)$, alors en tant que reprÈsentation
du parabolique $M_{(t+s)g}(K)$, on a des isomorphismes
$$\st_t(\pi\{ -\frac{s}{2} \} )_{|M_{tg}(K)} \times \speh_s (\pi \{ \frac{t}{2} \} ) \simeq
LT_{\pi} (t-1,s)_{|M_{(t+s)g}(K)},$$
et
$$\st_t(\pi\{ -\frac{s}{2} \} ) \times \speh_s (\pi \{ \frac{t}{2} \} )_{|M_{sg}(K)} \simeq 
LT_{\pi} (t,s-1)_{|M_{(t+s)g}(K)},$$
o˘ dans le premier isomorphisme, l'induite parabolique est relativement ‡ 
$$\left ( \begin{array}{cc} M_{tg} & U \\ 0 & GL_{sg} \end{array} \right )$$
alors que dans le deuxiËme, il s'agit de l'induite ‡ support compact relativement ‡
$$\left ( \begin{array}{ccc} 1 & 0 & V_{sg-1} \\ 0 & GL_{tg} & U \\ 0 & 0 & GL_{sg-1} \end{array} \right ).$$
\end{lemm}

\rem si dans le lemme prÈcÈdent, on remplace la reprÈsentation de Steinberg $\st_t(\pi)$ par 
$LT_{\pi}(\delta,t-\delta-1)$, le membre de droite dans le premier isomorphisme devient
$LT_{\pi}(\delta,t-\delta-1+s)$.

\begin{proof}
%Pour $n \geq 2$, on note
%$M_n(K)$ le sous-groupe de $P_{1,n}(K)$ dont le premier coefficient en haut ‡ gauche
%est Ègal ‡ $1$. Quitte ‡ tordre les actions par $g \mapsto \sigma (\lexp t g^{-1} ) \sigma^{-1}$, 
%o˘ $\sigma$ est la matrice de permutation associÈ au cycle $(12\cdots n)$,
%on reconnait le traditionnel groupe mirabolique dont le radical unipotent $V_{n-1}(K)$ 
%est abÈlien isomorphe ‡ $(K^{\times})^{n-1}$. 
%On fixe un caractËre $\psi$ non trivial de $F_v^\times$ et 
On rappelle cf. par exemple \cite{vigneras-livre} \S III.1.10, que
\begin{multline*}
0 \rightarrow 
\st_t(\pi\{ -\frac{s}{2} \} ) \times \speh_s (\pi \{ \frac{t}{2} \} )_{|M_{sg}(K)} \\
\longrightarrow
\Bigl ( \st_t(\pi\{ -\frac{s}{2} \} )  \times \speh_s (\pi \{ \frac{t}{2} \} ) \Bigr )_{|M_{(t+s)g}(K)} \\
\longrightarrow
 \st_t(\pi\{ -\frac{s}{2} \} )_{|M_{tg}(K)} \times \speh_s (\pi \{ \frac{t}{2} \} )
\rightarrow 0,
\end{multline*}
o˘ le deuxiËme terme est l'induite parabolique relativement ‡ 
$\left ( \begin{array}{cc} M_{tg} & U \\ 0 & GL_{sg} \end{array} \right )$
et le premier, l'induite ‡ support compact relativement ‡
$\left ( \begin{array}{ccc} 1 & 0 & V_{sg-1} \\ 0 & GL_{tg} & U \\ 0 & 0 & GL_{sg-1} \end{array} \right )$.
En outre pour tout $k \geq 0$, la dÈrivÈe d'ordre $k$ de 
$\st_t(\pi\{ -\frac{s}{2} \} )_{|M_{tg}(K)} \times \speh_s (\pi \{ \frac{t}{2} \} )$ est, cf. 
\cite{vigneras-livre} p153, donnÈe par celle de 
$\st_t(\pi\{ -\frac{s}{2} \} )_{|M_{tg}(K)}$ induite avec $\speh_s (\pi \{ \frac{t}{2} \} )$, i.e.
$$
\Bigl ( \st_t(\pi\{ -\frac{s}{2} \} )_{|M_{tg}(K)} \times \speh_s (\pi \{ \frac{t}{2} \} ) \Bigr )^{(k)} 
\simeq \Bigl ( \st_t(\pi\{ -\frac{s}{2} \})_{|M_{tg}(K)} \Bigr )^{(k)} \times 
\speh_s (\pi \{ \frac{t}{2} \} ).
$$
Rappelons que la dÈrivÈe d'ordre $k$ de $\st_t(\pi)$ est nulle sauf si $k$ est de la forme
$\delta g$ avec $0 \leq \delta \leq t$ auquel cas elle est isomorphe ‡ 
$\st_{t-\delta}(\pi \{ \frac{\delta}{2} \})$. Ainsi en raisonnant par rÈcurrence sur $t$, on en dÈduit que
$\st_t(\pi\{ -\frac{s}{2} \} )_{|M_{tg}(K)} \times \speh_s (\pi \{ \frac{t}{2} \} )$ et
$LT_{\pi_v}(t-1,s)_{|M_{(t+s)g}(K)} $ ont les mÍmes dÈrivÈes et qu'elles sont toutes d'ordre $\leq tg$.
ConsidÈrons alors le morphisme composÈ
\begin{multline*}
LT_{\pi_v}(t-1,s)_{|M_{(t+s)g}(K)} \hookrightarrow 
\Bigl ( \st_t(\pi\{ -\frac{s}{2} \} )  \times \speh_s (\pi \{ \frac{t}{2} \} ) \Bigr )_{|M_{(t+s)g}(K)} \\
\twoheadrightarrow 
\st_t(\pi\{ -\frac{s}{2} \} )_{|M_{tg}(K)} \times \speh_s (\pi \{ \frac{t}{2} \} )
\end{multline*}
et notons $K \hookrightarrow \st_t(\pi\{ -\frac{s}{2} \} ) \times \speh_s (\pi \{ \frac{t}{2} \} )_{|M_{sg}(K)}$
son noyau. D'aprËs \cite{zelevinski2} proposition 5.3 et corollaire 6.8,
$\st_t(\pi\{ -\frac{s}{2} \} ) \times \speh_s (\pi \{ \frac{t}{2} \} )_{|M_{sg}(K)}$ est homogËne, i.e.
tout sous-espace $M_{(t+s)g}(K)$-Èquivariant irrÈductible a une dÈrivÈe d'ordre $(t+1)g$, or on vient
de voir que les dÈrivÈes de $LT_{\pi}(t-1,s)_{|M_{(t+s)g}(K)}$ sont d'ordre $\leq tg$, de sorte que
$$LT_{\pi}(t-1,s) \hookrightarrow \st_t(\pi\{ -\frac{s}{2} \} )_{|M_{tg}(K)} \times 
\speh_s (\pi \{ \frac{t}{2} \} ).$$
Comme ces deux termes ont les mÍmes dÈrivÈes, cette injection est un isomorphisme.
%
% et sont donc isomorphes en tant que reprÈsentations
%de $M_{(t+s)g}(F_v)$. Par ailleurs comme 
%$$M_{(t+s)g}(F_v)/ \bigl ( M_{tg}(F_v) \times GL_{sg}(F_v) \bigr ) \simeq
%P_{1,(t+s)g}(F_v)/ \bigl ( P_{1,tg}(F_v) \times GL_{sg}(F_v) \bigr ),$$
%on en dÈduit que 
%$\st_t(\pi\{ -\frac{s}{2} \})_{|P_{1,tg}(F_v)} \times \speh_s (\pi \{ \frac{t}{2} \} )$ 
%est en tant que sous-espace de $\st_t(\pi \{ \frac{-s}{2} \} ) \times \speh_s (\pi \{ \frac{t}{2} \} )$,
%isomorphe ‡ $\st_t(\pi)$.
%
\end{proof}

\subsection{RÈduction modulo $l$ d'une Steinberg gÈnÈralisÈe d'aprËs 
\texorpdfstring{\cite{boyer-repmodl}}{Lg}}
\label{para-RI}

On rappelle que $l$ et $p$ dÈsignent des nombres premiers distincts et que $q$ est une puissance de $p$. On note
$e_l(q)$ l'ordre de l'image de $q$ dans $\Fm_l^\times$.
Afin de simplifier la lecture, dans la suite on utilisera la lettre $\pi$ pour dÈsigner
une $\overline \Qm_l$-reprÈsentation entiËre et les lettres $\varrho$ et $\rho$ pour
des $\overline \Fm_l$-reprÈsentations.

\begin{defi} \phantomsection \label{defi-mrho}
Une reprÈsentation $\varrho$ de $GL_n(K)$ est dite 
\begin{itemize}
\item \emph{cuspidale} si pour tout sous-groupe parabolique propre $P$ de $GL_n(K)$, 
$J_{P}(\varrho)$ est nul.

\item  Elle sera dite \emph{supercuspidale} si elle n'est pas un sous-quotient d'une induite
parabolique propre. Pour $\Lambda=\overline \Qm_l$ ou $\overline \Fm_l$, on notera 
$\scusp_\Lambda(g)$ l'ensemble des classes d'isomorphismes des reprÈsentations supercuspidales
de $GL_g(K)$.
\end{itemize}
\end{defi}

\rem sur $\overline \Qm_l$, les notions de cuspidales et de supercuspidales coÔncident, ce qui n'est
pas le cas sur $\overline \Fm_l$, cf. ci-aprËs.

\begin{prop} \label{prop-red-modl} (cf. \cite{vigneras-livre} III.5.10)
La rÈduction modulo $l$ d'une $\overline \Qm_l$-reprÈsentation irrÈductible cuspidale entiËre
de $GL_g(K)$ est irrÈductible cuspidale.
\end{prop}

\begin{prop} \cite{dat-jl} \S 2.2.3 \\
Soit $\pi$ une reprÈsentation irrÈductible cuspidale entiËre. Alors pour tout $s \geq 1$, la rÈduction 
modulo $l$ de $\speh_{s}(\pi)$ est irrÈductible.
\end{prop}

\begin{defi} Une $\overline \Fm_l$-reprÈsentation irrÈductible est dite $l$-Speh (resp. $l$-superSpeh)
si c'est la rÈduction modulo $l$ d'une $\overline \Qm_l$-reprÈsentation entiËre $\speh_s(\pi)$
pour $\pi$ irrÈductible cuspidale (resp. et dont la rÈduction modulo $l$ de $\pi$ est supercuspidale). 
\end{defi}

\begin{nota}
On notera  $\epsilon(\varrho)$ le cardinal de la droite de Zelevinski de
$\varrho$, i.e. de l'ensemble des classes d'Èquivalence $\{ \varrho\{ i\} ~/~ i \in \Zm \}$.
 On pose alors  cf. \cite{vigneras-induced} p.51
$$m(\varrho)=\left\{ \begin{array}{ll} \epsilon(\varrho), & \hbox{si } \epsilon(\varrho)>1; \\ l, & \hbox{sinon.} \end{array} \right.$$
\end{nota}

\rem $\epsilon(\varrho)$ est un diviseur de $e_l(q)$.

\begin{defi} \phantomsection \label{defi-nd}
…tant donnÈ un multi-ensemble $\underline s=\{ \rho_1^{n_1},\cdots,\rho_r^{n_r} \}$
de reprÈsentations cuspidales, on note d'aprËs
\cite{vigneras-induced} V.7, $\st(\underline s)$ l'unique reprÈsentation non dÈgÈnÈrÈe de l'induite
$$\rho(\underline s):=\overbrace{\rho_1 \times \cdots \times \rho_1}^{n_1} \times \cdots \times \overbrace{\rho_r \times \cdots \times \rho_r}^{n_r}.$$
\end{defi}

\rem d'aprËs \cite{vigneras-induced} V.7, toutes les reprÈsentations non dÈgÈnÈrÈes sont de cette forme.

\begin{nota} \phantomsection \label{nota-st}
Pour $\rho$ une reprÈsentation irrÈductible cuspidale et $s \geq 1$, on note
$\underline s(\rho)$ le multi-segment $\{ \rho, \rho \{ 1 \}, \cdots , \rho\{ s-1 \} \}$ et comme dans
\cite{vigneras-induced} V.4, $\st_s(\rho):=\st(\underline s(\rho))$.
\end{nota}

\begin{prop} \cite{vigneras-induced} V.4
Soit $\varrho$ une reprÈsentation irrÈductible cuspidale.
La reprÈsentation non dÈgÈnÈrÈe $\st_s(\varrho)$ est cuspidale si et seulement $s=1$ ou $m(\varrho)l^k$ pour $k \geq 0$.
\end{prop}

\rem d'aprËs \cite{vigneras-livre} III-3.15 et 5.14, toute reprÈsentation irrÈductible cuspidale est de la forme
$\st_s(\varrho)$ pour $\varrho$ irrÈductible supercuspidale et $s=1$ ou de la forme $m(\varrho)l^k$ avec 
$k \geq 0$. 

\begin{nota} \phantomsection \label{nota-rhoi}
Soit $\varrho$ une reprÈsentation irrÈductible cuspidale de $GL_g(K)$; on note 
$\varrho_{-1}=\varrho$ et pour tout $i \geq 0$, $\varrho_i=\st_{m(\varrho)l^i}(\varrho)$.
\end{nota}

\begin{defi}
On dira d'une $\overline \Qm_l$-reprÈsentation irrÈductible cuspidale entiËre qu'elle est
de type $\varrho$ si, ‡ torsion par un caractËre non ramifiÈ prËs, 
sa rÈduction modulo $l$ est de la forme $\varrho_i$ pour $i \geq -1$.
\end{defi}

\begin{nota}
Soit $s \geq 1$ un entier et $\varrho$ une reprÈsentation irrÈductible cuspidale de $GL_g(K)$. Soit
$\IC_\varrho(s)$ l'ensemble des suites $(m_{-1},m_0,\cdots)$ ‡ valeurs dans $\Nm$ telles que
$$s=m_{-1}+m(\varrho)\sum_{k=0}^{+\oo} m_k l^k.$$
On notera $\lg_\varrho(s)$ le cardinal de $\IC_{\varrho}(s)$. On munit enfin $\IC_{\varrho}(s)$
d'une relation d'ordre
$$(m_{-1},m_0,\cdots) > (m'_{-1},m'_0,\cdots) \Leftrightarrow \exists k \geq -1 \hbox{ tel que }
\forall i > k: m_i=m'_i \hbox{ et } m_k> m'_k.$$
\end{nota}

\begin{defi} \phantomsection \label{defi-spm}
Pour $\underline i=(i_{-1},i_0,\cdots) \in \IC_\varrho(s)$, on dÈfinit
$$\st_{\underline i}(\varrho):= \st_{\underline{i_{-1}^-}}(\varrho_{-1}) \times
\st_{\underline{i_{0}^-}}(\varrho_0) \times \cdots \times \st_{\underline{i_u^-}}(\varrho_u)$$
o˘ $i_k=0$ pour tout $k >u$ et o˘ les $\varrho_i$ sont dÈfinis en \ref{nota-rhoi}.
\end{defi}

\begin{theo} \phantomsection \label{theo-ss-quotient}
Soit $\pi$ une $\overline \Qm_l$-reprÈsentation irrÈductible cuspidale entiËre
de $GL_g(K)$ et $\varrho$ sa rÈduction modulo $l$.
Dans le groupe de Grothendieck des $\overline \Fm_l$-reprÈsentations de $GL_{sg}(K)$, on a l'ÈgalitÈ suivante:
$$r_l \Bigl ( \st_s(\pi) \Bigr )=\sum_{\underline i \in \IC_\varrho(s)} \st_{\underline i}(\varrho).$$
Par ailleurs pour tout $\underline i \in \IC_\varrho(s)$ et pour tout parabolique $P$, $J_{P} \Bigl (
\st_{\underline i} (\varrho) \Bigr )$ est Ègal ‡ la somme des constituants irrÈductibles de $\varrho$-niveau $\underline i$ de $r_l \Bigl (J_{P}(\st_s(\pi)) \Bigr )$.
\end{theo}

\rem pour $s<m(\varrho)$,
la rÈduction modulo $l$ de $\st_{s}(\pi)$ est irrÈductible.

\begin{defi}
On dira que  $l$ est \textit{banal} pour $GL_d(K)$ si $e_l(q)>d$.
\end{defi}

\rem dans le cas banal toute reprÈsentation cuspidale est supercuspidale, i.e. $m(\varrho) < s$ 
avec les notations prÈcÈdentes.

Pour $\pi$ une reprÈsentation irrÈductible cuspidale entiËre de $GL_g(K)$,
comme, d'aprËs \ref{prop-red-modl},
sa rÈduction modulo $l$, notÈe $\varrho$, est irrÈductible, on en dÈduit
qu'‡ isomorphismes prËs, $\pi$
possËde un unique rÈseau stable, cf. par exemple \cite{bellaiche-ribet} proposition 3.3.2 et la remarque qui suit.

\begin{defi} \phantomsection \label{defi-RI} (cf. \cite{boyer-repmodl})
…tant donnÈ un rÈseau de $\st_t(\pi)$, la surjection (resp. l'injection)
$$\st_t(\pi) \overrightarrow{\times} \pi  \twoheadrightarrow \st_{t+1}(\pi), 
%\quad resp.~ \st_{t+1}(\pi) \hookrightarrow \st_t(\pi\{ 1 \}) \times \pi
$$
induit un rÈseau de $\st_{t+1}(\pi)$ de sorte que par rÈcurrence on dispose d'un rÈseau  
$RI_{\overline \Zm_l,-}(\pi,t)$ que l'on qualifie de \textit{rÈseau d'induction}. On note alors
$$RI_{\overline \Fm_l,-}(\pi,t):= RI_{\overline \Zm_l,-}(\pi,t) \otimes_{\overline \Zm_l} \overline \Fm_l, 
% resp.~ RI_{\overline \Fm_l,+}(\pi,t):= RI_{\overline \Zm_l,+}(\pi,t) \otimes_{\overline \Zm_l} \overline \Fm_l.
$$
\end{defi}

\begin{prop} \phantomsection \label{prop-defi-Vk} 
(cf. \cite{boyer-repmodl} propositions 3.2.2 et 3.2.7)
Pour tout $0 \leq k \leq \lg_\varrho(s)$, il existe une sous-reprÈsentation $V_{\varrho,-}(s;k)$ 
de longueur $k$ de $RI_{\overline \Fm_l,-}(\pi,s)$
$$(0)=V_{\varrho,\pm}(s;0) \varsubsetneq V_{\varrho,\pm}(s;1) \varsubsetneq \cdots \varsubsetneq
V_{\varrho,\pm}(s;\lg_\varrho(s))= RI_{\overline \Fm_l,-}(\pi,s),$$
dÈfinie de sorte que l'image de $V_{\varrho,-}(s;k)$ dans le groupe de Grothendieck est telle
que tous ses constituants irrÈductibles sont
de $\varrho$-niveau strictement plus grand que n'importe quel constituant irrÈductible de
$W_{\varrho,-}(s;k):= V_{\varrho,-}(s;\lg_\varrho(s))/V_{\varrho,-}(s;k)$.
%(resp. $W_{\varrho,+}(s;k):= V_{\varrho,+}(s;\lg_\varrho(s))/V_{\varrho,+}(s;k)$)
\end{prop}

\begin{nota} \phantomsection \label{nota-V}
Une reprÈsentation irrÈductible $\varrho$ Ètant fixÈe ainsi qu'un entier $s$,
pour $k\geq 0$ tel que $m(\varrho)l^k \leq s$, on note:
\begin{itemize}
\item $\underline{\delta_k}=(0,\cdots,0,1,0,\cdots) \in \IC_\varrho(s)$ et

\item pour tout $t$ tel que $m(\varrho)l^kt \leq s$,
$V_{\varrho,-}(s,\geq t.\underline{\delta_k})$ le sous-espace $V_{\varrho,-}(s,\lg_\varrho(s))$ 
dÈfini ci-dessus tel que
tous les constituants irrÈductibles de $V_{\varrho,-}(s,\lg_\varrho(s))$ sont de $\varrho$-niveau plus
grand ou Ègal ‡ $t.\underline{\delta_k}$.
\end{itemize}
\end{nota}

\subsection{ReprÈsentations de \texorpdfstring{$D_{K,d}^\times$}{Lg} ‡ coefficients dans 
\texorpdfstring{$ \overline \Fm_l$}{Lg} et leurs relËvements}
\label{para-repD}

Soient $D_{K,d}$ l'algËbre ‡ division centrale sur $K$ d'invariant $1/d$,
$\DC_{K,d}$ son ordre maximal de radical $\PC_{K,d}$:
$1+\PC_{K,d} \subset \DC_{K,d}^\times \subset D_{K,d}^\times/\varpi^\Zm$
de quotients successifs $\Fm_{q^d}^\times$ et $\Zm/d\Zm$.

\begin{nota} \label{nota-JL} \phantomsection
Pour $\pi$ une reprÈsentation irrÈductible cuspidale de $GL_g(K)$ et $s \geq 1$, on note
$\pi[s]_D$ la reprÈsentation irrÈductible de $D_{K,sg}^\times$ image de $\st_s(\pi^\vee)$
par la correspondance de Jacquet-Langlands.
\end{nota}

\rem toute reprÈsentation irrÈductible de $D_{K,d}^\times$ s'Ècrit de maniËre unique
$\pi[s]_D$ pour $s|d$ et $\pi$ une reprÈsentation irrÈductible cuspidale de $GL_{d/s}(K)$.

\begin{nota} On notera $\RC_{ \overline \Fm_l}(h)$ l'ensemble des classes d'Èquivalence des $\overline \Fm_l$-
reprÈsentations irrÈductibles de $D_{v,h}^\times$.
\end{nota}

¿ torsion par un caractËre non ramifiÈ prËs, toute reprÈsentation 
irrÈductible $\tau$ de $D_{K,d}^\times$ se factorise par $D_{K,d}^\times/\varpi^\Zm$,
on choisit alors un facteur irrÈductible $\rho$ de $\tau_{|1+\PC_{K,d}}$ et on note
$N_\rho$ le normalisateur de sa classe d'isomorphisme dans 
$D_{K,d}^\times/\varpi^\Zm$. Comme $1+\PC_{K,d}$ est un pro-$p$-groupe, la dimension de $\rho$
est une puissance de $p$ de sorte que, un $p$-Sylow de $N_\rho/(1+\PC_{K,d})$
Ètant cyclique, $\rho$ admet un prolongement $\widetilde \rho$ ‡ $N_\rho$, cf. 
\cite{vigneras-apropos} lemme 1.19. Soit alors $\zeta$ une sous-reprÈsentation irrÈductible
de $N_\rho$ agissant sur $\hom_{1+\PC_{K,d}}(\widetilde \rho,\tau)$ laquelle d'aprËs 
la preuve de la proposition 2.3.2 de \cite{dat-jl} que nous suivons, est de la forme 
$\zeta \simeq \ind_J^{N_\rho}(\chi)$ pour un groupe $J$ contenant $N_\rho \cap \OC_{K,d}^\times$
et un caractËre $\chi$ de $J$ trivial sur $1+\PC_{K,d}$. 

\begin{prop} (cf. \cite{dat-jl} proposition 2.3.2) \label{prop-repD}\phantomsection \\
Avec les notations prÈcÈdentes, on a 
$$\tau = \ind_{J}^{D_{K,d}^\times/\varpi^\Zm} \bigl ( \widetilde \rho_{|J} \otimes \chi \bigr )$$
avec $J=N_\rho \cap N_\chi$ o˘ $N_\chi$ dÈsigne le normalisateur de $\chi$.
La rÈduction $r_l(\tau)$ de $\tau$ modulo $l$ est de la forme
$$\ind_{J}^{D_{K,d}^\times/\varpi^\Zm} \bigl ( r_l(\widetilde \rho_{|J}) \otimes r_l(\chi) \bigr )$$
o˘ $r_l(\widetilde \rho_{|J})$ reste irrÈductible de normalisateur $N_\rho$ de sorte que
la longueur de $r_l(\tau)$ est Ègale ‡
$$m(\tau):=[N_\rho \cap N_{r_l(\chi)}:J].$$
\end{prop}

\rem deux sous-quotients irrÈductibles $\tau$ et $\tau'$ quelconques de
la rÈduction modulo $l$ d'une reprÈsentation irrÈductible entiËre de $D_{K,d}^\times$, sont
inertiellement Èquivalents, i.e. il existe un caractËre $\zeta:\Zm \longrightarrow  \overline \Qm_l^\times$
tel que $\tau' \simeq \tau \otimes \zeta \circ \val \circ \rn$.

Par construction il existe des entiers $f',d',e'$ de produit Ègal ‡ $d$ tels que
$$J/(1+\PC_{K,d}) \simeq \Fm^\times_{q^{f'd'}} \rtimes m \Zm/e'd' \Zm,$$
o˘ $m$ est un diviseur de $d'$ tel que $f'm=[D_{K,d}^\times:\varpi^\Zm \DC_{d,K}^\times J]$.
L'abÈlianisÈ de $J/(1+\PC_{K,d})$ s'identifie, via la norme, ‡  
$\Fm^\times_{q^{f'm}} \times m \Zm/e'd' \Zm$ et le nombre de reprÈsentations irrÈductibles
strictement congrues ‡ $\tau$ est la plus grande puissance de $l$ divisant le cardinal
de cet abÈlianisÈ.

\begin{nota}  \label{nota-D0}
Pour $\bar \tau$ une $ \overline \Fm_l$-reprÈsentation irrÈductible de 
$D_{K,d}^\times$, avec les notations de la proposition prÈcÈdente et $\tau$ un relËvement,
avec les notations prÈcÈdentes, on notera
$$m(\bar \tau)=[N_\chi \cap N_{r_l(\chi)}:J] \quad \hbox{ et } \quad g(\bar \tau):=\frac{d}{e'}=f'd'.$$
Soit aussi $s(\bar \tau)$ la plus grande puissance de $l$ divisant 
$\frac{d}{m(\bar \tau)g(\bar \tau)}$ et on note
$$g_{-1}(\bar \tau)=g(\bar \tau) \quad \hbox{ et } \quad \forall 0 \leq i \leq s(\bar \tau),~
g_i(\bar \tau)=m(\bar \tau)l^i g(\bar \tau).$$
\end{nota}

\begin{defi}
Pour $\bar \tau$ une $\overline \Fm_l$-reprÈsentation irrÈductible de $D_{K,d}^\times$, soit
$\CC_{\bar \tau} \subset \rep_{\Zm_l^{nr}}^\oo(D_{K,d}^\times)$
la sous-catÈgorie pleine formÈe des objets dont tous les 
$\Zm_l^{nr}\DC_{K,d}^\times$-sous-quotients irrÈductibles sont isomorphes ‡
un sous-quotient de $\bar \tau_{|\DC_{K,d}^\times}$. On note $\bar \tau^0$ un sous-quotient irrÈductible
quelconque de $\bar \tau_{|\DC_{K,d}^\times}$.
\end{defi}

\begin{prop} (cf. \cite{dat-torsion} \S B.2) \label{prop-scindage} \phantomsection \\
Soit $P_{\bar \tau^0}$ une enveloppe projective de $\bar \tau^0$ dans
$\rep^\oo_{\Zm_l^{nr}}(\DC_{K,d}^\times)$. Alors la sous-catÈgorie 
$\CC_{\bar \tau}$ est facteur direct dans $\rep_{\Zm_l^{nr}}^\oo(D_{K,d}^\times)$
pro-engendrÈe par l'induite $P_{\bar \tau}:=\ind_{\DC_{K,d}^\times}^{D_{K,d}^\times} (P_{\bar \tau^0})$.
\end{prop}

Ainsi toute $\Zm_l^{nr}$-reprÈsentation $V_{\Zm_l^{nr}}$ de $D_{K,d}^\times$
se dÈcompose en une somme directe
$V_{\Zm_l^{nr}} \simeq \bigoplus_{\bar \tau} V_{\Zm_{l,\bar \tau}^{nr}}$
o˘ $\bar \tau$ dÈcrit les classes d'Èquivalence inertielles des $\overline \Fm_l$-reprÈsentations
irrÈductibles de $D_{K,d}^\times$ et o˘ $V_{\Zm_{l,\bar \tau}^{nr}}$
est un objet de $\CC_{\bar \tau}$, i.e. tous ses sous-quotients irrÈductibles
sont isomorphes ‡ un sous-quotient de $\bar \tau_{|\DC_{K,d}^\times}$.

\begin{prop} (cf. \cite{dat-jl} proposition 2.3.2) 
Soient $\bar \tau$ une $\overline \Fm_l$-reprÈsentation irrÈductible de $D_{K,d}^\times$
et $\tau' \in \CC_{\bar \tau}$ une $\overline \Qm_{l}$-reprÈsentation irrÈductible entiËre.
Il existe alors $-1 \leq i \leq s(\bar \tau)$ et une reprÈsentation irrÈductible cuspidale $\pi_i$
de $GL_{g_i(\bar \tau)}(K)$ telle que $\tau' \simeq \pi_i[\frac{d}{g_i(\bar \tau)}]_D$.
\end{prop}

\rem avec les notations de la proposition prÈcÈdente, la rÈduction modulo $l$ de $\pi_i$
est supercuspidale si et seulement si $i=-1$ et sinon son support supercuspidal
est un segment de Zelevinsky-VignÈras de longueur $m(\bar \tau)l^i$ associÈ ‡ une supercuspidale
$\varrho$ de $GL_{g_{-1}(\bar \tau)}(K)$. En outre $\bar \tau$ est uniquement dÈterminÈe par
$\varrho$, on dira qu'elle est de type $\varrho$.

\begin{nota} \label{nota-type} \phantomsection
…tant donnÈs une reprÈsentation irrÈductible supercuspidale $\varrho$ de $GL_g(K)$
et un multiple $h=tg$ de $g$,on notera $\varrho[t]_D \in \RC_{\overline \Fm_l}(h)$ la classe
de la reprÈsentation de $D^\times_{K,h}$ de type $\varrho$. On note
$$\RC_{\overline \Fm_l}(\varrho):= \bigl \{ \varrho[t]_D: ~1 \leq t \leq \frac{d}{g} \bigr \}.$$
%Avec les notations de la proposition prÈcÈdente, $\tau'$ (resp. un sous-quotient irrÈductible
%de la rÈduction modulo $l$ de $\tau'$) sera dit de $\bar \tau$-type $i$.
%Plus gÈnÈralement, pour tout $t \geq \frac{d}{g_i(\bar \tau)}$, la reprÈsentation $\pi_i[t]_D$
%(resp. un sous-quotient irrÈductible de la rÈduction modulo $l$ de $\pi_i[t]_D$)
%de $D_{K,tg_i(\bar \tau)}^\times$ sera dite de $\bar \tau$-type $i$. On dira de mÍme
%qu'une reprÈsentation de $\CC_{\bar \tau}$ est de $\bar \tau$-type $\geq i$ si tous ses constituants 
%irrÈductibles ont un $\bar \tau$-type qui est $\geq i$.
\end{nota}

\noindent \textit{Exemple:} soit $\pi_{-1}$ une $ \overline \Qm_l$-reprÈsentation irrÈductible cuspidale entiËre 
de $GL_g(K)$ dont la rÈduction modulo $l$ est supercuspidale. Soient $t \geq 1$ et
$\bar \tau$ un constituant irrÈductible de la rÈduction modulo $l$ de $\pi_{-1}[t]_D$: on a
alors $g_{-1}(\bar \tau)=g$.
Soit alors $\pi_i$ une reprÈsentation irrÈductible cuspidale de $GL_{g_i(\bar \tau)}(K)$ dont la rÈduction
modulo $l$ a pour support supercuspidal un segment de Zelevinsky $[r_l(\pi_{-1} \{ \frac{1-s}{2} \}),
r_l(\pi_{-1} \{ \frac{s-1}{2} \} )]$ avec $g_i(\bar \tau)=s g$. Pour tout $t_i$ tel que $t_i g_i(\bar \tau) \geq
tg$, la reprÈsentation $\pi_i[t_i]_D$ (resp. un sous-quotient irrÈductible de la rÈduction modulo $l$
de $\pi_i[t_i]_D$) est de $\bar \tau$-type $i$.

\begin{nota} \label{nota-taui}
Avec les notations prÈcÈdentes, on notera 
$\scusp_i(\bar \tau)$ ou $\scusp_i(\varrho)$ l'ensemble des classes d'Èquivalences 
des reprÈsentations $\pi_i$.
\end{nota}

\begin{prop} \label{prop-red-tau}
Soit $\pi_{v,i} \in \scusp_i(\bar \tau)$. Pour $t \geq 1$, on a
alors l'ÈgalitÈ suivante dans le groupe de Grothendieck
\addtocounter{smfthm}{1}
\begin{equation} \label{eq-redmodl}
r_l(\pi_{v,u}[t]_D)=l^u \sum_{i=0}^{m(\varrho)-1} r_l(\pi_{v,-1}[tm(\varrho)l^u] \nu^i.
\end{equation}
\end{prop}

\rem dans le cas o˘ $\epsilon(\varrho)=1$, la formule
(\ref{eq-redmodl}) s'Ècrit $r_l( \pi_{v,u}[t]_D)=l^{u+1} r_l(\pi_{v,-1}[t]_D)$.

\section{Rappels des rÈsultats faisceautiques sur $\overline \Qm_l$ d'aprËs \cite{boyer-invent2}}
\label{app-B}

Dans ce paragraphe nous rappelons les rÈsultats de \cite{boyer-invent2} que nous utilisons dans ce texte:
les notations sont celles des paragraphes prÈcÈdents. PrÈcisons de nouveau que tous les rÈsultats de 
loc. cit. sont sur $\overline \Qm_l$ ce qui permet de l'enlever des notations afin d'allÈger les Ècritures.

Rappelons que pour $\sigma_v$ (resp. $\pi_v$) une
reprÈsentation de $W_v$ (resp. de $GL_d(F_v)$), l'action d'un ÈlÈment $g$ sur 
$$\sigma_v(n) \qquad (\hbox{resp. } \pi_v\{ n \})$$ 
est donnÈe par $\sigma_v(g) |\art^{-1}_{F_v}(g)|^n$ (resp. $\pi_v(g) |\det g|^n$)
o˘ $|-|$ est la valeur absolue sur $F_v$ et $\art_{F_v}^{-1}:W_v \longto F_v^\times$ le morphisme
de la thÈorie du corps de classe qui envoie les frobenius gÈomÈtriques $\fr^{-1}$ de $W_v$ de $F_v$ sur
les uniformisantes, i.e. $v(\art_{F_v}^{-1} (\fr))=-1$.

\subsection{sur les systËmes locaux d'Harris-Taylor}
\label{app-B1}

Soient $\pi_v$ une reprÈsentation irrÈductible cuspidale de $GL_g(F_v)$ et $1 \leq t \leq \frac{d}{g}$.
La reprÈsentation $\pi_v[t]_D$ de $D_{v,tg}^\times$, cf. la notation \ref{nota-JL}, fournit 
un systËme local sur $X^{=tg}_{\IC,\bar s,\overline{1_h}}$
$$\LC(\pi_v[t]_D)_{\overline{1_h}}=\bigoplus_{i=1}^{e_{\pi_v}} 
\LC_{\overline \Qm_l}(\rho_{v,i})_{\overline{1_h}}$$
o˘ $(\pi_v[t]_D)_{|\DC_{v,h}^\times}=\bigoplus_{i=1}^{e_{\pi_v}} \rho_{v,i}$ avec $\rho_{v,i}$ irrÈductible
et muni d'une action de $P_{tg,d-tg}(F_v)$ via son quotient $GL_{d-tg} \times \Zm$.

\rem ce systËme local est notÈ $\FC(t,\pi_v)_1$ dans loc. cit. 1.4.7.

\begin{nota} Pour tout strate pure $X^{=tg}_{\IC,\bar s,a}$, on note 
$\LC(\pi_v[t]_D)_{a}$ l'image de $\LC(\pi_v[t]_D)_{\overline{1_h}}$ par un ÈlÈment quelconque
de $a \in GL_d(F_v)/P_{tg,d-tg}(F_v)$. 
%Plus gÈnÈralement pour $\AC$ 
%comme dans \ref{nota-1hh}, 
%un sous-ensemble fermÈ de strates pures de $X^{=tg}_{\IC,\bar s}$, on notera
%$$\LC(\pi_v[t]_D)_\AC:=\bigoplus_{a \in \AC} \LC(\pi_v[t]_D)_a.$$
\end{nota}

\noindent \textit{Exemple}: pour $h \leq tg$ on a ainsi $\LC(\pi_v[t]_D)_{\overline{1_h}}$
est la somme directe des $\LC(\pi_v[t]_D)_a$ sur les strates pures $X^{=tg}_{\IC,\bar s,a}$
contenues dans $X^{\geq h}_{\IC,\bar s,\overline{1_h}}$.

\begin{defi} Les systËmes locaux d'Harris-Taylor sont les
$$\widetilde{HT}_a(\pi_v,\Pi_t):=\LC(\pi_v[t]_D)_a \otimes \Pi_t \otimes \Xi^{\frac{tg-d}{2}}$$
et leurs versions induites
$$\widetilde{HT}(\pi_v,\Pi_t):=\Bigl ( \LC(\pi_v[t]_D)_{\overline{1_h}} 
\otimes \Pi_t \otimes \Xi^{\frac{tg-d}{2}} \Bigr) \times_{P_{tg,d-tg}(F_v)} GL_d(F_v),$$
o˘ l'action du radical unipotent de $P_{tg,d-tg}(F_v)$ est triviale, et celle de
$$(g^{\oo,v},\left ( \begin{array}{cc} g_v^c & *†\\ 0 & g_v^{et} \end{array} \right ),\sigma_v) 
\in G(\Am^{\oo,v}) \times P_{tg,d-tg}(F_v) \times W_v$$ 
est donnÈe
\begin{itemize}
\item par celle de  $g_v^c$ sur $\Pi_t$ et  $\deg(\sigma_v) \in \Zm$ sur $ \Xi^{\frac{tg-d}{2}}$ ainsi que 

\item celle de $(g^{\oo,v},g_v^{et},\val(\det g_v^c)-\deg \sigma_v)
\in G(\Am^{\oo,v}) \times GL_{d-tg}(F_v) \times \Zm$ sur $\LC_{\overline \Qm_l}
(\pi_v[t]_D)_{\overline{1_h}} \otimes \Xi^{\frac{tg-d}{2}}$,
o˘ $\Xi:\frac{1}{2} \Zm \longrightarrow 
\overline \Zm_l^\times$ est dÈfini par $\Xi(\frac{1}{2})=q^{1/2}$.
\end{itemize}
\end{defi}

On dit de l'action de $GL_{tg}(F_v)$ qu'elle est \emph{infinitÈsimale}, 
cf. aussi la remarque prÈcÈdant \ref{nota-1h}. 

\begin{nota} \label{nota-ind}
Pour 
\begin{itemize}
\item $z_\IC$ un point gÈomÈtrique
de $X^{=tg}_{\IC,\bar s}$ et $i_z:\overline{ \{ z \} } \hookrightarrow X_{\IC,\bar s}$,

\item ainsi que $HT(\pi_v,\Pi_t)$ un systËme local d'Harris-Taylor sur
 $X^{=tg}_{\IC,\bar s}$,
\end{itemize}
on note
$$\ind z^* HT(\pi_v,\Pi_t):= \ind_{GL_{tg}(\OC_v)}^{GL_{tg}(F_v)}
i_z^* HT(\pi_v,\Pi_t).$$
\end{nota}

%\rem pour $\AC$ une rÈunion de strates pures, on utilise aussi la notation
%$\widetilde{HT}_\AC(\pi_v,\Pi_t)=\bigoplus_{a \in \AC} \widetilde{HT}_a(\pi_v,\Pi_t)$.

\begin{notas} On pose
$$HT(\pi_v,\Pi_t):=\widetilde{HT}(\pi_v,\Pi_t)[d-tg],$$
et le faisceau pervers d'Harris-Taylor associÈ est
$$P(t,\pi_v):=j^{=tg}_{!*} HT(\pi_v,\st_t(\pi_v)) \otimes \Lm(\pi_v),$$
ou sa version non induite
$$P(t,\pi_v)_a=i^h_{a,*} i^{h,*}_a P(t,\pi_v)=j^{=tg}_{a,!*} HT_a(\pi_v,\st_t(\pi_v)) \otimes \Lm(\pi_v),$$
o˘ $\LF^\vee$ dÈsigne la correspondance locale de Langlands.
\end{notas}

\rem on rappelle que $\pi'_v$ est inertiellement Èquivalente ‡ $\pi_v$ si et seulement
s'il existe un caractËre $\zeta: \Zm \longrightarrow  \overline \Qm_l^\times$ tel que 
$\pi'_v \simeq \pi_v \otimes (\zeta \circ \val \circ \det)$.
Les faisceaux pervers $P(t,\pi_v)$ ne dÈpendent que de la classe d'Èquivalence inertielle de $\pi_v$
et sont de la forme 
$$P(t,\pi_v)=e_{\pi_v} \PC(t,\pi_v)$$ 
o˘ $\PC(t,\pi_v)$ est un faisceau pervers irrÈductible.

\begin{nota} \label{nota-divers}
On note $\Fil^{-tg}_*(\pi_v,\Pi_t)$ le noyau du morphisme d'adjonction
$$j^{=tg}_! HT(\pi_v,\Pi_t) \twoheadrightarrow j^{=tg}_{!*} HT(\pi_v,\Pi_t),$$
qui est ‡ support dans $X^{\geq tg+1}_{\IC,\bar s}$.
%Avec la notation \ref{nota-P}, $\Fil^{-tg}_*(\pi_v,\Pi_t)$ est notÈ $P_{\widetilde{HT}(\pi_v,\Pi_t)}$.
\end{nota}

\begin{prop} (cf. \cite{boyer-torsion} 3.3.5)
Le faisceau pervers $\Fil^{-tg}_*(\pi_v,\Pi_t)$ est en fait ‡ support dans $X^{\geq t(g+1)}_{\IC,\bar s}$
et le morphisme d'adjonction 
$$j^{=(t+1)g}_! j^{=(t+1)g,*} \Fil^{-tg}_*(\pi_v,\Pi_t) \longrightarrow \Fil^{-tg}_*(\pi_v,\Pi_t)$$
est surjectif.
\end{prop}

En itÈrant cette proposition, on obtient, pour $s=\lfloor \frac{d}{g} \rfloor$, la rÈsolution
\addtocounter{smfthm}{1}
\begin{multline} \label{eq-resolution0}
0 \rightarrow j_!^{=s g} HT(\pi_v,\Pi_t \overrightarrow{\times} \speh_{s-t}(\pi_v)) \otimes \Xi^{\frac{s-t}{2}}
\longrightarrow \cdots \longrightarrow j^{=(t+2)g}_! HT(\pi_v,\Pi_t†\overrightarrow{\times} \speh_2(\pi_v))
\otimes \Xi^1 \\
\longrightarrow j_!^{= (t+1)g} HT(\pi_v,\Pi_t \overrightarrow{\times} \pi_v) \otimes \Xi^{\frac{1}{2}}
\longrightarrow \Fil^{-tg}_*(\pi_v,\Pi_t) \rightarrow 0.
\end{multline}
o˘ pour $\pi_1$ et $\pi_2$ des reprÈsentations de respectivement $GL_{t_1g}(F_v)$ et $GL_{t_2g}(F_v)$,
$\pi_1 \overrightarrow{\times} \pi_2$ dÈsigne l'induite normalisÈe 
$\pi_1 \{ -\frac{t_2}{2} \} \times \pi_2 \{ \frac{t_1}{2} \}$.

\rem de cette rÈsolution on en dÈduit:
\begin{itemize}
\item le calcul des faisceaux de cohomologie des faisceaux pervers
d'Harris-Taylor du thÈorËme 2.2.5 de \cite{boyer-invent2}. PrÈcisÈment, pour $g>1$, le $i$-Ëme faisceau
de cohomologie de $\lexp p j^{=tg}_{!*} HT(\pi_v,\Pi_t)$ est nul sauf si $i$ est de la forme
$-d+tg+\delta(g-1)$ auquel cas il est isomorphe ‡ $HT(\pi_v,\Pi_t \{ \frac{-\delta}{2} \} \times \speh_\delta
(\pi_v \{ \frac{t}{2} \} ) )[tg-d] \otimes \Xi^{\frac{\delta}{2}}$.

\item la description des constituants irrÈductibles de $j^{=tg}_! HT(\pi_v,\Pi_t)$ de la proposition 4.3.1, 
complÈtÈe par le corollaire 5.4.1, de \cite{boyer-invent2}, cf. aussi le corollaire 3.3.8 de \cite{boyer-torsion},
qui dans le groupe de Grothendieck correspondant, s'Ècrit
\begin{equation} \label{eq-extzero}
\bigl [ j^{=tg}_! HT(\pi_v,\Pi_t) \bigr ]=\sum_{r=0}^{\lfloor \frac{d}{g} \rfloor -t}
\bigl [ j^{=(t+r)g}_{!*} HT(\pi_v,\Pi_t \overrightarrow{\times} \st_r(\pi_v)) \otimes \Xi^{\frac{r}{2}} \bigr ].
\end{equation}

\end{itemize}

\subsection{sur le faisceau pervers des cycles Èvanescents}
\label{para-rap-psi}

La premiËre Ètape de \cite{boyer-invent2} consiste, cf. loc. cit. dÈfinition 2.2.2, ‡ dÈcouper $\Psi_{\IC}$
selon les classes d'Èquivalence inertielles $\cusp_v(g)$ des reprÈsentations irrÈductibles cuspidales 
de $GL_g(F_v)$ pour $g$ variant de $1$ ‡ $d$:
\addtocounter{smfthm}{1}
\begin{equation} \label{eq-decoupageQl}
\Psi_{\IC} \simeq \bigoplus_{g=1}^d \bigoplus_{\pi_v \in \cusp_v(g)} \Psi_{\IC,\pi_v}.
\end{equation}
Les rÈsultats du \S 7 de \cite{boyer-invent2} sur $\Psi_{\IC,\pi_v}$ sont rÈsumÈs par la proposition 3.4.3
de \cite{boyer-torsion} que nous reproduisons ci-dessous.

\begin{prop} \phantomsection \label{prop-psi-fil} (cf. \cite{boyer-torsion} 3.4.3)
Soit 
$$0=\Fil^0_!(\Psi_{\IC,\pi_v}) \subset \Fil^1_!(\Psi_{\IC,\pi_v}) \subset \cdots \subset
\Fil^d_!(\Psi_{\IC,\pi_v})=\Psi_{\IC,\pi_v}$$
la filtration de stratification de $\Psi_{\IC,\pi_v}$ de la proposition \ref{prop-ss-filtration}. 
Pour tout $r$ non divisible par $g$, le graduÈ $\gr^r_!(\Psi_{\IC,\pi_v})$ est nul et pour
$r=tg$ avec $1 \leq t \leq s$, il est ‡ support dans $X^{\geq tg}_{\IC,\bar s}$ avec\footnote{‡ un facteur 
$e_{\pi_v}$ prËs, $j^{\geq tg,*} \gr^{tg}_!(\Psi_{\IC,\pi_v})$ est donc isomorphe ‡
$\widetilde{HT}(\pi_v,\st_t(\pi_v)) \otimes \Lm(\pi_v)(\frac{1-t}{2})$}
$$j^{\geq tg,*} \gr^{tg}_!(\Psi_{\IC,\pi_v}) \simeq j^{\geq tg,*} \PC(t,\pi_v) (\frac{1-t}{2}).$$ 
La surjection
$j^{= tg}_! j^{\geq tg,*} \gr^{tg}_!(\Psi_{\IC,\pi_v}) \twoheadrightarrow \gr^{tg}_!(\Psi_{\IC,\pi_v}),$
o˘ dans le groupe de Grothendieck
$$\Bigl [  \gr^{tg}_!(\Psi_{\IC,\pi_v}) \Bigr ]=\sum_{i=t}^s \Bigl [ \PC(i,\pi_v)(\frac{1+i-2t}{2}) \Bigr ].$$
\end{prop}

\rem dans \cite{boyer-torsion}, on explique comment ce rÈsultat permet de calculer les fibres
des faisceaux de cohomologie de $\Psi_{\IC,\pi_v}$, i.e. d'en dÈduire le corollaire 2.2.10 de 
\cite{boyer-invent2}. En particulier la fibre en un point fermÈ de $X^{=tg}_{\IC,\bar s}$ de
$\hi^{h-d} \Psi_{\IC,\pi_v}$ est munie d'une action de 
$(D_{v,tg}^\times)^0:=\ker \bigl ( \val \circ \rn: D_{v,tg}^\times \longrightarrow \Zm \bigr )$
et de $\varpi_v^\Zm$ que l'on voit plongÈ dans $F_v^\times \subset D_{v,tg}^\times$. 
D'aprËs le thÈorËme 2.2.6 de \cite{boyer-invent2}, ou plus simplement d'aprËs la proposition prÈcÈdente,
on a
\begin{eqnarray} 
\ind_{(D_{v,tg}^\times)^0 \varpi_v^\Zm}^{D_{v,tg}^\times} 
\bigl ( \hi^{tg-d} \Psi_{\IC,\pi_v} \bigr )_{|X^{=tg}_{\IC,\bar s}} &
\simeq & \hi^{tg-d} P(t,\pi_v)(\frac{1-t}{2}) \nonumber  \\
&\simeq &  \widetilde{HT}(\pi_v,\st_t(\pi_v)) \otimes \Lm(\pi_v)  \otimes \Xi^{\frac{1-t}{2}}. \label{eq-psi-dh}
\end{eqnarray}

\begin{nota} \label{nota-grmoins}
On note
$$0 \rightarrow \gr^{tg,-}_!(\Psi_{\IC,\pi_v}) \longrightarrow \gr^{tg}_!(\Psi_{\IC,\pi_v})
\longrightarrow \gr^{tg,+}_!(\Psi_{\IC,\pi_v}) \rightarrow 0,$$
avec $\gr^{tg,+}_!(\Psi_{\IC,\pi_v}):= \PC(t,\pi_v)(\frac{1-t}{2})$.
\end{nota}

\subsection{ComplÈments}
\label{para-comple}

Les rÈsultats de ce paragraphe sont dÈveloppÈs plus en dÈtail dans \cite{boyer-FT}. 
L'article n'Ètant pas encore publiÈ, nous donnons ci-aprËs un rapide aperÁu du \S 6 de loc. cit.

Etant donnÈe une inclusion de strates pures 
$X^{\geq h+1}_{\IC,\bar s,c} \hookrightarrow X^{\geq h}_{\IC,\bar s,a}$,
l'ingrÈdient essentiel et nouveau par rapport ‡ \cite{boyer-invent2} est l'utilisation des morphismes
$$j^{\geq h}_{a-c }: X^{\geq h}_{\IC,\bar s,a} - X^{\geq h+1}_{\IC,\bar s,c} \hookrightarrow
X^{\geq h}_{\IC,\bar s,a}$$
que l'on notera simplement $j^{\geq h}_{\neq c}$ lorsque $a=\overline{1_h}$. On s'intÈressera plus
particuliËrement, pour $P$ un $\overline \Qm_l$-faisceau pervers
supportÈ dans $X^{\geq h}_{\IC,\bar s,a}$, 
aux morphismes d'adjonction $j^{=h}_{a-c,!} j^{=h,*}_{a-c} P \longrightarrow P$. Dans les applications
$P$ sera en fait ‡ support dans $X^{\geq h+g}_{\IC,\bar s,a}$, pour $g \geq 1$, de sorte que
$$j^{=h}_{a- c,!} j^{=h,*}_{a- c} P = j^{=h+g}_{a-c,!} j^{=h+g,*}_{a-c} P,$$
o˘ $j^{=h+g}_{a-c}=i^{h+g} \circ j^{\geq h+g}_{a-c}$ avec
$$j^{\geq h+g}_{a-c}: X^{\geq h+g}_{\IC,\bar s,a}-X^{\geq h+g}_{\IC,\bar s,c} \hookrightarrow
X^{\geq h+g}_{\IC,\bar s,a}.$$

\rem on autorise $h=0$ auquel cas $X^{\geq 0}_{\IC,\bar s,a}$ est la variÈtÈ de Shimura 
$X_{\IC} \rightarrow \spec \OC_v$.

L'intÈrÍt principal des inclusions $j^{\geq h}_{a-c}$ est son caractËre affine rappelÈ dans le lemme suivant.

\begin{lemm} \label{lem-jaffine} \phantomsection
L'inclusion ouverte $j^{\geq h}_{a-c}$ est affine.
\end{lemm}

\begin{proof} Par symÈtrie il suffit de traiter le cas de $a=\overline{1_h}$ et
$c=\overline{1_{h+1}}$. Or
pour tout $\spec A \longrightarrow X^{\geq h}_{\IC,\bar s,\overline{1_h}}$, le fermÈ 
$X^{\geq h+1}_{\IC,\bar s,\overline{1_{h+1}}} \times_{X^{\geq h}_{\IC,\bar s,\overline{1_h}}} 
\spec A$ est, avec les notations du \S \ref{para-shimura}, donnÈ par l'annulation de $\iota(e_{h+1})$.
\end{proof}

%L'utilisation des inclusions $j^{\geq h}_{a-c}$ devrait permettre de simplifier grandement les dÈmonstrations
%des rÈsultats principaux de \cite{boyer-invent2}. En attendant de proposer une telle simplification,
En utilisant la description des faisceaux de cohomologie des faisceaux pervers d'Harris-Taylor et
du complexe des cycles Èvanescents, donnÈe dans \cite{boyer-invent2} et rappelÈe plus haut,
nous allons donner quelques rÈsultats sur l'effet du foncteur exact $j^{=h}_{\neq c,!} j^{=h,*}_{\neq c}$
sur les faisceaux pervers d'Harris-Taylor et le complexe des cycles Èvanescents. 
On commence par $h=0$ et un faisceau pervers d'Harris-Taylor $HT(\pi_v,\st_t(\pi_v))$
‡ support dans $X^{\geq tg}_{\IC,\bar s}$ de sorte que pour toute strate pure $X^{\geq 1}_{\IC,\bar s,c}$,
$$j^{=1}_{\neq c,!} j^{=1,*}_{\neq c} HT(\pi_v,\st_t(\pi_v)) \simeq \cdots \simeq
j^{=tg}_{\neq c,!} j^{=tg,*}_{\neq c} HT(\pi_v,\st_t(\pi_v)).$$
Pour simplifier les notations on prendra $c=\overline{1_1}$.

\begin{lemm} (cf. le corollaire 6.6 de \cite{boyer-FT}) \label{lem-j-c}
Etant donnÈ un systËme local d'Harris-Taylor $HT(\pi_v,\st_t(\pi_v))$ relativement ‡ une reprÈsentation irrÈductible cuspidale $\pi_v$ de $Gl_g(F_v)$, on a la suite exacte courte
\begin{multline*}
0 \rightarrow \Ind_{P_{1,(t+1)g-1,d-(t+1)g}(F_v)}^{P_{1,d-1}(F_v)}
\lexp p j^{=(t+1)g}_{\overline{1_{(t+1)g}},!*} HT_{\overline{1_{(t+1)g}}}
\bigr (\pi_v,\st_{t+1}(\pi_v)_{P_{1,(t+1)g-1}(F_v)} \bigl ) \otimes \Xi^{\frac{1}{2}} \\
\longrightarrow j^{=1}_{\neq \overline{1_1},!}j^{=1,*}_{\neq \overline{1_1}} 
\bigl ( \lexp p j^{=tg}_{!*} HT(\pi_v,\st_t(\pi_v))
\bigr ) \longrightarrow \lexp p  j^{=1}_{\neq \overline{1_1},!*}j^{=1,*}_{\neq \overline{1_1}} 
HT(\pi_v,\st_t(\pi_v)) \rightarrow 0.
\end{multline*}
\end{lemm}

\begin{proof}
Il s'agit simplement d'identifier $\lexp p \hi^{-1} i^{1,*}_1 (  \lexp p j^{=tg}_{!*} 
HT(\pi_v,\st_t(\pi_v)) )$ avec le premier terme de la suite exacte courte de l'ÈnoncÈ.
La technique est commune avec les
preuves des lemmes qui suivront: on construit tout d'abord une surjection de
$\lexp p \hi^{-1} i^{1,*}_{\overline{1_1}} (  \lexp p j^{=tg}_{!*} HT(\pi_v,\st_t(\pi_v)) )$ vers le faisceau 
pervers de l'ÈnoncÈ
et on vÈrifie dans un deuxiËme temps que les faisceaux de cohomologie de
ces deux faisceaux pervers possËdent les mÍmes germes en tout point gÈomÈtrique.
Pour le premier point notons tout d'abord que
$$\lexp p \hi^{-1} i^{1,*}_{\overline{1_1}} (  \lexp p j^{=tg}_{!*} HT(\pi_v,\st_t(\pi_v)) ) \simeq \lexp p \hi^{-1}
 i^{1,*}_{\overline{1_1}} (  \lexp p j^{=tg}_{\neq \overline{1_1},!*} HT(\pi_v,\st_t(\pi_v)) ).$$
Dans \cite{boyer-invent2} 4.5.1, on dÈcrit $j^{=tg}_{\overline{1_{tg}},!} 
HT_{\overline{1_{tg}}} (\pi_v,\st_t(\pi_v))$ dans le groupe de Grothendieck des faisceaux pervers
Hecke Èquivariant. En appliquant alors le foncteur $\lexp p \hi^{-1} i^{tg+1,*}_{\overline{1_1}}$
‡ la filtration par les poids de $j^{=tg}_{\overline{1_{tg}},!} HT_{\overline{1_{tg}}} (\pi_v,\st_t(\pi_v))$,
on obtient
\begin{multline*} 
\lexp p \hi^{-1} i^{tg+1,*}_{\overline{1_1}} \bigl( \lexp p j^{=tg}_{\overline{1_{tg}},!*} HT_{\overline{1_{tg}}}
(\pi_v,\st_t(\pi_v)) \bigr )  \twoheadrightarrow \\ 
\Ind_{P_{1,tg-1,g,d-(t+1)g}(F_v)}^{P_{1,tg-1,d-tg}(F_v)}
\lexp p j^{=(t+1)g}_{\overline1_{(t+1)g},!*} HT_{\overline{1_{(t+1)g}}} \Bigl ( \pi_v,
\bigl (\st_t(\pi_v \{ \frac{-1}{2} \} )  \bigr )_{|P_{1,tg-1}(F_v)} 
\otimes  \pi_v \{ \frac{t}{2} \} \Bigr ) \otimes \Xi^{\frac{1}{2}}
\end{multline*}
et donc 
\begin{multline*} 
\lexp p \hi^{-1} i^{tg+1,*}_{\overline{1_1}} \bigl( \lexp p j^{=tg}_{!*} HT (\pi_v,\st_t(\pi_v)) \bigr )  
\twoheadrightarrow \\  \Ind_{P_{1,(t+1)g-1,d-(t+1)g}(F_v)}^{P_{1,d-1}(F_v)}
\lexp p j^{=(t+1)g}_{\overline1_{(t+1)g},!*} HT_{\overline{1_{(t+1)g}}} \Bigl ( \pi_v,
\bigl (\st_t(\pi_v \{ \frac{-1}{2} \} )  \bigr )_{|P_{1,tg-1}(F_v)} 
\times  \pi_v \{ \frac{t}{2} \} \Bigr ) \otimes \Xi^{\frac{1}{2}}
\end{multline*}
o˘ d'aprËs le lemme \ref{lem-mirabolique}, 
$$\bigl ( \st_t(\pi_v \{ \frac{-1}{2} \} )  \bigr )_{|P_{1,tg-1}(F_v)} \times \pi_v \{ \frac{t}{2} \} 
\simeq \st_{t+1}(\pi_v)_{|P_{1,(t+1)g-1}(F_v)}.$$ 
Ainsi on obtient bien une surjection de $ \lexp p \hi^{-1} i^{1,*}_{\overline{1_1}} (  
\lexp p j^{=tg}_{\neq \overline{1_1},!*} HT(\pi_v,\st_t(\pi_v)) )$ vers le faisceau pervers induit de l'ÈnoncÈ.

Passons ‡ prÈsent au deuxiËme point. Pour tout point gÈomÈtrique $z$ de 
$X^{=h}_{\IC,\bar s,\overline{1_1}}$, le
germe en $z$ du $i$-Ëme faisceau de cohomologie de 
$\lexp p \hi^{-1} i^{1,*}_1 (  \lexp p j^{=tg}_{!*} HT(\pi_v,\st_t(\pi_v)) )$ est isomorphe ‡ celui du $(i-1)$-Ëme
de $\lexp p j^{=tg}_{!*} HT(\pi_v,\st_t(\pi_v))$. Ainsi d'aprËs \cite{boyer-invent2}, ce germe est
\begin{itemize}
\item nul si $(h,i)$ n'est pas de la forme $\bigl ( (t+\delta)g, (t+\delta)g-d-\delta \bigr )$ avec
$(t+\delta)g \leq d$,

\item et sinon il est isomorphe au germe en $z$ de $HT(\pi_v,\Pi)$ o˘ $\Pi$ est l'induite normalisÈe 
$$\Pi:=\bigl ( \st_t(\pi_v \{ -\frac{\delta}{2} \} ) \bigr )_{|P_{1,tg-1}(F_v)} \times \speh_\delta 
(\pi_v \{ \frac{t}{2} \} ) \simeq LT(t,\delta-1,\pi_v)_{|P_{1,(t+\delta)g-1}(F_v)},$$ 
le dernier isomorphisme Ètant donnÈ par le lemme \ref{lem-mirabolique}.
\end{itemize}
On calcule de mÍme le germe en $z$ du $i$-Ëme faisceau de cohomologie du premier faisceau pervers
de la suite exacte courte de l'ÈnoncÈ. Les conditions d'annulation sont les mÍmes et sinon on trouve
le germe en $z$ de $HT(\pi_v,\Pi')$ avec
$$\Pi':=\st_{t+1}(\pi_v)_{|P_{1,(t+1)g-1}(F_v)}  \{ \frac{1-\delta}{2} \} \times \speh_{\delta-1} 
(\pi_v \{ \frac{t+1}{2} \})$$ 
qui d'aprËs le lemme \ref{lem-mirabolique} est isomorphe ‡ 
$\Pi \simeq LT_{\pi_v}(t,\delta-1)_{|P_{1,(t+\delta)g-1}(F_v)}$.

\end{proof}

ConsidÈrons ‡ prÈsent $h \geq 1$ et $a=\overline{1_h}$. Pour $\LC[d-h]=HT_{\overline{1_h}}(\pi_v,\Pi_t)$
un systËme local d'Harris-Taylor relativement ‡ une reprÈsentation irrÈductible cuspidale $\pi_v$
de $GL_g(F_v)$ et $\Pi_t$ une reprÈsentation de $GL_{tg}(F_v)$ avec $h=tg$, 
ne jouant aucun rÙle, on notera
$$P_\LC:=i^{h+1}_* \lexp p \hi^{-1} i^{h+1,*}_{\overline{1_h}} \bigl ( \lexp p j^{=h}_{\overline{1_h},!*}
HT_{\overline{1_h}}(\pi_v,\Pi_t) \bigr ),$$ 
de sorte que
$$0 \rightarrow P_\LC \longrightarrow j^{=h}_{\overline{1_h},!} HT_{\overline{1_h}}(\pi_v,\Pi_t)
\longrightarrow \lexp p j^{=h}_{\overline{1_h},!*} HT_{\overline{1_h}}(\pi_v,\Pi_t) \rightarrow 0.$$

\begin{lemm} \label{lem-Ql1}
On a une suite exacte courte de faisceaux pervers
$P_{\Delta(c)}(F_v)$-Èquivariants
$$0 \rightarrow j^{=h+1}_{\neq c,!} j^{=h+1,*}_{\neq c } P_\LC \longrightarrow P_\LC 
\longrightarrow \lexp p j^{=h+g}_{c,!*} j^{=h+g,*}_c P_\LC\ \rightarrow 0.$$
\end{lemm}

\rem rappelons que dans \cite{boyer-invent2}, on calcule dans un premier temps l'image de
$j^{=h}_{\overline{1_h},!}HT_{\overline{1_h}}(\pi_v,\Pi_t)$ dans un groupe de Grothendieck
de faisceaux pervers Èquivariants au moyen de la formule des traces calculant la somme alternÈe
de la cohomologie ‡ support compact de $HT_{\overline{1_h}}(\pi_v,\Pi_t)$.
Moralement donc la suite exacte de l'ÈnoncÈ est ‡ portÈe de la formule des traces puisque
\begin{itemize}
\item d'une part l'injectivitÈ du morphisme d'adjonction 
$j^{=h+1}_{\neq c,!} j^{h+1,*}_{\neq c } P_\LC \longrightarrow P_\LC$
est formelle et dÈcoule du caractËre affine de $j^{h\leq +1}_{\overline{1_h},\neq c}$,

\item et que l'image dans le groupe de Grothendieck du quotient se dÈduit par rÈcurrence de 
celle de $P_\LC$ et de $j^{=h+1}_{\neq c,!} j^{h+1,*}_{\neq c } P_\LC$.
\end{itemize}

\begin{proof}
En ce qui concerne le fait que le morphisme d'adjonction 
$$j^{=h+1}_{\neq c,!} j^{=h+1,*}_{\neq c } P_\LC \longrightarrow P_\LC$$ 
soit injectif, on renvoie le
lecteur au lemme \ref{lem-HT1} qui traite ce point plus gÈnÈralement sur $\overline \Zm_l$.
Le quotient de ce morphisme d'adjonction est 
$\lexp p \hi^0 i_c^{h+1,*} P_\LC$ et il s'agit donc de montrer qu'il est isomorphe ‡
$\lexp p j^{=h+g}_{c,!*} j^{=h+g,*}_c P_\LC$. On suit la stratÈgie dÈtaillÈe dans la preuve du lemme 
prÈcÈdent: le morphisme d'adjonction 
$j^{=h+g}_{c,!} j^{=h+g,*}_c (\lexp p \hi^0 i_c^{h+1,*} P_\LC) \longrightarrow 
(\lexp p \hi^0 i_c^{h+1,*} P_\LC)$ est surjectif et fournit donc une surjection
$$(\lexp p \hi^0 i_c^{h+1,*} P_\LC) \twoheadrightarrow \lexp p j^{=h+g}_{c,!*} j^{=h+g,*}_c P_\LC.$$
On est, comme prÈcÈdemment, ramenÈ ‡ montrer que les faisceaux de cohomologie de 
ces deux faisceaux pervers possËdent les mÍmes germes en tout point gÈomÈtrique de 
$X^{\geq h+g}_{\IC,\bar s,c}$. Par symÈtrie du problËme et afin de simplifier les notations
on suppose que $c=\overline{1_{h+1}}$ et que $z$ est un point gÈomÈtrique de
$X^{=h+\delta}_{\IC,\bar s,\overline{1_{h+\delta}}}$. Notons que le germe en $z$ du $i$-Ëme faisceau
de cohomologie de $\lexp p \hi^0 i_c^{h+1,*} P_\LC$ est isomorphe ‡ celui
de $P_\LC$ lequel, d'aprËs la suite exacte courte
$$0 \rightarrow P_\LC \longrightarrow j^{=h}_{!} \LC[d-h] \longrightarrow 
\lexp p j^{=h}_{!*} \LC[d-h] \rightarrow 0,$$
est isomorphe ‡ celui du $(i-1)$-Ëme faisceau de cohomologie de $\lexp p j^{=h}_{!} \LC[d-h]$.
D'aprËs \cite{boyer-invent2}, le germe en $z$
de $\hi^i \bigl ( \lexp p j^{=h}_{!} \LC[d-h] \bigr )$ est 
non nul si et seulement si $\delta$ est de la forme $rg$ avec $h+rg \leq d$ et $i=h-d+r(g-1)$ auquel cas
il est isomorphe, en tant que $P_{\Delta(1_{\overline{h+1}} \subset \overline{1_{h+rg}})}(F_v)$ au germe
en $z$ de 
$$HT_{\overline{1_{h+rg}}} \Bigl ( \pi_v, \Pi_t \{ \frac{-r}{2} \} \otimes \bigl ( \speh_r(\pi_v \{ \frac{t}{2} \} 
\bigr )_{|P_{1,rg-1}(F_v)} \bigr ) \Bigr ) [rg+h-d].$$

En ce qui concerne le germe en $z$ de $\lexp p j^{=h+g}_{c,!*} j^{=h+g,*}_c P_\LC$, 
notons tout d'abord que 
$$j^{=h+g,*}_{\overline{1_{h+1}}} P_\LC \simeq 
\Ind_{P_{\Delta(\overline{1_{h+1}} \subset \overline{1_{h+g}})}(F_v)}^{P_{\Delta(\overline{1_{h+1}}}(F_v)}
\Bigl ( j^{=h+g,*}_{\overline{1_{h+g}}} P_\LC \Bigr ),$$
et que donc, d'aprËs \cite{boyer-invent2}, le germe en question est isomorphe ‡ celui de
$$HT_{\overline{1_{h+rg}}} \Bigl ( \pi_v, \Pi_t \{ \frac{-r}{2} \} \otimes \bigl (
( \pi_v \{ \frac{t-r+1}{2} \} )_{|P_{1,g-1}(F_v)} \times \speh_{r-1}(\pi_v \{ \frac{t+1}{2} \} ) \bigr ) \Bigr ) 
[rg+h-d].$$
Le rÈsultat dÈcoule alors du fait, cf. \ref{lem-mirabolique}, que $\bigl ( \speh_r (\pi_v) \bigr)_{|P_{1,rg-1}(F_v)}$
est irrÈductible isomorphe ‡ 
$( \pi_v \{ \frac{-r+1}{2} \} )_{|P_{1,g-1}(F_v)} \times \speh_{r-1}(\pi_v \{ \frac{1}{2} \} ) \bigr )$.

\end{proof} 

Terminons par l'effet de $\bar j_{\neq a,!} \bar j^{*}_{\neq a}$ sur le complexe des cycles 
Èvanescents. Rappelons qu'en notant
$$\bar j: X_{\IC,\bar \eta} \hookrightarrow X_{\IC} \hookleftarrow X_{\IC,\bar s}: \bar i,$$
o˘ $\bar j_!=\lexp p {\bar j_!}$ et  $\bar j_*=\lexp p {\bar j_*}$ pour la $t$-structure dont la construction
est rappelÈe au dÈbut du \S \ref{para-ext-pp-preuve}, le complexe des
cycles Èvanescents est $\lexp p \hi^{-1} \bar i^* \bar j_* \bar \Qm_l$. Ainsi en utilisant que
pour toute strate pure $X^{\geq 1}_{\IC,\bar s,a}$, l'inclusion
$\bar j_{\neq a}:\overline X_{\IC}-X^{\geq 1}_{\IC,\bar s,a} \hookrightarrow
\overline X_{\IC}:=X_\IC \times_{\spec \OC_v} \spec \overline \OC_v$ 
Ètant affine, le lemme \ref{lem-HT1} nous donne que
$i^{1,*}_a \Psi_\IC$ est pervers et libre, ce qui donne la suite exacte courte
de $\FC(X_{\IC,\bar s},\overline \Zm_l)$
$$0 \rightarrow \bar j_{\neq a,!} \bar j^{*}_{\neq a} \Psi_\IC \longrightarrow \Psi_\IC
\longrightarrow i^1_{a,*} \lexp p \hi^0 i^{1,*}_a \Psi_\IC \rightarrow 0.$$
Rappelons, cf. \cite{boyer-invent2}, que $\Psi_{\IC,\overline \Qm_l}$ se dÈcompose en une somme
directe
$$\Psi_{\IC,\overline \Qm_l} \simeq \bigoplus_\pi \Psi_{\IC,\overline \Qm_l,\pi}$$
o˘ $\pi$ dÈcrit les classes d'Èquivalence inertielle des reprÈsentations irrÈductibles cuspidales
de $GL_g(F_v)$ pour $1 \leq g \leq d$.

\begin{prop} \label{prop-QL-Psi}
Avec les notations prÈcÈdentes, et pour $\pi$ irrÈductible cuspidale de $GL_g(F_v)$, le faisceau pervers
$ i^1_{a,*}\lexp p \hi^0 i^{1,*}_a \Psi_{\IC,\overline \Qm_l,\pi}$ admet une filtration
$$\Fil^1_a(\Psi_{\IC,\overline \Qm_l,\pi}) \subset \Fil^2_a(\Psi_{\IC,\overline \Qm_l,\pi}) \subset
\cdots \subset \Fil^s_a(\Psi_{\IC,\overline \Qm_l,\pi}),$$
o˘ $s=\lfloor \frac{d}{g} \rfloor$ avec pour graduÈs
$$\gr^k_a(\Psi_{\IC,\overline \Qm_l,\pi}) \simeq \lexp p j^{=kg}_{a,!*} HT_a(\pi_v,\st_k(\pi_v))
\otimes \Xi^{\frac{1-k}{2}}.$$
\end{prop}

\rem comme dans la remarque suivant le lemme \ref{lem-Ql1}, ce rÈsultat est ‡ la portÈe de
la formule des traces et permet de simplifier les arguments les plus complexes de
\cite{boyer-invent2}.

\begin{proof}
Par symÈtrie il suffit de considÈrer le cas $a=\overline{1_1}$.
On rappelle que d'aprËs \cite{boyer-invent2}, l'image $I_\pi$ du morphisme d'adjonction 
$$\Psi_{\IC,\overline \Qm_l,\pi} \longrightarrow j^{=g}_* j^{=g,*} \Psi_{\IC,\overline \Qm_l,\pi}$$ 
admet une filtration $\Fil^1(I_\pi) \subset \cdots \subset \Fil^s(I_\pi)$ dont les graduÈs sont
$\gr^k(I_\pi) \simeq  \lexp p j^{=kg}_{!*} HT(\pi_v,\st_k(\pi_v)) \otimes \Xi^{\frac{1-k}{2}}$.

\begin{lemm} \label{lem-conoyau1}
Pour tout $1 \leq r \leq s$, le faisceau pervers $i^{1}_{\overline{1_1},*} \lexp p \hi^0 i^{1,*}_{\overline{1_1}} 
\Fil^r(I_\pi)$ admet une filtration similaire ‡ celle de l'ÈnoncÈ de la proposition prÈcÈdente, i.e. dont les
graduÈs sont les $\lexp p j^{=kg}_{\overline{1_1},!*} HT_{\overline{1_1}}(\pi_v,\st_k(\pi_v))
\otimes \Xi^{-\frac{k}{2}}$ pour $1†\leq k \leq r$.
\end{lemm}

\begin{proof}
On raisonne par rÈcurrence sur $r$ de $1$ ‡ $s$. Comme il est clair que
$$i^{1}_{\overline{1_1},*} \lexp p \hi^0 i^{1,*}_{\overline{1_1}} \bigl (  \lexp p j^{=rg}_{!*} 
HT(\pi_v,\st_r(\pi_v)) \bigr ) 
\simeq \lexp p j^{=rg}_{\overline{1_1},!*} HT_{\overline{1_1}}(\pi_v,\st_r(\pi_v)),$$
il suffit de vÈrifier que
\addtocounter{smfthm}{1}
\begin{equation} \label{eq-fleche1}
i^{1}_{\overline{1_1},*} \lexp p \hi^{-1} i^{1,*}_{\overline{1_1}}  \bigl (  \lexp p j^{=rg}_{!*} 
HT(\pi_v,\st_r(\pi_v))\otimes \Xi^{\frac{1-r}{2}} \bigr ) 
\longrightarrow \lexp p \hi^{0} i^{1,*}_{\overline{1_1}}  \bigl ( \Fil^{r-1}(I_\pi) \bigr )
\end{equation}
est nulle. Or d'aprËs l'hypothËse de rÈcurrence tous les constituants irrÈductibles de
$\lexp p \hi^{0} i^{1,*}_{\overline{1_1}}  \bigl ( \Fil^{r-1}(I_\pi) \bigr )$ sont des extensions intermÈdiaires
de systËmes locaux sur des strates $X^{=kg}_{\IC,\bar s}$ pour $1 \leq k \leq r-1$ alors que
$\lexp p \hi^{-1} i^{1,*}_{\overline{1_1}}  \bigl (  \lexp p j^{=rg}_{!*} HT(\pi_v,\st_r(\pi_v)) \otimes
\Xi^{\frac{1-r}{2}} \bigr )$
est ‡ support dans $X^{\geq (r+1)g}_{\IC,\bar s}$ de sorte que (\ref{eq-fleche1}) est nÈcessairement nulle.

\end{proof}

On a ainsi une surjection de $i^{1}_{\overline{1_1},*}
\lexp p \hi^0 i^{1,*}_{\overline{1_1}} \Psi_{\IC,\overline \Qm_l,\pi}$
vers le $\Fil^s_{\overline{1_1}}(\Psi_{\IC,\overline \Qm_l,\pi})$ de l'ÈnoncÈ
et comme dans la preuve du lemme \ref{lem-Ql1} il suffit de montrer que ces deux faisceaux pervers
ont les mÍmes germes en tout point gÈomÈtrique de $X^{\geq 1}_{\IC,\bar s,\overline{1_1}}$.
ConsidÈrons donc $z$ un point gÈomÈtrique de 
$X^{=h}_{\IC,\bar s,\overline{1_h}}$. D'aprËs \cite{boyer-invent2}, la fibre en $z$ du $i$-Ëme faisceau
de cohomologie $\hi^i j^{=kg}_{\overline{1_{kg}},*} HT_{\overline{1_{kg}}}(\pi_v,\st_k(\pi_v))
\otimes \Xi^{\frac{1-k}{2}}$ est nulle si $(h,i)$ n'est pas de la forme $(tg-d,tg-d+k-t)$ avec 
$k \leq t \leq \lfloor \frac{d}{g} \rfloor$ et sinon est isomorphe ‡ celle de 
$$HT_{\overline{1_{tg}}} \bigl ( \pi_v,\st_k(\pi_v\{ \frac{k-t}{2} \}) \otimes 
\speh_{t-k} ( \pi_v \{ \frac{k}{2} \} ) \bigr ) \otimes \Xi^{\frac{1+t-2k}{2}}.$$
On en dÈduit alors que la fibre en $z$ de $j^{=kg}_{\overline{1_1},*} HT_{\overline{1_1}}
(\pi_v,\st_k(\pi_v)) \otimes \Xi^{\frac{1-k}{2}}$  est isomorphe ‡ celle de 
$$HT_{\overline{1_{tg}}} \Bigl ( \pi_v, \bigl ( \st_k(\pi_v\{ \frac{k-t}{2} \} \bigr )_{|P_{1,kg-1}(F_v)} \times 
\speh_{t-k} ( \pi_v \{ \frac{k}{2} \} ) \bigr ) \Bigr ) \otimes \Xi^{\frac{1+t-2k}{2}},$$
o˘ on induit de $P_{1,kg-1}(F_v) \otimes GL_{(t-k)g}(F_v)$ ‡ $P_{1,tg-1}(F_v)$. Par ailleurs, en regardant
les poids, on voit que la suite spectrale calculant la fibre des faisceaux de cohomologie de
$\Fil^s_{\overline{1_1}}(\Psi_{\IC,\overline \Qm_l,\pi})$ en fonction de ceux des 
$\gr^k_{\overline{1_1}}(\Psi_{\IC,\overline \Qm_l,\pi})$ dÈgÈnËre en $E_1$.
D'aprËs \ref{lem-mirabolique}, 
$$\bigl ( \st_k(\pi_v\{ \frac{k-t}{2} \}) \bigr )_{|P_{1,kg-1}(F_v)} \times \speh_{t-k} ( \pi_v \{ \frac{k}{2} \} )
\simeq \bigl ( LT_\pi(k,t-1-k)_{\pi_v} \bigr )_{|P_{1,tg-1}(F_v)}$$
de sorte que d'aprËs le rÈsultat principal de \cite{boyer-invent2}, la fibre en $z$ de
$\hi^i \bigl ( \Fil^s_{\overline{1_1}}(\Psi_{\IC,\overline \Qm_l,\pi}) \bigr )$ est isomorphe 
‡ celle de $\hi^i \bigl ( \lexp p \hi^0 i_{\overline{1_1}}^{1,*} \Psi_{\IC,\overline \Qm_l,\pi} \bigr )$,
d'o˘ le rÈsultat.

\end{proof}

\rem les mÍmes
arguments que prÈcÈdemment appliquÈs ‡ $\Fil^{rg}_!(\Psi_{\IC,\overline \Qm_l,\pi})$ donnent
$$i^1_{a,*} \lexp p \hi^0 i^{1,*}_a \bigl ( \Fil^{rg}_!(\Psi_{\IC,\overline \Qm_l,\pi}) \bigr ) \simeq 
\Fil^r_a(\Psi_{\IC,\overline \Qm_l,\pi}).$$
On en dÈduit alors que l'image du morphisme d'adjonction
$$j^{=g}_{\overline{1_1},!} j^{=g,*}_{\overline{1_1}} \Bigl ( i^{1}_{\overline{1_1},*}
\lexp p \hi^0 i_{\overline{1_1}}^{1,*}  \Fil^{rg}_! \bigl (\Psi_{\IC,\overline \Qm_l,\pi} \bigr ) \Bigr ) 
\longrightarrow   i^{1}_{\overline{1_1},*}
\lexp p \hi^0 i_{\overline{1_1}}^{1,*}  \Fil^{rg}_! \bigl ( \Psi_{\IC,\overline \Qm_l,\pi} \bigr )$$
est $j^{=g}_{\overline{1_1},!*} HT(\pi_v,\pi_v)(\frac{1}{2})$. Plus gÈnÈralement les applications successives
des morphismes d'adjonction
$$j^{=kg}_{\overline{1_1},!} j^{=kg,*}_{\overline{1_1}} \Bigl ( i^{1}_{\overline{1_1},*}
 \lexp p \hi^0 i_{\overline{1_1}}^{k,*} 
\Fil^{rg}_! \bigl ( \Psi_{\IC,\overline \Qm_l,\pi} \bigr ) \Bigr ) \longrightarrow  i^{1}_{\overline{1_1},*}
\lexp p \hi^0 i_{\overline{1_1}}^{r,*} \Bigl ( \Fil^{rg}_! \bigl (\Psi_{\IC,\overline \Qm_l,\pi} \bigr ) \Bigr )$$
donnent pour images les $\lexp p j^{=kg}_{\overline{1_1},!*} HT_{\overline{1_1}}
(\pi_v,\st_k(\pi_v)) \otimes \Xi^{\frac{1-k}{2}}$ pour $1 \leq k \leq r$.

\rem on notera enfin que de la description des faisceaux de cohomologie des faisceaux pervers
d'Harris-Taylor, la suite spectrale calculant les faisceaux de cohomologie de 
$\lexp p \hi^0 i^{1,*}_{\overline{1_1}} \Psi_{\IC,\overline \Qm_l}$ ‡ partir de ceux des graduÈs
de la proposition \ref{prop-QL-Psi}, dÈgÈnËre en $E_1$.

%
%En utilisant l'Èquivariance, on construit une surjection
%$\Psi_{\IC,\pi_v} \twoheadrightarrow \Psi_{\IC,\pi_v}^-$ o˘ la filtration de stratification
%$\Fil^\bullet_!(\Psi_{\IC,\pi_v}^-)$ a pour graduÈs successifs 
%$\gr^h_!(\Psi_{\IC,\pi_v}^-) \simeq $
%

%
%\begin{prop}
%La cofiltration de stratification de \cite{boyer-torsion}
%$$\Psi_{\IC,\pi_v} \twoheadrightarrow j^{=g}_*j^{=g,*} \Psi_{\IC,\pi_v}$$
%fournit une filtration
%$$0 \rightarrow \Fil^1_*(\Psi_{\IC,\pi_v}) \longrightarrow \Psi_{\IC,\pi_v} \longrightarrow
%\CoFil^1_*(\Psi_{\IC,\pi_v}) \rightarrow 0,$$
%o˘ la filtration de stratification via les morphismes d'adjonction $j_!j^* \rightarrow \Id$ munit
%$\Fil^1_*(\Psi_{\IC,\pi_v})$ (resp. $\CoFil^1_*(\Psi_{\IC,\pi_v})$)
%d'une filtration dont les graduÈs successifs sont les
%$\gr^{tg,-}_!(\Psi_{\IC,\pi_v})$ (resp. les
%$\PC(s-t+1,\pi_v)(\frac{t-s}{2})$), pour $1 \leq t \leq g$.
%\end{prop}
%
%
%
%\begin{figure}[ht]
%\input{LT-pic2-fin.tex}
%%\caption{Illustration}
%%\end{figure}
%
%%\end{document}
%
%%\begin{figure}[ht]
%\begin{center}
%\input{LT-pic1.tex}
%\end{center}
%\caption{\label{fig-1} Illustration graphique de $HR_\Psi(h+1)$}
%\end{figure}
%%\includepdf[scale=.7]{LT-pic1.pdf}
%
%\clearpage
%

\bibliographystyle{plain}
\bibliography{bib-ok}

\def\cftil#1{\ifmmode\setbox7\hbox{$\accent"5E#1$}\else
  \setbox7\hbox{\accent"5E#1}\penalty 10000\relax\fi\raise 1\ht7
  \hbox{\lower1.15ex\hbox to 1\wd7{\hss\accent"7E\hss}}\penalty 10000
  \hskip-1\wd7\penalty 10000\box7} \def\cprime{$'$}
\begin{thebibliography}{10}

\bibitem{ast}
A.~A. Beilinson, J.~Bernstein, and P.~Deligne.
\newblock Faisceaux pervers.
\newblock In {\em Analysis and topology on singular spaces, I (Luminy, 1981)},
  volume 100 of {\em Ast\'erisque}, pages 5--171. Soc. Math. France, Paris,
  1982.

\bibitem{bellaiche-ribet}
J.~Bella{\"{\i}}che.
\newblock \`{A} propos d'un lemme de {R}ibet.
\newblock {\em Rend. Sem. Mat. Univ. Padova}, 109:45--62, 2003.

\bibitem{berk0}
V.G. Berkovich.
\newblock \'etale cohomology for non-archimedean analytic spaces.
\newblock {\em Inst. Hautes \'etudes Sci. Publ. Math.}, 78:5--161, 1993.

\bibitem{berk2}
V.G. Berkovich.
\newblock Vanishing cycles for formal schemes. {II}.
\newblock {\em Invent. Math.}, 125(2):367--390., 1996.

\bibitem{zelevinski1}
I.N. Bernstein and A.V. Zelevinsky.
\newblock {Induced representations of reductive p-adic groups. I.}
\newblock {\em Ann. Scient. de l'ENS 4e s\'erie, tome 10 n∞4}, pages 441--472,
  1977.

\bibitem{B-R}
C\'edric {Bonnaf\'e} and Rapha\"el {Rouquier}.
\newblock {Coxeter orbits and modular representations.}
\newblock {\em {Nagoya Math. J.}}, 183:1--34, 2006.

\bibitem{boyer-invent2}
P.~Boyer.
\newblock Monodromie du faisceau pervers des cycles \'evanescents de quelques
  vari\'et\'es de {S}himura simples.
\newblock {\em Invent. Math.}, 177(2):239--280, 2009.

\bibitem{boyer-repmodl}
P.~Boyer.
\newblock {R}\'eseaux d'induction des repr\'esentations elliptiques de
  {L}ubin-{T}ate.
\newblock {\em Journal of Algebra}, 336, issue 1:28--52, 2011.

\bibitem{boyer-torsion}
P.~Boyer.
\newblock Filtrations de stratification de quelques vari\'et\'es de {S}himura
  simples.
\newblock {\em Bulletin de la SMF}, 142, fascicule 4:777--814, 2014.

\bibitem{boyer-FT}
P.~Boyer.
\newblock Groupe mirabolique, stratification de {N}ewton raffin\'ee et
  cohomologie des espaces de {L}ubin-{T}ate.
\newblock {\em Bull. SMF}, pages 1--18, 2019.

\bibitem{dat-torsion}
J.-F. Dat.
\newblock Th\'eorie de {L}ubin-{T}ate non-ab\'elienne $l$-enti\`ere.
\newblock {\em Duke Math. J. 161 (6)}, pages 951--1010, 2012.

\bibitem{dat-jl}
J.-F. Dat.
\newblock Un cas simple de correspondance de {J}acquet-{L}anglands modulo $l$.
\newblock {\em Proc. London Math. Soc. 104}, pages 690--727, 2012.

\bibitem{h-t}
M.~Harris, R.~Taylor.
\newblock {\em The geometry and cohomology of some simple {S}himura varieties},
  volume 151 of {\em Annals of Mathematics Studies}.
\newblock Princeton University Press, Princeton, NJ, 2001.

\bibitem{ill}
L.~Illusie.
\newblock Autour du th\'eor\`eme de monodromie locale.
\newblock In {\em P\'eriodes $p$-adiques}, number 223 in Ast\'erisque, 1994.

\bibitem{juteau}
D.~Juteau.
\newblock Decomposition numbers for perverse sheaves.
\newblock {\em Annales de l'Institut Fourier, 59 (3)}, pages 1177--1229, 2009.

\bibitem{vigneras-livre}
M.-F. Vign{\'e}ras.
\newblock {\em Repr\'esentations {$l$}-modulaires d'un groupe r\'eductif
  {$p$}-adique avec {$l\ne p$}}, volume 137 of {\em Progress in Mathematics}.
\newblock Birkh\"auser Boston Inc., Boston, MA, 1996.

\bibitem{vigneras-apropos}
M.-F. Vign{\'e}ras.
\newblock \`{A}\ propos d'une conjecture de {L}anglands modulaire.
\newblock In {\em Finite reductive groups (Luminy, 1994)}, volume 141 of {\em
  Progr. Math.}, pages 415--452. Birkh\"auser Boston, Boston, MA, 1997.

\bibitem{vigneras-induced}
M.-F. Vign{\'e}ras.
\newblock Induced {$R$}-representations of {$p$}-adic reductive groups.
\newblock {\em Selecta Math. (N.S.)}, 4(4):549--623, 1998.

\bibitem{zelevinski2}
A.~V. Zelevinsky.
\newblock Induced representations of reductive {${p}$}-adic groups. {II}. {O}n
  irreducible representations of {${\rm GL}(n)$}.
\newblock {\em Ann. Sci. \'Ecole Norm. Sup. (4)}, 13(2):165--210, 1980.

\end{thebibliography}

\end{document}